\documentclass[11pt,reqno,twosdie]{amsart}

\usepackage{amsmath,amssymb,mathrsfs}

\usepackage[asymmetric,top=3.2cm,bottom=3.9cm,left=2.8cm,right=2.8cm]{geometry}
\usepackage{amsmath,amsfonts,amsthm,mathrsfs,amssymb,cite}
\usepackage[usenames]{color}

\usepackage{mathtools} 
\usepackage{graphics}
\usepackage{framed}

\usepackage{enumerate}

\usepackage{subfigure}

\usepackage{hyperref}
\usepackage{xcolor}
\hypersetup{
  colorlinks,
  linkcolor={blue!80!black},
  urlcolor={blue!80!black},
  citecolor={blue!80!black}
}

\usepackage{graphicx}
\usepackage{latexsym} 
\usepackage{enumerate}

\theoremstyle{plain}
\newtheorem{theorem}{Theorem}[section]
\newtheorem{lemma}[theorem]{Lemma}
\newtheorem{corollary}[theorem]{Corollary}
\newtheorem{proposition}[theorem]{Proposition}

\theoremstyle{remark}
\newtheorem{definition}[theorem]{Definition}
\newtheorem{remark}[theorem]{Remark}

\renewcommand{\eqref}[1]{\textnormal{(\ref{#1})}}

\numberwithin{equation}{section}

\usepackage{color}

\newcommand{\C}{\mathcal{C}}
\newcommand{\R}{\mathcal{R}}
\newcommand{\Q}{\mathcal{Q}}
\newcommand{\M}{\mathcal{M}}
\renewcommand{\S}{\mathcal{S}}

\def\Oh{{\mathcal  O}}

\usepackage{booktabs}
\newcommand{\bl}[1]{\color{black}#1}

\title[Electromagnetic]{The Effective Permittivity and Permeability Generated by a Cluster of Moderately Contrasting  Nanoparticles}

\author{Xinlin Cao}
\address{Radon institute (RICAM), Austrian Academy of Sciences, 69 Altenbergerstrasse, A4040, Linz, Austria. This author is supported by the Austrian Science Fund (FWF): P 32660.}
\email{xinlin.cao@oeaw.ac.at; xlcao.math@foxmail.com}

\author{Mourad Sini}
\address{Radon institute (RICAM), Austrian Academy of Sciences, 69 Altenbergerstrasse, A4040, Linz, Austria, This author is partially supported by the Austrian Science Fund (FWF): P 32660.}
\email{mourad.sini@oeaw.ac.at}

\begin{document}

\begin{abstract}

In a $3D$ bounded and $C^{1, \alpha}$-smooth domain $\Omega$, $\alpha \in (0, 1)$, we distribute a cluster of nanoparticles enjoying moderately contrasting relative permittivity and permeability which can be anisotropic. 
We show that the effective permittivity and permeability generated by such cluster is explicitly characterized by the corresponding electric and magnetic polarization tensors of the fixed shape.
The error of the approximation of the scattered fields corresponding to the cluster and the effective medium is inversely proportional to the dilution parameter $c_r:=\frac{\delta}{a}$ where $a$ is the maximum diameter of the nanoparticles and $\delta$ the minimum distance between them. The constant of the proportionality is given in terms of a priori bounds on the cluster of nanoparticles (i.e. upper and lower bounds on their permittivity and permeability parameters, upper bound on the dilution parameter $c_r$, the used incident frequency and the domain $\Omega$).


A key point of the analysis is to show that the Foldy-Lax field appearing in the meso-scale approximation, derived in \cite{AM}, is a discrete form of a (continuous) system of Lippmann-Schwinger equations with a related effective permittivity and permeability contrasts. To derive this, we prove that the Lippmann-Schwinger operator, for the Maxwell system, is invertible in the H\"{o}lder spaces. As a by-product, this shows a H\"{o}lder regularity property of the electromagnetic fields up to the boundary of the inhomogeneity.

\medskip 
\noindent{\bf Keywords}: Maxwell system, anisotropic permittivity and permeability, effective medium theory, integral equations.
 
\medskip
\noindent{\bf AMS subject classification}: 35C15; 35C20; 35Q60

\end{abstract}

\maketitle

\section{Introduction}\label{sec:Intro}
\subsection{Background and motivation}

In this work, we are interested in estimating the electromagnetic fields generated by a cluster of small particles enjoying moderate contrasts of both their electric permittivity and magnetic permeability as compared with the ones of the background. Such a topic, which enters in the general framework of understanding the interaction between waves and matter, traces back at least as far as to the pioneering works of Rayleigh's and Kirchhoff's. At their time, it was known already that the diffracted wave by small
scaled inhomogeneities is dominated by the first multi-poles (poles or dipoles) which  are given by (polarized) point sources located inside the particles. In case of spherically shaped particles, Mie \cite{Mie} derived the full expansion of the electromagnetic fields.These formal expansions were later mathematically justified, see for instance \cite{Dassios}, in the framework of low frequencies expansions. A further step was achieved by Ammari and Kang, see \cite{Ammari-book-1, Ammari-book-2}, where the full expansion at any order and any shape is derived and justified.

These mentioned key works focused on a single or well-separated particles where their interactions is neglected . When the particles are close enough, i.e. behave as a cluster, the mutual interactions between them and the waves must be taken into account. Motivated by the propagation of acoustic waves in the presence of small bubbles, Foldy proposed, in his seminal work \cite{Foldy}, a formal way how to handle the multiple interactions between the bubbles. Under the assumption that the bubbles behave as point-like potentials, he states a close form of the scattered wave by simply eliminating the singularity on the locations of these potentials (see \cite{Martin-book} for more details). Later on, these formal representations of the scattered waves were justified by Berezin and Faddeev, see \cite{Berezin}, in the framework of the Krein's selfadjoint extension of symmetric operators. Another approach to give sense to Foldy's formal method is the regularization method, see \cite{Albeverio-et-al}. Let us mention here that the link between Foldy's result, the one by Berezin and Faddeev and the one in  \cite{Albeverio-et-al} was made in a recent work \cite{Hu-Mantile-Sini}. There was no mention of Foldy's work in Berezin and Faddeev's work \cite{Berezin} nor in \cite{Albeverio-et-al}. These two approaches give us, via the Krein's resolvent representations or regularization, exact formulas to represent the scattered waves
generated by singular potentials. 

Inspired by these representations, it is natural to expect the dominant part of the fields, generated by a cluster of small scaled particles, to be reminiscent to the exact formulas described above with a difference that the scattering coefficients (which are also called the polarization tensors) entering into the close form dominating field are modeled by geometric or contrasts properties of the inhomogeneities. This is called
the Foldy-Lax approximation or the point-interaction approximation. 

There are different ways to derive these dominant fields. We can cite the variation method, via the maximum principle, as proposed by Maz'ya and Movchan \cite{Mazya-1} or integral equations as proposed firstly by Ramm \cite{Ramm-book-1} for some particular models and scaling regimes, and developed since then by many authors \cite{Challa-14, Challa-15, Challa-16}, \cite{Mazya-16, Mazya-17, Nieves-17} and \cite{Ramm-7,Ramm-14, Ramm-15}. We also cite the method of matched asymptotic expansions, see \cite{Bendali-1, Labat-Peron-Tordeux} and the references therein. Regarding the Maxwell system, apart from \cite{Ramm-14, Ramm-15} which are derived more formally with questionable issues, few results are known. The first complete results are derived in \cite{AM-1} for the conductive particles and in \cite{AM} for particles enjoying moderate contrasts of the permittivity and permeability as compared to the ones of the background. The important issue in those two works is that the meso-scaled regime could be handled. In this regime the maximum radius of the particles $a$ and the minimum distance $\delta$ between them could be of the same order. The dilution parameter $c_r:=\frac{\delta}{a}$ indicates how dense the particles can be distributed. The main result in \cite{AM} is to have derived the Foldy-Lax approximation in the meso-scale regime with an error of approximation inversely proportional to $c_r$. In addition, the close form of the dominating field is provided taking into account both the electric permittivity and magnetic permeability contrasts, allowing, eventually, their anisotropy and non-symmetry.  

The work under consideration is related to the effective medium theory for Maxwell. The goal is to derive the effective permittivity and permeability that provides the same scattered electromagnetic field as the one derived in \cite{AM}. In the case of periodically distributed small inhomogeneities, the homogenization applies, see \cite{Bensoussan, Jikov} and the references therein, and provides the equivalent media with averaged materials. In our analysis, we do not rely on homogenization techniques. Rather, our starting point is the Foldy-Lax approximation derived in \cite{AM}. The key point of the analysis is to show that the Foldy-Lax field appearing in the meso-scale approximation is a discrete form of a (continuous) system of Lippmann-Schwinger equations with a related effective permittivity and permeability contrasts. To derive this, we prove an invertibility result for  Lippmann-Schwinger operator, for Maxwell, in the H\"{o}lder spaces. As a by-product, the corresponding electromagnetic fields has a H\"{o}lder regularity up to the boundary of the inhomogeneity. 

In this work, we discuss electromagnetic material with moderate and positive contrasts. The study of the wave propagation in highly contrasting or negative materials is highly attractive and we witness a rapid growth of the number of published works in the recent years. The reason is that for special large scales, we can have resonances (as Minnaert for the acoustics and Mie (or dielectric) or plasmonic resonances in electromagnetism). So far, most of the results concern the acoustic model in the presence of bubbles. The observation, first made in \cite{AFGLH}, is that when the contrast of both the mass density and bulk modulus have a similar but large contrasts, as compared to the ones of the background (the air, for instance), then we have a resonance, named Minnaert resonance. This is a resonance in the sense that if the used incident frequency is close to it then the scattering coefficient (the polarization tensor) becomes large. Such an enhancement can be used to generate volumetric and low dimensional metamaterials (materials with single or double negativity), see \cite{AH, AFGLH, AFLYH, ACP-2, ACP}. As far as the effective medium theory is concerned, few results are known for the electromagnetism with highly contrasting or negative permittivity of permeability. We can cite \cite{L-S:2016, BBF, Schweizer:2017} who assume periodicity and derive the equivalent coefficients, via homogenization, for dielectric nanoparticles. However, we do believe that in the near future this gap will be filled.

\subsection{The electromagnetic scattered fields generated by a cluster of nanoparticles.}\label{subsec1}

We denote by $D:=\cup_{m=1}^{M}D_m$ a collection of $M$ connected, bounded and Lipschitz-smooth particles of $\mathbb{R}^3$. We deal with the electromagnetic wave propagation from this collection of particles when excited by time-harmonic electromagnetic plane wave at a given frequency $\omega$.The nature of the wave propagation is described by the following Maxwell system:

\begin{equation}\label{model1}
	\begin{cases}
	\Delta \times\mathcal{E}-i \omega\mu\mathcal{H}=0,\quad \mbox{in}\quad \mathbb{R}^3\backslash\partial D,\\
	\Delta\times\mathcal{H}+i\omega\epsilon\mathcal{E}=\sigma\mathcal{E},\quad\mbox{in}\quad \mathbb{R}^3\backslash\partial D,\\
	\mathcal{E}=\mathcal{E}^{in}+\mathcal{E}^{s},\\
	\nu\times\mathcal{E}|_{+}=\nu\times\mathcal{E}|_{-}, \nu\times\mathcal{H}|_{+}=\nu\times\mathcal{H}|_{-}\quad\mbox{on}\quad \partial D,
	\end{cases}
\end{equation}
where $\mathcal{E}$ and $\mathcal{H}$ denote the total electric and magnetic fields, $\epsilon$ and $\mu$ are the electric permittivity and the magnetic permeability, respectively and $\sigma$ is the corresponding conductivity. The coefficients $\epsilon, \mu, \sigma$ can be either real or complex tensors or scalar valued functions. Here, $\mathcal{E}^{in}$, being an entire solution of the Maxwell equation, is an incident field and $\mathcal{E}^{s}$ stands for the associated scattered field. The surrounding background of $D$ is homogeneous with the constant parameters $\epsilon_0, \mu_0$ and null conductivity $\sigma$. Furthermore, the scattered wave fields $\mathcal{E}^s$ and $\mathcal{H}^s$ satisfy the Sylver-M\"{u}ller radiation conditions 

\begin{equation}\notag
	\sqrt{\mu_0\epsilon_0^{-1}}\mathcal{H}^{s}(x)\times\frac{x}{|x|}-\mathcal{E}^s(x)=\Oh(\frac{1}{|x|^2}).
\end{equation}

Let $k:=\omega\sqrt{\epsilon_0\mu_0}$, $\mathcal{E}:=\sqrt{\epsilon_0\mu_0^{-1}}E$ and $\mathcal{H}=\sqrt{\epsilon_0\mu_0^{-1}}H$ with $\mu_r:=\frac{\mu}{\mu_0}$ and $\epsilon_r:=\frac{\epsilon+i\sigma/\omega}{\epsilon_0}$, we then derive from \eqref{model1} that

\begin{equation}\label{model-m}
	\begin{cases}
	\mathrm{curl} E-i k \mu_rH=0,\quad\mathrm{curl} H+i k \epsilon_rE=0\quad\mbox{in}\ D,\\	
	\mathrm{curl} E-i k H=0,\quad\mathrm{curl} H+i k E=0\quad\mbox{in}\ \mathbb{R}^3\backslash D,\\
	E=E^{in}+E^s,\\
	\nu\times E|_{+}=\nu\times E|_{-}, \quad \nu\times H|_{+}=\nu\times H|_{-},\quad\mbox{on} \ \partial D,\\
	H^s\times\frac{x}{|x|}-E^s=\Oh({\frac{1}{|x|^2}}),\quad \mbox{as} \ |x|\rightarrow\infty,
	\end{cases}
\end{equation}
where $E^{in}$ and $H^{in}$ are the relative incident wave fields fulfilling the whole system in $\mathbb{R}^3$ that 
\begin{equation}\notag
	\begin{cases}
	\mathrm{curl} E^{in}-i k H^{in}=0,\\
	\mathrm{curl}H^{in}+i k E^{in}=0.
	\end{cases}
\end{equation}

Let the incident waves $(E^{in}, H^{in})$ be typical plane waves of the form $E^{in}=E^{in}(x, \theta)=P e^{i k\theta\cdot x}$ for $x\in\mathbb{R}^3$ and $H^{in}=H^{in}(x,\theta)=P\times\theta e^{i k \theta\cdot x}/i k$, where $P$ is the constant vector modeling the polarization direction and $\theta$ fulfilling $|\theta|=1$ and $P\cdot\theta=0$ is the incident direction. 


Let $D_m=a \mathbf{B}_m+\mathbf{z}_m, m=1,\cdots, M$, be the connected small components of the medium $D$, which are characterized by the parameter $a>0$ and the locations $\mathbf{z}_m$. Each $\mathbf{B}_m$ containing the origin is a bounded Lipschitz domain and fulfills that $\mathbf{B}_m\subset B_0^{1}$. In fact, we have the following definition.


\begin{definition}\label{def-back}
	Define $a:=\max\limits_{1\leq m \leq M}\mathrm{diam}(D_m)$, and $\delta:=\min\limits_{{1\leq m,j\leq M}\atop{m\neq j}}\delta_{mj}:=\min\limits_{{1\leq m,j\leq M}\atop{m\neq j}} \mathrm{dist}(D_m, D_j)$.
\end{definition}

Our analysis in this work is based on \cite{AM}. Therefore, we first recall the needed assumptions stated there regarding the model \eqref{model-m}. More details can be found in \cite{AM}.

\uppercase\expandafter{\romannumeral1}. \emph{Assumptions on the exterior domain}. The electric permittivity $\epsilon_0$ and the magnetic permeability $\mu_0$ of the background are real, scalar and constants  such that
\begin{equation}\notag
	k=\omega\sqrt{\epsilon_0\mu_0}>0.
\end{equation}



\uppercase\expandafter{\romannumeral2}. \emph{Assumptions on the permittivity and permeability of each particle}. Assume that $\epsilon_r$ and $\mu_r$ satisfy the following  conditions:
\begin{enumerate}
	\item {\bl $\epsilon_r$ and $\mu_r$ are $\mathbb{W}^{1,\infty}$-regular and both positive definite real symmetric matrices.}
%
%

    \item Define the contrasts with respect to $\epsilon_r$ and $\mu_r$ as 
    \begin{equation}\notag
    	\mathcal{C}_K:=K-I,
    \end{equation}
    where $K=\epsilon_r, \mu_r$ and $I\in\mathbb{R}^3\times\mathbb{R}^3$ is the identity matrix. $\mathcal{C}_K$ is a {real} valued $3\times3$-tensor with the corresponding derivatives satisfy the essentially uniformly bounded condition as
    \begin{equation}\label{bdd-contr}
    	\lVert\mathcal{C}_K\rVert_{\mathbb{W}^{1,\infty}(\cup_{m=1}^MD_m)}\leq c_\infty,
    \end{equation}
    as well as the essentially uniformly coercive conditions:
    \begin{align}\label{coer-contr}
    	\Re \left(\mathcal{C}_{\epsilon_r}(x)\xi\cdot\bar{\xi}\right)&\geq c_\infty^{\epsilon-}|\xi|^2,\notag\\
    	\mathcal{C}_{\mu_r}(x)\xi\cdot\bar{\xi}&\geq c_\infty^{\mu-}|\xi|^2,
    \end{align}
    for $x\in D$, where $c_\infty^{\epsilon-}$ and $c_\infty^{\mu-}$ are two positive constants. 
\end{enumerate}

\uppercase\expandafter{\romannumeral3}. \emph{Assumptions on the distribution of the cluster of particles}. Let $\Omega$ be a bounded domain of unit volume, containing the particles $D_m$, $m=1,2,\cdots, M$, with $\partial\Omega$ which is $C^{1,\alpha}$-regular for $\alpha\in (0,1)$. We divide $\Omega$ into $[\delta^{-3}]$ subdomains $\Omega_m$, $m=1,2,\cdots, [\delta^{-3}]$, distributed periodically, such that $D_m\subset\Omega_m$. Then we see that $M=\Oh(\delta^{-3})$. Let each $\Omega_m$ be a cube of volume $\delta^3$. Furthermore, the minimum distance defined in Definition \ref{def-back} fulfills that
\begin{equation}\label{dis}
\delta=c_r a,
\end{equation}
where $c_r$ is a constant satisfying the lower bound
\begin{equation}\notag
	c_r=\frac{\delta}{a}\geq\max\left\{1, \frac{2c_\wedge k c_\infty^2}{\max(c_\infty^{\epsilon-},c_\infty^{\mu-})}\right\},
\end{equation}
with $c_\wedge$ being a positive constant independent on the parameters of the model. See Fig \ref{fig:omega-perodic} and Fig \ref{fig:local} for the schematic illustrations of the global distribution as well as the local relation between any two particles.

\begin{figure}[h!]
	\centering
	\includegraphics[width=0.4\linewidth]{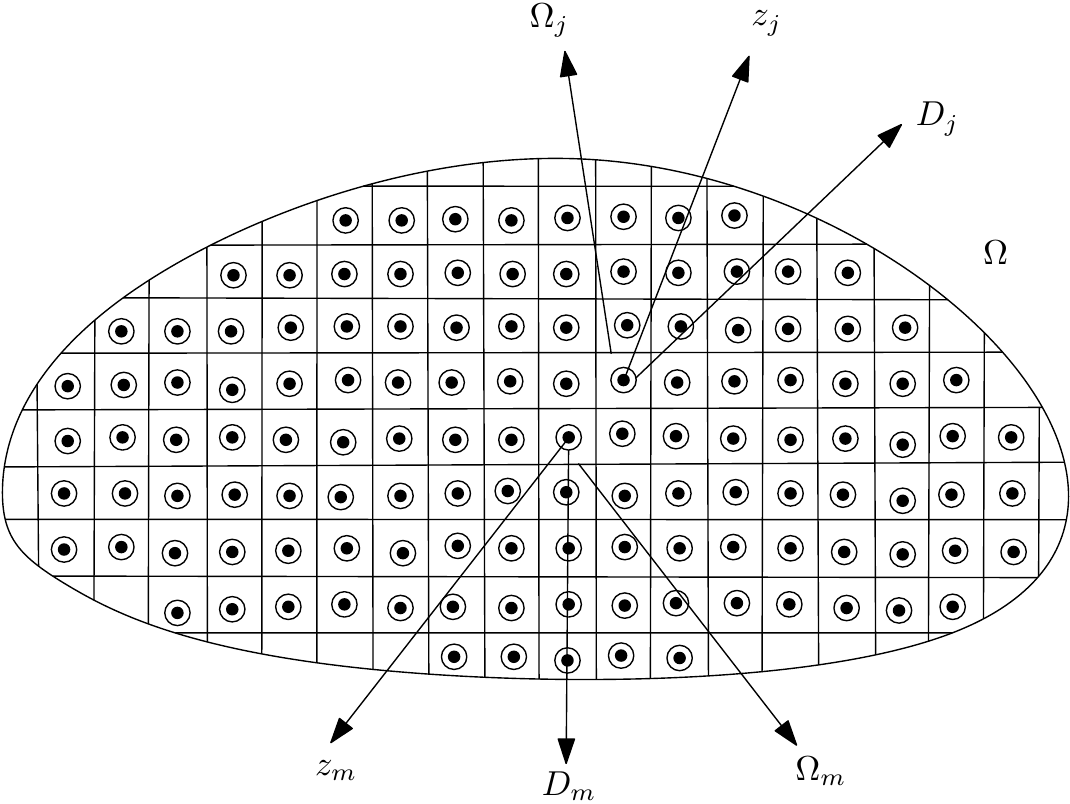}
	\caption{A schematic illustration for the global distribution of the particles in $\Omega$.}
	\label{fig:omega-perodic}
\end{figure}

\begin{figure}[h!]
	\centering
	\includegraphics[width=0.3\linewidth]{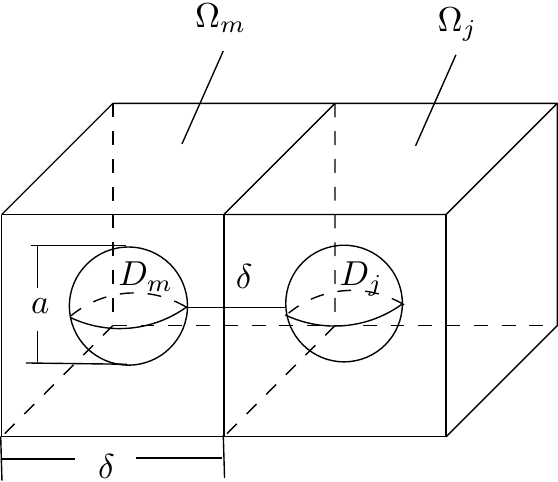}
	\caption{A schematic illustration for the local relation between any two particles.}
	\label{fig:local}
\end{figure}

\begin{remark}\label{rem1}
	In the above condition \uppercase\expandafter{\romannumeral3}, without loss of generality, we assume that the particles are distributed periodically in $\Omega$. However, the periodicity is used only for simplicity of the exposition. It is not required to derive our results.  We refer to \cite{Helm1} for the case of non-periodic distribution of the particles.
\end{remark}

\begin{remark}\label{layer}
	Since the shape of $\Omega$ is arbitrary, the intersecting part between the set of the cubes and $\partial\Omega$ is not necessarily to be empty, unless $\Omega$ is a cube. It is known that for any $m=1,\cdots, M$, the volume of $\Omega_m$ satisfies $|\Omega_m|=\delta^3$, which implies that the maximum radius of $\Omega_m$ is of order $\delta$, and thus the intersecting surfaces with $\partial \Omega$ possess the area of order $\delta^2$. Since the total area of $\partial\Omega$ is of order one, we can derive that the number of such particles will not exceed of order $\delta^{-2}$, and therefore the volume of this set is of order $\delta$, for sufficiently small $a$.
\end{remark}


\subsection{Main results.}

Our starting point is the following Foldy-Lax approximation (or point-interaction approximation) of the electromagnetic fields generated by a cluster of small scaled inhomogeneties in the mesoscale regime derived in \cite{AM}.

\begin{theorem} \cite[Theorem 2.1]{AM}.\label{FL-approximation}
	Let $\Omega$ be a bounded domain with Lipschitz boundary. 
	Under the assumptions (\uppercase\expandafter{\romannumeral1}, \uppercase\expandafter{\romannumeral2},\uppercase\expandafter{\romannumeral3}) above, we have the following expansion:

	\begin{equation}\label{dis-far}
		E^{\infty}(\hat{x})=\sum_{m=1}^M\left(\frac{k^2}{4\pi}e^{-i k \hat{x}\cdot z_m}\hat{x}\times({\R}_m^{\epsilon_r}\times\hat{x})+\frac{ik}{4\pi}e^{-i k \hat{x}\cdot z_m}\hat{x}\times{\Q}_m^{\mu_r}\right)+\Oh(k c_{\infty}(\frac{1}{c_{\infty}^{\epsilon-}}+\frac{1}{c_{\infty}^{\mu-}})c_r^{-7}),
	\end{equation}
	with $(\R_m^{\epsilon_r}, \Q_m^{\mu_r})$ fulfilling the following invertible linear algebraic system
	\begin{align}
		[P_{D_m}^{\epsilon_r}]^{-1}\R_m^{\epsilon_r}&=\sum_{j=1\atop j\neq m}^M \left[\Pi_k(z_m, z_j)\R_j^{\epsilon_r}+i k \nabla\Phi_k(z_m, z_j)\times\Q_j^{\mu_r}\right]+E^{in}(z_m),\notag\\
		[P_{D_m}^{\mu_r}]^{-1}\Q_m^{\mu_r}&=\sum_{j=1\atop j\neq m}^M \left[\Pi_k(z_m, z_j)\Q_j^{\mu_r}-i k \nabla\Phi_k(z_m, z_j)\times\R_j^{\epsilon_r}\right]+H^{in}(z_m)\notag
	\end{align}
	for $m=1,\cdots, M$. $c_{\infty}^{\epsilon_r-}, c_{\infty}^{\mu_r-}$ are the positive constants introduced in \eqref{coer-contr} 
\end{theorem}

Here, $[P_{D_m}^{\epsilon_r}]$ and $[P_{D_m}^{\mu_r}]$ are respectively the electric polarization tensor and the magnetic polarization tensor associated with \eqref{model-m}, which satisfy the scaling property
    \begin{equation}\label{polar}
    [P_{D_m}^{\epsilon_r}]=a^3[P_{\mathbf{B}_m}^{\epsilon_r}],\quad [P_{D_m}^{\mu_r}]=a^3[P_{\mathbf{B}_m}^{\mu_r}],
    \end{equation}
    as well as the boundedness property
    \begin{equation}\label{eigen}
    \lambda^-_{\epsilon_r}a^3|V|^2\leq[P_{D_m}^{\epsilon_r}]V\cdot V\leq \lambda^+_{\epsilon_r}a^3|V|^2, \ \lambda^-_{\mu_r}a^3|V|^2\leq[P_{D_m}^{\mu_r}]V\cdot V\leq \lambda^+_{\mu_r}a^3|V|^2,
    \end{equation}
    with positive constants $\lambda^-_{\epsilon_r}, \lambda^-_{\mu_r}$, for any vector $V$.
    Then we denote
    \begin{equation}\label{lamda+}
    \lambda^-_{(\epsilon_r, \mu_r)}:=\min\{\lambda^-_{\epsilon_r}, \lambda^-_{\mu_r} \},\quad \mbox{and} \ \lambda^+_{(\epsilon_r, \mu_r)}:=\max\{\lambda^+_{\epsilon_r}, \lambda^+_{\mu_r} \}.
    \end{equation}
    
\uppercase\expandafter{\romannumeral4}. \emph{Assumption on the shape of $D_m$'s.}\label{\romannumeral4} We assume that the $B_m$'s have the same shapes and we denote
    \begin{equation}\label{polar2}
    	[P_{\mathbf{B}_m}^{\epsilon_r}]:=[P_0^{\epsilon_r}],\quad [P_{\mathbf{B}_m}^{\mu_r}]:=[P_0^{\mu_r}].
    \end{equation}
Actually, we just need the shapes to have the same polarization tensors. 
\bigskip

Now, we are in a position to state the main contribution of this work.
\begin{theorem}\label{main-1}
	Let $\Omega$ be a bounded domain with $C^{1,\alpha}$ regularity for $\alpha\in(0,1)$. 
	{\bl Suppose that the wave number $k$ fulfills 
	\begin{equation}\notag
		g(\alpha, k)c_{\mathrm{reg}}(\alpha, \Omega)c_r^{-3}c_\infty^{(\epsilon_r, \mu_r)}<1,
\end{equation}
where 
\begin{equation}\notag
	g(\alpha, k):=k^{3+\alpha}+k^3+k^2+k+1,
\end{equation}
and $c_\infty^{(\epsilon_r, \mu_r)}$ is defined by 
\begin{equation}\notag
	c_\infty^{(\epsilon_r, \mu_r)}:=\max\{\lVert[P_0^{\epsilon_r}]\rVert_{L^{\infty}(\Omega)}, \lVert[P_0^{\mu_r}]\rVert_{L^{\infty}(\Omega)}\},
\end{equation}
and $c_{\mathrm{reg}}(\alpha, \Omega)$ is a positive constant depending on $\alpha$ and $\Omega$, see \eqref{c-reg}. Then} under the assumptions (\uppercase\expandafter{\romannumeral1}, \uppercase\expandafter{\romannumeral2},\uppercase\expandafter{\romannumeral3}, \uppercase\expandafter{\romannumeral4}) above, we have the following expansion:
	\begin{equation}\label{main-diff}
	E^\infty(\hat{x})- E_{\mathrm{eff}}^{\infty}(\hat{x})=\Oh\left(k\lambda^+_{(\epsilon_r,\mu_r)}(c_0^{\epsilon_r}+c_0^{\mu_r})(kc_{P_0}^{\epsilon_r}c_0^{\epsilon_r,-1}+c_{P_0}^{\mu_r}c_0^{\mu_r,-1})c_r^{-6}
		\right),
	\end{equation}
{\bl	where $\lambda^+_{(\epsilon_r,\mu_r)},  c_{P_0}^{\epsilon_r}, c_{P_0}^{\mu_r}, c_0^{\epsilon_r,-1}, c_0^{\mu_r,-1}, c_0^{\epsilon_r}\mbox{ and }c_0^{\mu_r}$ 
	are positive constants depending on the a-prior bounds for $\epsilon_r$, $\mu_r$ and $\Omega$ and they are given by \eqref{lamda+}, \eqref{add2} and \eqref{notation}, respectively. }
	Here, $E^{\infty}_{\mathrm{eff}}$ is the far-field pattern corresponding to the following electromagnetic scattering problem for the effective medium as $a\to 0$:
	\begin{equation}\label{model-equi}
	\begin{cases}
	\mathrm{curl}U^{\mathring{\epsilon}_r}-i k \mathring{\mu}_r V^{\mathring{\mu}_r}=0,\quad \mathrm{curl} V^{\mathring{\mu}_r}+i k\mathring{\epsilon_r} U^{\mathring{\epsilon}_r}=0\quad\mbox{in} \ D,\\
	\mathrm{curl} U^{\mathring{\epsilon}_r}-i k V^{\mathring{\mu}_r}=0,\quad\mathrm{curl} V^{\mathring{\mu}_r}+i k U^{\mathring{\epsilon}_r}=0\quad\mbox{in}\ \mathbb{R}^3\backslash D,\\
	U^{\mathring{\epsilon}_r}=U_s^{\mathring{\epsilon}_r}+U^{\mathring{\epsilon}_r}_{in},\\
	\nu\times U^{\mathring{\epsilon}_r}|_{+}=\nu\times U^{\mathring{\epsilon}_r}|_{-}, \quad \nu\times V^{\mathring{\mu}_r}|_{+}=\nu\times V^{\mathring{\mu}_r}|_{-},\quad\mbox{on} \ \partial D,\\
	V^{\mathring{\mu}_r}_s\times\frac{x}{|x|}-U^{\mathring{\epsilon}_r}_s=\Oh({\frac{1}{|x|^2}}),\quad \mbox{as} \ |x|\rightarrow\infty,
	\end{cases}
	\end{equation}
The electric permittivity and magnetic permeability of the effective medium are given by
	\begin{equation}\notag
		\mathring{\epsilon}_r=I+c_r^{-3}[P^{\epsilon_r}_0], \quad
		\mathring{\mu}_r=I+c_r^{-3}[P^{\mu_r}_0].
	\end{equation}
\end{theorem}

{\bl
\begin{remark}
	The assumption in (\uppercase\expandafter{\romannumeral2}. (1)) can be changed into ``$\epsilon_r$ and $\mu_r$ are real positive definite matrices with arbitrary $\Re \epsilon_r$ ". In particular, $\epsilon_r$ and $\mu_r$ do not need to be symmetric. In this case, for $m=,1,2, \cdots, M$, we define
	\begin{equation}\notag
		\mathcal{A}_f^m:=\frac{1}{D_m}\int_{D_m} f\,dv,
	\end{equation}
	for any $f\in L^2(D_m)$. Then if either $\epsilon_r$ or $\mu_r$ is not symmetric, with $(\mathcal{A}^m_{\epsilon_r})_{m=1}^M$ and $(\mathcal{A}^m_{\mu_r})_{m=1}^M$ standing for the respective average over each $D_m$, the Foldy-Lax approximation for the far-field keeps the same form as \eqref{dis-far}. But the invertible linear algebraic system with repsect to $(\R_m^{\epsilon_r}, \Q_m^{\mu_r})$ now becomes
	\begin{align}
	[\mathcal{T}_{D_m}^{\mathcal{A}^m_{{\epsilon_r}^{\mathrm{T}}}}]^{-1}\R_m^{\epsilon_r}&=\sum_{j=1\atop j\neq m}^M \left[\Pi_k(z_m, z_j)\R_j^{\epsilon_r}+i k \nabla\Phi_k(z_m, z_j)\times\Q_j^{\mu_r}\right]+E^{in}(z_m),\notag\\
	[\mathcal{T}_{D_m}^{\mathcal{A}^m_{{\mu_r}^{\mathrm{T}}}}]^{-1}\Q_m^{\mu_r}&=\sum_{j=1\atop j\neq m}^M \left[\Pi_k(z_m, z_j)\Q_j^{\mu_r}-i k \nabla\Phi_k(z_m, z_j)\times\R_j^{\epsilon_r}\right]+H^{in}(z_m)\notag,
	\end{align}
	for $m=1,2, \cdots, M$, where $[\mathcal{T}_{D_m}^{\mathcal{A}^m_{{\epsilon_r}^{\mathrm{T}}}}]^{-1}$ and $[\mathcal{T}_{D_m}^{\mathcal{A}^m_{{\mu_r}^{\mathrm{T}}}}]^{-1}$ are given by 
	\begin{equation}\notag
		[\mathcal{T}_{D_m}^{\mathcal{A}^m_{{\epsilon_r}^{\mathrm{T}}}}]^{-1}:=\mathcal{A}^m_{\C_{{\epsilon_r}^{\mathrm{T}}}}[P_{D_m}^{\mathcal{A}^m_{{\epsilon_r}^{\mathrm{T}}}}]^{-1}(\mathcal{A}^m_{\C_{\epsilon_r}})^{-1},
	\end{equation}
	and
	\begin{equation}\notag
		[\mathcal{T}_{D_m}^{\mathcal{A}^m_{{\mu_r}^{\mathrm{T}}}}]^{-1}:=\mathcal{A}^m_{\C_{{\mu_r}^{\mathrm{T}}}}[P_{D_m}^{\mathcal{A}^m_{{\mu_r}^{\mathrm{T}}}}]^{-1}(\mathcal{A}^m_{\C_{\mu_r}})^{-1}.
	\end{equation}
	We refer \cite[Theorem 2.1]{AM} for more detailed discussions. Based on \eqref{dis-far}, we show that the corresponding effective medium under this circumstances still fulfills \eqref{model-equi}, with the new electric permittivity $\mathring{\epsilon_r}$ and $\mathring{\mu_r}$ being represented as
	\begin{equation}\notag
		\mathring{\epsilon_r}=I+c_r^{-3}[\mathcal{T}_{0}^{\mathcal{A}_{{\epsilon_r}^{\mathrm{T}}}}],\quad
		\mathring{\mu_r}=I+c_r^{-3}[\mathcal{T}_{0}^{\mathcal{A}_{{\mu_r}^{\mathrm{T}}}}],
	\end{equation}
	under the similar assumption as \uppercase\expandafter{\romannumeral4}.
	\end{remark}
}

\subsection{A brief description of the arguments}

Here, we briefly describe the whole steps and ideas of the proof of Theorem \ref{main-1}.

Recall the Foldy-Lax approximation given in \cite[Theorem 2.1]{AM}. By the variable substitution $U_m:=[P^{\epsilon_r}_{D_m}]^{-1}\R^{\epsilon_r}_m, V_m:=[P^{\mu_r}_{D_m}]^{-1}\Q^{\mu_r}_m$, we can rewrite the far-field expansion as 
\begin{align}\label{far-dis0}
E^{\infty}(\hat{x})=\sum_{m=1}^{M}\left(\frac{k^2}{4\pi}e^{-i k \hat{x}\cdot z_m}\hat{x}\times (a^3[P_0^{\epsilon_r}]U_m\times \hat{x})+\frac{i k}{4\pi}e^{-i k \hat{x}\cdot z_m}\hat{x}\times a^3[P_0^{\mu_r}]V_m\right)\notag\\
+\Oh(k c_\infty(\frac{1}{c_{\infty}^{\epsilon-}}+\frac{1}{c_{\infty}^{\mu-}})c_r^{-7}),
\end{align}
where $(U_m, V_m)_{m=1}^M$ are solutions to the invertible linear algebraic system 
\begin{align}\label{dislinear0}
&U_m-a^3\sum_{j=1\atop j\neq m}^M\left[ \Pi_k(z_m,z_j)[P_0^{\epsilon_r}]U_j+i k \nabla\Phi_k(z_m,z_j)\times([P^{\mu_r}_0]V_j)\right]=E^{in}(z_m),\notag\\
&V_m-a^3\sum_{j=1\atop j\neq m}^M\left[ \Pi_k(z_m,z_j)[P_0^{\mu_r}]V_j-i k \nabla\Phi_k(z_m,z_j)\times([P^{\epsilon_r}_0]U_j)\right]=H^{in}(z_m),
\end{align}
and $c_\infty$, $c_{\infty}^{\epsilon-}$ and $c_{\infty}^{\mu-}$ are defined by \eqref{bdd-contr} and \eqref{coer-contr}, respectively. 

Based on \eqref{dislinear0}, we introduce the Lippmann-Schwinger equation with respect to two constant $3\times3$-tensors $[C^{\mathrm{T}}_{\epsilon_r}]$ and $[C^{\mathrm{T}}_{\mu_r}]$ as
\begin{align}\label{ls10}
&U^{\mathring{\epsilon}_r}(x)-\int_{\Omega}\left(\Pi_k(x,z)c_r^{-3}[C^{\mathrm{T}}_{\epsilon_r}]U^{\mathring{\epsilon}_r}+i k \nabla\Phi_k(x,z)\times(c_r^{-3}[C^{\mathrm{T}}_{\mu_r}]V^{\mathring{\mu}_r})\right)\,dz =E^{in}(x,\theta),\notag\\
&V^{\mathring{\mu}_r}(x)-\int_{\Omega}\left(\Pi_k(x,z)c_r^{-3}[C^{\mathrm{T}}_{\mu_r}]V^{\mathring{\mu}_r}-i k \nabla\Phi_k(x,z)\times(c_r^{-3}[C^{\mathrm{T}}_{\epsilon_r}]U^{\mathring{\epsilon}_r})\right)\,dz =H^{in}(x,\theta),
\end{align}
which models the unique solution to the electromagnetic scattering problem of equivalent medium \eqref{model-equi}. More explanations for $[C^{\mathrm{T}}_{\epsilon_r}]$ and $[C^{\mathrm{T}}_{\mu_r}]$ will be presented later in Section \ref{sec-compare} by Definition \ref{def2}. The corresponding far-field possesses the form 
\begin{equation}\label{far-intro}
	E^{\infty}_{\mathrm{eff}}(\hat{x})=\int_{\Omega}\left( \frac{k^2}{4\pi}e^{-ik\hat{x}\cdot z}\hat{x}\times (c_r^{-3}[C^{\mathrm{T}}_{\epsilon_r}]U^{\mathring{\epsilon}_r}(z)\times\hat{x})+\frac{ik}{4\pi}e^{-ik\hat{x}\cdot z}\hat{x}\times c_r^{-3}[C^{\mathrm{T}}_{\mu_r}]V^{\mathring{\mu}_r}(z)\right)\,dz.
\end{equation}
Our main result is to derive the following error estimate
\begin{equation}\label{error-m}
	E^\infty(\hat{x}) - E_{\mathrm{eff}}^{\infty}(\hat{x})=\Oh\left(k\lambda^+_{(\epsilon_r,\mu_r)}(c_0^{\epsilon_r}+c_0^{\mu_r})(kc_{P_0}^{\epsilon_r}c_0^{\epsilon_r,-1}+c_{P_0}^{\mu_r}c_0^{\mu_r,-1})c_r^{-6}\right),\quad a\to 0.
\end{equation}
We show that $[C^{\mathrm{T}}_{\epsilon_r}]=[P_0^{\epsilon_r}] +O(c^{-3}_r)$ and $[C^{\mathrm{T}}_{\mu_r}]=[P_0^{\mu_r}] +O(c^{-3}_r)$. Therefore, in terms of \eqref{error-m}, $E^{\infty}_{\mathrm{eff}}(\hat{x})$ in \eqref{far-intro} reduces, keeping the same notation, to
\begin{equation}\label{add1}
E^{\infty}_{\mathrm{eff}}(\hat{x})=\int_{\Omega}\left(\frac{k^2}{4\pi}e^{-i k\hat{x}\cdot z}\hat{x}\times (c_r^{-3}[P_0^{\epsilon_r}] U^{\mathring{\epsilon}_r}\times\hat{x})+\frac{i k}{4\pi}e^{-i k\hat{x}\cdot z}\hat{x}\times c_r^{-3}[P_0^{\mu_r}]V^{\mathring{\mu}_r}\right)\,dz.
\end{equation}

In order to derive \eqref{error-m}, the first step is to investigate the regularity of the solution $(U^{\mathring{\epsilon_r}}, V^{\mathring{\mu_r}})$ of (\ref{ls10}). We prove that if the bounded domain $\Omega$ is $C^{1,\alpha}$ regular, then {\bl $(U^{\mathring{\epsilon_r}}, V^{\mathring{\mu_r}})\in C^{0,\alpha}({\overline{\Omega}})\times C^{0,\alpha}({\overline{\Omega}})$}, for $\alpha\in (0,1)$, {\bl provided that 
	\begin{equation}\notag
	g(\alpha, k)c_{\mathrm{reg}}(\alpha, \Omega)c_r^{-3}c_\infty^{(\epsilon_r, \mu_r)}<1,
	\end{equation}
	where $g(\alpha, k):=k^{3+\alpha}+k^3+k^2+k+1$, and $c_{\mathrm{reg}}(\alpha, \Omega)$ is a positive constant depending on $\alpha$ and $\Omega$. $c_\infty^{(\epsilon_r, \mu_r)}$ is a positive constant related to $[P_0^{\epsilon_r}]$ and $[P_0^{\mu_r}]$}.
\bigskip

Next, with the help of counting lemma and the above regularity property, we derive the most important estimate  of the difference between the solutions $(U_m, V_m)$ to the linear algebraic system and the solutions $(U^{\mathring{\epsilon_r}}, V^{\mathring{\mu_r}})$ of the Lippmann-Schwinger equation. Precisely, we estimate the $l_2$-norm of the vectors $
	(I-K^{\epsilon_r})\frac{1}{|S_m|}\int_{S_m}U^{\mathring{\epsilon}}(x)\,dx-U_m,
$
 and $
	(I-K^{\mu_r})\frac{1}{|S_m|}\int_{S_m}V^{\mathring{\epsilon}}(x)\,dx-V_m,
$, for $m=1,\cdots, M$,  as
\begin{equation}\notag
	\left(\sum_{m=1}^M|(I-K^{\epsilon_r})\frac{1}{|S_m|}\int_{S_m}U^{\mathring{\epsilon}}(x)\,dx-U_m|^2\right)^{1/2}=\Oh\left(\lambda^+_{(\epsilon_r,\mu_r)}c_0^{\epsilon_r,-1}(c_0^{\epsilon_r}+c_0^{\mu_r})c_r^{-\frac{9}{2}}a^{-\frac{3}{2}}\right),
\end{equation}
and 
\begin{equation}\notag
	\left(\sum_{m=1}^M|(I-K^{\mu_r})\frac{1}{|S_m|}\int_{S_m}V^{\mathring{\epsilon}}(x)\,dx-V_m|^2\right)^{1/2}=\Oh\left(\lambda^+_{(\epsilon_r,\mu_r)}c_0^{\mu_r,-1}(c_0^{\epsilon_r}+c_0^{\mu_r})c_r^{-\frac{9}{2}}a^{-\frac{3}{2}}\right),
\end{equation}
where $K^{\epsilon_r}$ and $K^{\mu_r}$ are bounded linear operators associated with the constant tensors $[C_{\epsilon_r}^{\mathrm{T}}]$ and $[C_{\mu_r}^{\mathrm{T}}]$ respectively; see Definition \ref{def2} in Section \ref{sec-compare} for more detailed discussions. 
To derive these estimates, with tedious but needed computations, we systematically utilize the properties of the Newtonian operator, as Cald\'{e}ron-Zygmund inequality, and the invertibility property of the algebraic system in Theorem \ref{FL-approximation}. Based on these estimates, we finally derive the error estimate for $E^\infty-E^\infty_{\mathrm{eff}}$.
\bigskip

The rest of this  paper is organized as follows. In Section \ref{sec-dis}, we first recall the Foldy-Lax approximation of the electromagnetic fields introduced in \cite{AM} for the discrete case with slightly different representations, which is more appropriate for our later comparison with the continuous model. In addition, we derive the $C^{0, \alpha}$-regularity of the solution $(U^{\mathring{\epsilon_r}}, V^{\mathring{\mu_r}})$ appearing in the, limiting, far-field \eqref{add1}. Section \ref{sec-compare} is devoted to prove our main result by investigating the equivalent medium as $a\rightarrow0$ and evaluate the difference between the discrete form of the far-field $E^{\infty}(\hat{x})$ and the corresponding asymptotic form $E^\infty_{\mathrm{eff}}(\hat{x})$, which shows the error estimate of $E^\infty_{\mathrm{eff}}(\hat{x})-E^{\infty}(\hat{x})$. Finally, in Section \ref{sec-prop}, we give the proof of one of the  essential propositions used in Section \ref{sec-compare}.

\section{Lippmann-Schwinger equation and the $C^{0,\alpha}$ regularity to its solution}\label{sec-dis}

\subsection{Mesoscale Approximation for the Electromagnetic Fields: Discrete form}

In this subsection, based on \cite{AM}, we recall the discrete form of the Foldy-Lax Approximation as well as the corresponding linear algebraic system as the following proposition.

\begin{proposition}\cite[Theorem 2.1]{AM}\label{lem-dis}
	Under the assumptions $(\uppercase\expandafter{\romannumeral1}, \uppercase\expandafter{\romannumeral2}, \uppercase\expandafter{\romannumeral3})$, the mesoscale electromagnetic scattering problem \eqref{model-m} has one and only one solution and the corresponding far-field fulfills the following expansion:
	\begin{equation}\label{dis-far0}
	E^{\infty}(\hat{x})=\sum_{m=1}^{M}\left(\frac{k^2}{4\pi}e^{-i k \hat{x}\cdot z_m}\hat{x}\times (\mathcal{R}^{\epsilon_r}_m\times \hat{x})+\frac{i k}{4\pi}e^{-i k \hat{x}\cdot z_m}\hat{x}\times\mathcal{Q}^{\mu_r}_m\right)+\Oh(k c_\infty(\frac{1}{c_{\infty}^{\epsilon-}}+\frac{1}{c_{\infty}^{\mu-}})c_r^{-7}).
	\end{equation}
where $(\R^{\epsilon_r}_m, \Q^{\mu_r}_m)_{m=1}^M$ are the solutions to the following invertible linear algebraic system 
	\begin{align}\label{dis-linear0}
	[P_{D_m}^{\epsilon_r}]^{-1}\R_m^{\epsilon_r}&=\sum_{j=1\atop j\neq m}^M[\Pi_k(z_m,z_j)\R^{\epsilon_r}_j+ i k \nabla \Phi_k(z_m,z_j)\times\Q^{\mu_r}_j]+E^{in} (z_m),\notag\\
	[P_{D_m}^{\mu_r}]^{-1}\Q_m^{\mu_r}&=\sum_{j=1\atop j\neq m}^M[\Pi_k(z_m,z_j)\Q^{\mu_r}_j- i k \nabla \Phi_k(z_m,z_j)\times\R^{\epsilon_r}_j]+H^{in}(z_m).  	  	
	\end{align}
	Furthermore, $(\R^{\epsilon_r}_m, \Q^{\mu_r}_m)_{m=1}^M$ satisfy the estimates for $k>1$ as,
	\begin{align}\label{dis-inver0}
&\left(\sum_{m=1}^M |{\R}_m^{\epsilon_r}|^2\right)^\frac{1}{2}\leq \frac{9\lambda^+_{(\epsilon_r,\mu_r)}a^3}{8}\left(\frac{1}{3} \left( \sum_{m=1}^M |H^{in}(z_m)|^2\right)^\frac{1}{2}+\left( \sum_{m=1}^M|E^{in}(z_m)|^2\right)^{\frac{1}{2}}\right),\notag\\
		&\left(\sum_{m=1}^M |{\Q}_m^{\mu_r}|^2\right)^\frac{1}{2}\leq \frac{9\lambda^+_{(\epsilon_r,\mu_r)}a^3}{8}\left( \left( \sum_{m=1}^M |H^{in}(z_m)|^2\right)^\frac{1}{2}+\frac{1}{3}\left( \sum_{m=1}^M|E^{in}(z_m)|^2\right)^{\frac{1}{2}}\right),
	\end{align}
	provided $c_r=\delta/a\geq 3k\lambda^+_{(\epsilon_r,\mu_r)}$.
\end{proposition}

The proof of deriving \eqref{dis-linear0} and \eqref{dis-inver0} can be founded directly in \cite[Proposition 5.1 and Proposition 5.2]{AM}. The corresponding far-field representation here \eqref{dis-far0} is slightly different from the form presented in \cite[Theorem 2.1]{AM}.

Now, based on the scaling property \eqref{polar} as well as the constant assumption for the polarization tensors \eqref{polar2}, we rewrite the linear algebraic system \eqref{dis-linear0} for $(\R^{\epsilon_r}_m, \Q^{\mu_r}_m)_{m=1}^M$  as
\begin{align}\label{rewr1}
	&[P^{\epsilon_r}_{D_m}]^{-1}\R^{\epsilon_r}_m-a^3\sum_{j=1\atop j\neq m}^M\left[ \Pi_k(z_m,z_j)[P_0^{\epsilon_r}][P_{D_j}^{\epsilon_r}]^{-1}\R^{\epsilon_r}_j+i k \nabla\Phi_k(z_m,z_j)\times[P_0^{\mu_r}][P_{D_j}^{\mu_r}]^{-1}\Q^{\mu_r}_j\right]=E^{in}(z_m),\notag\\
	&[P^{\mu_r}_{D_m}]^{-1}\Q^{\mu_r}_m-a^3\sum_{j=1\atop j\neq m}^M\left[ \Pi_k(z_m,z_j)[P_0^{\mu_r}][P_{D_j}^{\mu_r}]^{-1}\Q^{\mu_r}_j-i k \nabla\Phi_k(z_m,z_j)\times[P_0^{\epsilon_r}][P_{D_j}^{\epsilon_r}]^{-1}\R^{\epsilon_r}_j\right]=H^{in}(z_m).
\end{align}
Denote
\begin{equation}\label{v-change}
	U_m:=[P^{\epsilon_r}_{D_m}]^{-1}\R^{\epsilon_r}_m,\quad V_m:=[P^{\mu_r}_{D_m}]^{-1}\Q^{\mu_r}_m.
\end{equation}
Then \eqref{rewr1} becomes
\begin{align}\label{dislinear}
&U_m-a^3\sum_{j=1\atop j\neq m}^M\left[ \Pi_k(z_m,z_j)[P_0^{\epsilon_r}]U_j+i k \nabla\Phi_k(z_m,z_j)\times([P^{\mu_r}_0]V_j)\right]=E^{in}(z_m),\notag\\
&V_m-a^3\sum_{j=1\atop j\neq m}^M\left[ \Pi_k(z_m,z_j)[P_0^{\mu_r}]V_j-i k \nabla\Phi_k(z_m,z_j)\times([P^{\epsilon_r}_0]U_j)\right]=H^{in}(z_m).
\end{align}

In the following discussions, when it comes to the discrete far-field approximation, we mainly use the representation \eqref{dis-far0} as well as \eqref{dislinear}.


\subsection{Lippmann-Schwinger equation and the $C^{0,\alpha}$ regularity of the associated solutions}\label{subsec-reg}

In this subsection, we first construct the Lippmann-Schwinger equation of the electromagnetic scattering problem for the equivalent medium and investigate the regularity of the solutions, which is of significant use in the proof of our main theorem.

{\bl Suppose $[C^{\mathrm{T}}_{\epsilon_r}]$ and $[C^{\mathrm{T}}_{\mu_r}]$ are two $3\times3$-tensors, which can be denoted by two positive definite real symmetric matrices.}
Inspired by \eqref{dislinear}, let us consider the Lippmann-Schwinger equation with respect to $c_r^{-3}[C^{\mathrm{T}}_{\epsilon_r}]$ and $c_r^{-3}[C^{\mathrm{T}}_{\mu_r}]$ as 
\begin{align}\label{ls1}
	&U^{\mathring{\epsilon}_r}(x)-\int_{\Omega}\left(\Pi_k(x,z)c_r^{-3}[C^{\mathrm{T}}_{\epsilon_r}]U^{\mathring{\epsilon}_r}+i k \nabla\Phi_k(x,z)\times(c_r^{-3}[C^{\mathrm{T}}_{\mu_r}]V^{\mathring{\mu}_r})\right)\,dz =E^{in}(x,\theta),\notag\\
	&V^{\mathring{\mu}_r}(x)-\int_{\Omega}\left(\Pi_k(x,z)c_r^{-3}[C^{\mathrm{T}}_{\mu_r}]V^{\mathring{\mu}_r}-i k \nabla\Phi_k(x,z)\times(c_r^{-3}[C^{\mathrm{T}}_{\epsilon_r}]U^{\mathring{\epsilon}_r})\right)\,dz =H^{in}(x,\theta).
\end{align}
for $x\in\Omega$, where $c_r>1$ is the constant defined in \eqref{dis}. {\bl Indeed, we can see by Remark \ref{add-rem} from the explicit forms of $[C^{\mathrm{T}}_{\epsilon_r}]$ and $[C^{\mathrm{T}}_{\mu_r}]$ that they are two constant tensors since $[P_0^{\epsilon_r}]$ and $[P_0^{\mu_r}]$ are constant}. It is direct to verify that $(U^{\mathring{\epsilon}_r}, V^{\mathring{\mu}_r})$ satisfies the following electromagnetic scattering problem of the equivalent medium 
\begin{equation}\label{LipSch}
\begin{cases}
\mathrm{curl}U^{\mathring{\epsilon}_r}-i k \mathring{\mu}_r V^{\mathring{\mu}_r}=0,\quad \mathrm{curl} V^{\mathring{\mu}_r}+i k\mathring{\epsilon_r} U^{\mathring{\epsilon}_r}=0\quad\mbox{in} \ D,\\
\mathrm{curl} U^{\mathring{\epsilon}_r}-i k V^{\mathring{\mu}_r}=0,\quad\mathrm{curl} V^{\mathring{\mu}_r}+i k U^{\mathring{\epsilon}_r}=0\quad\mbox{in}\ \mathbb{R}^3\backslash D,\\
U^{\mathring{\epsilon}_r}=U_s^{\mathring{\epsilon}_r}+U^{\mathring{\epsilon}_r}_{in},\\
\nu\times U^{\mathring{\epsilon}_r}|_{+}=\nu\times U^{\mathring{\epsilon}_r}|_{-}, \quad \nu\times V^{\mathring{\mu}_r}|_{+}=\nu\times V^{\mathring{\mu}_r}|_{-},\quad\mbox{on} \ \partial D,\\
V^{\mathring{\mu}_r}_s\times\frac{x}{|x|}-U^{\mathring{\epsilon}_r}_s=\Oh({\frac{1}{|x|^2}}),\quad \mbox{as} \ |x|\rightarrow\infty,
\end{cases}
\end{equation}
where $\mathring{\epsilon}_r$ and $\mathring{\mu}_r$ are formulated by
\begin{equation}\label{new-eu}
\mathring{\epsilon}_r=I+c_r^{-3}[C_{\epsilon_r}^{\mathrm{T}}],\quad
\mathring{\mu}_r=I+c_r^{-3}[C_{\mu_r}^{\mathrm{T}}].
\end{equation}
{\bl It is obvious that $\mathring{\epsilon}_r$ and $\mathring{\mu}_r$ are constant and symmetric under Assumption \uppercase\expandafter{\romannumeral4} by the representations \eqref{exp-C1} and \eqref{exp-C2}. }

Recall the Green's function for the Helmhotz operator in \cite[Section 3]{AM}
\begin{equation}\notag
\Phi_k(x,z)=\frac{1}{4\pi}\frac{e^{i k |x-z|}}{|x-z|}, \quad\mbox{for} \ x\neq z,
\end{equation}
and the corresponding electromagnetic dyadic Green's function 
\begin{equation}\notag
\Pi_k(x,z)=k^2\Phi_k(x,z)I+\nabla_x\nabla_x\Phi_k(x,z)=k^2\Phi_k(x,z)I-\nabla_x\nabla_y\Phi_k(x,z)\quad\mbox{for} \ x\neq z.
\end{equation}

For any essentially bounded tensor $A$, denote the corresponding vector Newtonian operator as
\begin{equation}\label{newton}
\S_\Omega^{k,A}(W):=\int_{\Omega}\Phi_k(x,z)AW(z)\,dz\quad\mbox{for} \ W\in L^2(\Omega).
\end{equation}

Therefore, in order to investigate the regularity of $(U^{\mathring{\epsilon}_r}, V^{\mathring{\mu}_r})$, it suffices to study the regularity of $(E,H)$ solution to the following equivalent Lippmann-Schwinger  equation
\begin{equation}\label{ls2}
	\left(\begin{array}{c}
	E\\ 
	H
	\end{array} \right)-
	\left(\begin{array}{cc}
	(k^2+\nabla\mathrm{div})\S_{\Omega}^{k,\C_{\mathring{\epsilon}_r}}& i k \mathrm{curl}\S_{\Omega}^{k,\C_{\mathring{\mu}_r}}  \\ 
	-i k \mathrm{curl}\S_{\Omega}^{k, \C_{\mathring{\epsilon}_r}}& (k^2+\nabla\mathrm{div})\S_{\Omega}^{k, \C_{\mathring{\mu}_r}} 
	\end{array} \right)
	\left(\begin{array}{c}
	E\\ 
	H
	\end{array}
	\right)=
	\left(\begin{array}{c}
	E^{in}\\ 
	H^{in}
	\end{array}\right),
\end{equation}
where $\S_{\Omega}^{k, \C_{\mathring{\epsilon}_r}}, \S_{\Omega}^{k, \C_{\mathring{\mu}_r}}$ are the Newtonian operators defined by \eqref{newton} and $(E^{in}, H^{in})$ are the incident plane waves introduced in Subsection \ref{subsec1}.

We present the main regularity result as follows.

\begin{theorem}\label{thm-regu}
	Let $\Omega$ be a bounded domain {\bl with $C^{1,\alpha}$ regularity}
	for $\alpha\in (0,1)$. Consider the electromagnetic scattering problem \eqref{model-m}. {\bl There exists a positive constant $c_{\mathrm{reg}}(\alpha, \Omega)$, depending only on $\alpha$ and $\Omega$, such that if 
	\begin{equation}\label{cond}
		g(\alpha, k)c_{\mathrm{reg}}(\alpha, \Omega)c_r^{-3}c_\infty^{(\epsilon_r, \mu_r)}<1,
	\end{equation} }
	where  {\bl
	\begin{equation}\notag
	g(\alpha, k):=k^{3+\alpha}+k^3+k^2+k+1\quad\mbox{and}\quad c_\infty^{(\epsilon_r, \mu_r)}:=\max\{\lVert[P_0^{\epsilon_r}]\rVert_{L^{\infty}(\Omega)}, \lVert[P_0^{\mu_r}]\rVert_{L^{\infty}(\Omega)}\},
	\end{equation}}
 then we have
	\begin{equation}\notag
		(E,H)\in C^{0,\alpha}({\overline{\Omega}})\times C^{0,\alpha}({\overline{\Omega}}).
	\end{equation}
\end{theorem} 

\begin{proof}
	We split the proof of Theorem \ref{thm-regu} into the following three steps:
	
(\romannumeral1)
\emph{Construction of the coupled surface-volume system of integral equations}. 

Recall the Lippmann-Schwinger equation \eqref{ls2}. Using the integration by parts for the term $(k^2+\nabla\mathrm{div})\S_{\Omega}^{k,\C_{\mathring{\epsilon}_r}}$, we have
\begin{align}\label{re1}
	(k^2+\nabla\mathrm{div})\S_{\Omega}^{k,\C_{\mathring{\epsilon}_r}}(E)&=k^2\S_{\Omega}^{k,\C_{\mathring{\epsilon}_r}}(E)+\nabla\mathrm{div}\S_{\Omega}^{k,\C_{\mathring{\epsilon}_r}}(E)=k^2\S_{\Omega}^{k,\C_{\mathring{\epsilon}_r}}(E)+\nabla\int_{\Omega}\nabla_x\Phi_k(x,z)\C_{\mathring{\epsilon}_r}E(z)\,dz\notag\\
	&=k^2\S_{\Omega}^{k,\C_{\mathring{\epsilon}_r}}(E)-\nabla\int_{\Omega}\nabla_z\Phi_k(x,z)\C_{\mathring{\epsilon}_r}E(z)\,dz\notag\\
	&=k^2\S_{\Omega}^{k,\C_{\mathring{\epsilon}_r}}(E)+\nabla\int_{\Omega}\Phi_k(x,z)\mathrm{div}\C_{\mathring{\epsilon}_r}E(z)\,dz-\nabla\int_{\partial\Omega}\Phi_k(x,z)\C_{\mathring{\epsilon}_r}E(z)\cdot\vec{n}\,d\sigma(z).
\end{align}
	where $\C_{\mathring{\epsilon}_r}:=\mathring{\epsilon}_r-I$. {\bl Combining with \eqref{new-eu} as well as the explicit forms of $[C^{\mathrm{T}}_{\epsilon_r}]$ and $[C^{\mathrm{T}}_{\mu_r}]$ presented in \eqref{exp-C1} and \eqref{exp-C2}, since $\C_{\mathring{\epsilon}_r}$ is constant under Assumption \uppercase\expandafter{\romannumeral4},}
there holds
	\begin{equation}\notag
		\mathrm{div}\C_{\mathring{\epsilon}_r}E=-\frac{\C_{\mathring{\epsilon}_r}}{i k \mathring{\epsilon}_r}\mathrm{div}(\mathrm{curl}H)=0\quad\mbox{in} \ \Omega,
	\end{equation}
	which implies 
	\begin{equation}\notag
		\nabla\int_{\Omega}\Phi_k(x,z)\mathrm{div}\C_{\mathring{\epsilon}_r}E(z)\,dz=0.
	\end{equation}
	
Define the single layer operator as 
\begin{equation}\label{sing-lay}
	\mathscr{S}_{\partial\Omega}^{k, \C_{\mathring{K}}}(V_t):=\int_{\partial\Omega}\Phi_k(x,z)\C_{\mathring{K}}V_t(z)\,d\sigma(z)
\end{equation}	
	for any surface vector function $V_t$, where $\mathring{K}=\mathring{\epsilon}_r, \mathring{\mu}_r$. Then \eqref{re1} can be simplified as 
	\begin{equation}\label{re2}
		(k^2+\nabla\mathrm{div})\S_{\Omega}^{k,\C_{\mathring{\epsilon}_r}}(E)=k^2\S_{\Omega}^{k,\C_{\mathring{\epsilon}_r}}(E)-\nabla\mathscr{S}_{\partial\Omega}^{k,\C_{\mathring{\epsilon}_r}}(e*)
	\end{equation}
by denoting $e*:=E(z)\cdot\vec{n}|_{\partial\Omega}$.	
Similarly, we have
\begin{equation}\label{re3}
	(k^2+\nabla\mathrm{div})\S_{\Omega}^{k,\C_{\mathring{\mu}_r}}(H)=k^2\S_{\Omega}^{k,\C_{\mathring{\mu}_r}}(H)-\nabla\mathscr{S}_{\partial\Omega}^{k,\C_{\mathring{\mu}_r}}(h*),
\end{equation}
where $h*:=H(z)\cdot\vec{n}|_{\partial\Omega}$.

Denote
\begin{equation}\notag
	\M_\Omega^{k,\C_{\mathring{K}}}:=i k \mathrm{curl}\S_{\Omega}^{k,\C_{\mathring{K}}},\quad \gamma_n\S_{\Omega}^{k,\C_{\mathring{K}}}:=\vec{n}\cdot\S_{\Omega}^{k,\C_{\mathring{K}}}|_{\partial\Omega}, \quad \gamma_n\M_\Omega^{k,\C_{\mathring{K}}}:= \vec{n}\cdot\M_\Omega^{k,\C_{\mathring{K}}}|_{\partial\Omega},
\end{equation}
\begin{equation}\notag
	\gamma_1\mathscr{S}_{\partial\Omega}^{k,\C_{\mathring{K}}}:=\vec{n}\cdot\nabla\mathscr{S}_{\partial\Omega}^{k,\C_{\mathring{K}}},\quad e^{in}:=\vec{n}\cdot E^{in}|_{\partial\Omega},\quad h^{in}:=\vec{n}\cdot H^{in}|_{\partial\Omega},\quad\mbox{for }  \mathring{K}=\mathring{\epsilon}_r, \mathring{\mu}_r.
\end{equation}
	Then by taking \eqref{re2} and \eqref{re3} into \eqref{ls2}, we can derive the following generalized Lippmann-Schwinger system
	\begin{equation}\label{reg-m}
		\left(
		\begin{array}{cccc}
		I-k^2\S_{\Omega}^{k,\C_{\mathring{\epsilon}_r}}&-\M_\Omega^{k,\C_{\mathring{\mu}_r}}  & \nabla\mathscr{S}^{k,\C_{\mathring{\epsilon}_r}}_{\partial\Omega}  & 0 \\ 
		\M_\Omega^{k,\C_{\mathring{\epsilon}_r}}& I-k^2\S_{\Omega}^{k,\C_{\mathring{\mu}_r}} & 0 &  \nabla\mathscr{S}^{k,\C_{\mathring{\mu}_r}}_{\partial\Omega}  \\ 
		-k^2\gamma_n\S_{\Omega}^{k,\C_{\mathring{\epsilon}_r}}& -\gamma_n\M_\Omega^{k,\C_{\mathring{\mu}_r}} & I+\gamma_1\mathscr{S}_{\partial\Omega}^{k,\C_{\mathring{\epsilon}_r}} & 0  \\ 
		\gamma_n\M_\Omega^{k,\C_{\mathring{\epsilon}_r}}& -k^2\gamma_n\S_{\Omega}^{k,\C_{\mathring{\mu}_r}} & 0 & I+\gamma_1\mathscr{S}_{\partial\Omega}^{k,\C_{\mathring{\mu}_r}}
		\end{array} 
		\right)
		\left(
		\begin{array}{c}
		E\\ 
		H\\ 
		e*\\ 
		h*
		\end{array} 
		\right)=
		\left(
		\begin{array}{c}
		E^{in}\\ 
		H^{in}\\ 
		e^{in}\\ 
		h^{in}
		\end{array} 
		\right)
	\end{equation}
	where the last two lines are deduced by restricting the first two lines on $\partial\Omega$.
	
	Denote
	\begin{equation}\notag
	\mathscr{A}:=
	\left(
	\begin{array}{cccc}
	-k^2\S_{\Omega}^{k,\C_{\mathring{\epsilon}_r}}&-\M_\Omega^{k,\C_{\mathring{\mu}_r}}  & \nabla\mathscr{S}^{k,\C_{\mathring{\epsilon}_r}}_{\partial\Omega}  & 0 \\ 
	\M_\Omega^{k,\C_{\mathring{\epsilon}_r}}& -k^2\S_{\Omega}^{k,\C_{\mathring{\mu}_r}} & 0 &  \nabla\mathscr{S}^{k,\C_{\mathring{\mu}_r}}_{\partial\Omega}  \\ 
	-k^2\gamma_n\S_{\Omega}^{k,\C_{\mathring{\epsilon}_r}}& -\gamma_n\M_\Omega^{k,\C_{\mathring{\mu}_r}} & \gamma_1\mathscr{S}_{\partial\Omega}^{k,\C_{\mathring{\epsilon}_r}} & 0  \\ 
	\gamma_n\M_\Omega^{k,\C_{\mathring{\epsilon}_r}}& -k^2\gamma_n\S_{\Omega}^{k,\C_{\mathring{\mu}_r}} & 0 & \gamma_1\mathscr{S}_{\partial\Omega}^{k,\C_{\mathring{\mu}_r}}
	\end{array} 
	\right).
	\end{equation}
	Then \eqref{reg-m} can be simplified as
	\begin{equation}\notag
		\left(I+\mathscr{A}\right)\left(
		\begin{array}{c}
		E\\ 
		H\\ 
		e*\\ 
		h*
		\end{array} \right)=
		\left(
		\begin{array}{c}
		E^{in}\\ 
		H^{in}\\ 
		e^{in}\\ 
		h^{in}
		\end{array} 
		\right)
	\end{equation}
	Indeed, \eqref{reg-m} reflects the coupled surface-volume system of the integral form.
	
	(\romannumeral2)\emph{Claim that $\mathscr{A}:C^{0,\alpha}\rightarrow C^{0,\alpha}$, $\alpha\in(0,1)$}.
	
	If $(E,H)\in C^{0,\alpha}({\overline{\Omega}})\times C^{0,\alpha}({\overline{\Omega}})$, by the mapping properties of the Newtonian operator and the Sobolev embedding theorem, we know that $\S_{\Omega}^{k,\C_{\mathring{K}}}: L^p(\Omega)\rightarrow W^{2,p}(\Omega)\hookrightarrow C^{1,\alpha}({\overline{\Omega}})$ for $p>3/(1-\alpha)$, $\mathring{K}=\mathring{\epsilon}_r,\mathring{\mu}_r$. This indicates that $\S_{\Omega}^{k,\mathring{\epsilon}_r}E, \S_{\Omega}^{k, \mathring{\mu}_r}H\in C^{1,\alpha}({\overline{\Omega}})$. Since $\partial\Omega\in C^{1,\alpha}$, which implies that $\vec{n}\in C^{0,\alpha}$, we can directly have
	\begin{equation}\notag
		\gamma_n\S_{\Omega}^{k,\C_{\mathring{\epsilon}_r}}E=\vec{n}\cdot\S_{\Omega}^{k,\C_{\mathring{\epsilon}_r}}E|_{\partial\Omega}\in C^{0,\alpha}(\partial\Omega),\quad
		\gamma_n\S_{\Omega}^{k,\C_{\mathring{\mu}_r}}H=\vec{n}\cdot\S_{\Omega}^{k,\C_{\mathring{\mu}_r}}H|_{\partial\Omega}\in C^{0,\alpha}(\partial\Omega).
	\end{equation}
	
	As $\M_\Omega^{k,\C_{\mathring{\epsilon}_r}}E=i k \mathrm{curl}\S_{\Omega}^{k,\C_{\mathring{\epsilon}_r}}E$, it is obvious that $\M_\Omega^{k,\C_{\mathring{\epsilon}_r}}E\in C^{0,\alpha}({\overline{\Omega}})$ for $p>3/(1-\alpha)$. Similarly, we have $\M_\Omega^{k,\mathring{\mu}_r}H\in C^{0,\alpha}({\overline{\Omega}})$. Thus there holds
	\begin{equation}\notag
		\gamma_n\M_\Omega^{k,\C_{\mathring{\epsilon}_r}}E=\vec{n}\cdot\M_\Omega^{k,\C_{\mathring{\epsilon}_r}}E|_{\Omega}\in C^{0,\alpha}(\partial\Omega),\quad \gamma_n\M_\Omega^{k,\C_{\mathring{\mu}_r}}H=\vec{n}\cdot\M_\Omega^{k,\C_{\mathring{\mu}_r}}H|_{\Omega}\in C^{0,\alpha}(\partial\Omega).
	\end{equation}
	
	From \cite{Metrea}, or \cite[Theorem 3.3]{CK}, we know that if $(e*, h*)\in C^{0,\alpha}(\partial\Omega)\times C^{0,\alpha}(\partial\Omega)$,  then the single layer potential fulfills
	\begin{equation}\label{bd}
		\mathscr{S}_{\partial\Omega}^{k,\C_{\mathring{\epsilon}_r}}(e*)\in C^{1,\alpha}({\overline{\Omega}}),\quad
		\mathscr{S}_{\partial\Omega}^{k,\C_{\mathring{\mu}_r}}(h*)\in C^{1,\alpha}({\overline{\Omega}}),
	\end{equation}
	which clearly indicates that
	\begin{equation}\label{bd2}
		\nabla\mathscr{S}_{\partial\Omega}^{k,\C_{\mathring{\epsilon}_r}}(e*)\in C^{0,\alpha}({\overline{\Omega}}),\quad
		\nabla \mathscr{S}_{\partial\Omega}^{k,\C_{\mathring{\mu}_r}}(h*)\in C^{0,\alpha}({\overline{\Omega}}).
	\end{equation}
	Therefore restricting \eqref{bd} as well as \eqref{bd2} on $\partial\Omega$, we obtain that
	\begin{equation}\notag
	\gamma_1\mathscr{S}_{\partial\Omega}^{k,\C_{\mathring{\epsilon}_r}}(e*)=\vec{n}\cdot\nabla\mathscr{S}_{\partial\Omega}^{k,\C_{\mathring{\epsilon}\epsilon_r}}(e*)\in C^{0,\alpha}(\partial\Omega),\quad
	\gamma_1\mathscr{S}_{\partial\Omega}^{k,\C_{\mathring{\mu}_r}}(h*)=\vec{n}\cdot\nabla\mathscr{S}_{\partial\Omega}^{k,\C_{\mathring{\mu}_r}}(h*)\in C^{0,\alpha}(\partial\Omega).
	\end{equation}
	Recalling the definition of $\mathscr{A}$, it yields that $\mathscr{A}: C^{0,\alpha}\rightarrow C^{0,\alpha}$ for $\alpha\in(0,1)$.

	(\romannumeral3)\emph{The investigation of the invertibility of $I+\mathscr{A}$}.
	
	Denote
	\begin{equation}\notag
		X:=(E, H, e*, h*)^{\mathrm{T}}\in C^{0,\alpha}({\overline{\Omega}})\times C^{0,\alpha}({\overline{\Omega}})\times C^{0,\alpha}(\partial\Omega)\times C^{0,\alpha}(\partial\Omega),
	\end{equation}	
	and
	\begin{equation}\notag
		\lVert X\rVert_{C^{0,\alpha}}:=\lVert E\rVert_{C^{0,\alpha}({\overline{\Omega}})}+\lVert H\rVert_{C^{0,\alpha}({\overline{\Omega}})}+\lVert e*\rVert_{C^{0,\alpha}(\partial\Omega)}+\lVert h*\rVert_{C^{0,\alpha}(\partial\Omega)}.
	\end{equation}
	Then by direct computation
	\begin{equation}\notag
		\mathscr{A}X=
		\left(
		\begin{array}{c}
		-k^2 \S_{\Omega}^{k,\C_{\mathring{\epsilon}_r}}E-\M_\Omega^{k,\C_{\mathring{\mu}_r}}H-\nabla\mathscr{S}_{\partial\Omega}^{k,\C_{\mathring{\epsilon}_r}}e*\\ 
		\M_\Omega^{k,\C_{\mathring{\epsilon}_r}}E-k^2\S_{\Omega}^{k,\C_{\mathring{\mu}_r}}H-\nabla\mathscr{S}^{k,\C_{\mathring{\mu}_r}}_{\partial\Omega}h*\\ 
		-k^2\gamma_n\S_{\Omega}^{k,\C_{\mathring{\epsilon}_r}}E-\gamma_n\M_\Omega^{k,\C_{\mathring{\mu}_r}}H-\gamma_1\mathscr{S}_{\partial\Omega}^{k,\C_{\mathring{\epsilon}_r}}e*\\ 
		\gamma_n\M_\Omega^{k,\C_{\mathring{\epsilon}_r}}E-k^2\gamma_n\S_{\Omega}^{k,\C_{\mathring{\mu}_r}}H-\gamma_1\mathscr{S}_{\partial\Omega}^{k,\C_{\mathring{\mu}_r}}h*
		\end{array} 
		\right),
	\end{equation}
	and
	\begin{align}\label{norm}
	\lVert \mathscr{A}X\rVert_{C^{0,\alpha}}&\leq k^2 \lVert\S_{{{\Omega}}}^{k,\C_{\mathring{\epsilon}_r}}E\rVert_{C^{0,\alpha}({{{\overline{\Omega}}}})}+\lVert \M_\Omega^{k,\C_{\mathring{\mu}_r}}H\rVert_{C^{0,\alpha}({\overline{\Omega}})}+\lVert \nabla\mathscr{S}_{\partial\Omega}^{k,\C_{\mathring{\epsilon}_r}}e*\rVert_{C^{0,\alpha}({\overline{\Omega}})}\notag\\
	&+\lVert \M_\Omega^{k,\C_{\mathring{\epsilon}_r}}E\rVert_{C^{0,\alpha}({\overline{\Omega}})}+k^2\lVert\S_{\Omega}^{k,\C_{\mathring{\mu}_r}}H\rVert_{C^{0,\alpha}({\overline{\Omega}})}+\lVert \nabla\mathscr{S}_{\partial\Omega}^{k,\C_{\mathring{\mu}_r}}h*\rVert_{C^{0,\alpha}({\overline{\Omega}})}\notag\\
	&+ k^2\lVert \gamma_n\S_{\Omega}^{k,\C_{\mathring{\epsilon}_r}}E\rVert_{C^{0,\alpha}(\partial\Omega)}+\lVert\gamma_n\M_\Omega^{k,\C_{\mathring{\mu}_r}}H\rVert_{C^{0,\alpha}(\partial\Omega)}+\lVert\gamma_1\mathscr{S}_{\partial\Omega}^{k,\C_{\mathring{\epsilon}_r}}e*\rVert_{C^{0,\alpha}(\partial\Omega)}\notag\\
	&+\lVert\gamma_n\M_\Omega^{k,\C_{\mathring{\epsilon}_r}}E\rVert_{C^{0,\alpha}(\partial\Omega)}+ k^2\lVert \gamma_n\S_{\Omega}^{k,\C_{\mathring{\mu}_r}}H\rVert_{C^{0,\alpha}(\partial\Omega)}+\lVert\gamma_1\mathscr{S}_{\partial\Omega}^{k,\C_{\mathring{\mu}_r}}h*\rVert_{C^{0,\alpha}(\partial\Omega)}
	\end{align}
	
	First, let us consider $\lVert\S_{\Omega}^{k,\C_{\mathring{\epsilon}_r}}E\rVert_{C^{0,\alpha}({\overline{\Omega}})}$.
	\begin{align}\notag
		\S_{\Omega}^{k,\C_{\mathring{\epsilon}_r}}E=\int_{\Omega}\Phi_k(x,z)\C_{\mathring{\epsilon}_r}E(z)\,dz&=\int_{\Omega}\Phi_0(x,z)\C_{\mathring{\epsilon}_r}E(z)\,dz+\int_{\Omega}(\Phi_k-\Phi_0)(x,z)\C_{\mathring{\epsilon}_r}E(z)\,dz\notag\\
		&:=\S_{\Omega}^{0,\C_{\mathring{\epsilon}_r}}E+\S_{\Omega}^{R,\C_{\mathring{\epsilon}_r}}E.\notag
	\end{align}
	{\bl Denote
	\begin{equation}\label{add2}
		c_{P_0}^{\epsilon_r}:=\lVert[P_0^{\epsilon_r}]\rVert_{L^{\infty}(\Omega)},\quad
		c_{P_0}^{\mu_r}:=\lVert[P_0^{\mu_r}]\rVert_{L^{\infty}(\Omega)},\quad\mbox{and}\quad
		c_{\infty}^{(\epsilon_r,\mu_r)}=:\max\{c_{P_0}^{\epsilon_r}, c_{P_0}^{\mu_r}\}.
\end{equation}
By direct computations, it is easy to see that 
\begin{equation}\notag
	\lVert\C_{\mathring{\epsilon}_r}\rVert_{L^\infty({{\Omega}})}=\lVert\mathring{\epsilon_r}-I\rVert_{L^{\infty}(\Omega)}\leq 2c_r^{-3}\lVert [P_0^{\epsilon_r}]\rVert_{L^{\infty}(\Omega)}\leq 2c_r^{-3}c_{P_0}^{\epsilon_r}.
\end{equation}}
From the regularity discussions in part (ii) 
	for $\S_{\Omega}^{0,\C_{\mathring{\epsilon_r}}}E (k=0)$, we obtain {\bl
	\begin{equation}\label{ad2}
		\lVert\S_{\Omega}^{0,\C_{\mathring{\epsilon}_r}}E\rVert_{C^{1,\alpha}({\overline{\Omega}})}\leq c_{s,1}(\alpha,\Omega)c_r^{-3}c_{P_0}^{\epsilon_r}\lVert E\rVert_{C^{0,\alpha}({\overline{\Omega}})},
	\end{equation} }
	where $c_{s,1}(\alpha,\Omega)$ is the positive constant appearing in the mapping properties the Newtonian operator $\S_{\Omega}^{0,\C_{\mathring{\epsilon}_r}}$ and the Sobolev embedding theorem, which depends on $\alpha\mbox{ and } \Omega$. 
	Since
	\begin{equation}\notag
		(\Phi_k-\Phi_0)(x,z)=\frac{i k }{4\pi}\int_{0}^{1}e^{-i k t|x-z|}\,dt,
	\end{equation}
	{\bl 
	and 
	\begin{equation}\label{phi}
		|\partial_x^\beta(\Phi_k-\Phi_0)(x, z)|=\Oh(k^{1+\beta}), \quad\mbox{for}\quad\beta\geq0,
	\end{equation}}
	it is easy to see that {\bl 
	\begin{equation}\label{ad1}
		\lVert\S_{\Omega}^{R,\C_{\mathring{\epsilon}_r}}E\rVert_{C^{1,\alpha}({\overline{\Omega}})}
		\leq k(k^\alpha+1)|\Omega|c_r^{-3}c_{P_0}^{\epsilon_r}\lVert E\rVert_{C^{0,\alpha}({\overline{\Omega}})},
	\end{equation}  }
	reflecting the smoothness of $\Phi_k-\Phi_0$, where $|\Omega|$ is the volume of $\Omega$. Thus {\bl
		by denoting
		\begin{equation}\notag
			c_{\S}^{(1)}:=\max\{c_{s,1}(\alpha,\Omega), |\Omega|\},
		\end{equation} 
		we can obtain that
	\begin{equation}\label{re-s1}
		\lVert\S_{{{\Omega}}}^{k,\C_{\mathring{\epsilon}_r}}E\rVert_{C^{1,\alpha}({\overline{\Omega}})}
		\leq(k^{1+\alpha}+k+1)c_{\S}^{(1)}(\alpha, \Omega)c_r^{-3}c_{P_0}^{\epsilon_r}\lVert E\rVert_{C^{0,\alpha}({\overline{\Omega}})}.
	\end{equation}  }
	Similarly, we have {\bl
	\begin{equation}\notag
		\lVert\S_{\Omega}^{k,\C_{\mathring{\mu}_r}}H\rVert_{C^{1,\alpha}({\overline{\Omega}})}\leq (k^{1+\alpha}+k+1)c_{\S}^{(1)}(\alpha, \Omega)c_r^{-3}c_{P_0}^{\mu_r}\lVert H\rVert_{C^{0,\alpha}({\overline{\Omega}})}.
	\end{equation} }
	
	Next, we study $\lVert\M_\Omega^{k,\C_{\mathring{\epsilon}_r}}E\rVert_{C^{0,\alpha}(\overline{\Omega})}$. By the definition for the operator $\M_\Omega^{k,\C_{\mathring{\epsilon}_r}}$, we obtain {\bl
	\begin{equation}\notag
		\M_\Omega^{k,\C_{\mathring{\epsilon}_r}}E=i k \mathrm{curl}\S_{\Omega}^{k,\C_{\mathring{\epsilon}_r}}E:=i k \mathrm{curl}\S_{\Omega}^{0,\C_{\mathring{\epsilon}_r}}E+i k \mathrm{curl}\S_{\Omega}^{R,\C_{\mathring{\epsilon}_r}}E,
	\end{equation} 
	where
	\begin{equation}\notag
		\mathrm{curl}\S_{\Omega}^{0,\C_{\mathring{\epsilon}_r}}E=\int_{\Omega}\nabla_x\Phi_0(x,z)\times \C_{\mathring{\epsilon}_r}E(z)\,dz,\quad
		\mathrm{curl}\S_{\Omega}^{R,\C_{\mathring{\epsilon}_r}}E=\int_{\Omega}\nabla_x(\Phi_k-\Phi_0)(x,z)\times \C_{\mathring{\epsilon}_r}E(z)\,dz.
	\end{equation}
	Utilizing \eqref{phi}, similar to \eqref{ad2} and \eqref{ad1}, we have
	\begin{equation}\notag
		\lVert\M_\Omega^{0,\C_{\mathring{\epsilon}_r}}E\rVert_{C^{0,\alpha}({\overline{\Omega}})}\leq k c_{s,1}(\alpha, \Omega)c_r^{-3}c_{P_0}^{\epsilon_r}\lVert E\rVert_{C^{0,\alpha}({\overline{\Omega}})},
 	\end{equation}
and
\begin{equation}\notag
	\lVert\M_\Omega^{R,\C_{\mathring{\epsilon}_r}}E\rVert_{C^{0,\alpha}({\overline{\Omega}})}\leq k^3 |\Omega|c_r^{-3}c_{P_0}^{\epsilon_r}\lVert E\rVert_{C^{0,\alpha}({\overline{\Omega}})}.
\end{equation}
Therefore, we can deduce that 
	\begin{align}\label{re-m1}
		&\lVert\M_\Omega^{k,\C_{\mathring{\epsilon}_r}}E\rVert_{C^{0,\alpha}({\overline{\Omega}})}\leq k(k^2+1)c_\S^{(1)}(\alpha, \Omega)c_r^{-3}c_{P_0}^{\epsilon_r}\lVert E\rVert_{C^{0,\alpha}({\overline{\Omega}})},\notag\\
		&\lVert\M_\Omega^{k,\C_{\mathring{\epsilon}_r}}H\rVert_{C^{0,\alpha}({\overline{\Omega}})}\leq k(k^2+1)c_\S^{(1)}(\alpha, \Omega)c_r^{-3}c_{P_0}^{\mu_r}\lVert H\rVert_{C^{0,\alpha}({\overline{\Omega}})}.
	\end{align} }
	
	Now we consider the estimate for the single layer potential $\mathscr{S}_{\partial\Omega}^{k,\C_{\mathring{\epsilon}_r}}e*$. {\bl Similar to $\S_{\Omega}^{k, C_{\mathring{\epsilon_r}}}E$, we can write
	\begin{equation}\notag
		\nabla \mathscr{S}_{\partial\Omega}^{k,\C_{\mathring{\epsilon}_r}}e*:=\mathscr{S}_{\partial\Omega}^{0,\C_{\mathring{\epsilon}_r}}e*+\mathscr{S}_{\partial\Omega}^{R,\C_{\mathring{\epsilon}_r}}e*,
    \end{equation}
    where 
    \begin{equation}\notag
    	\mathscr{S}_{\partial\Omega}^{0,\C_{\mathring{\epsilon}_r}}e*=\nabla\int_{\partial\Omega}\Phi_0(x, z)\C_{\mathring{\epsilon_r}}e*(z)\,d\sigma(z)\quad\mbox{and}\quad
    	\mathscr{S}_{\partial\Omega}^{R,\C_{\mathring{\epsilon}_r}}e*=\nabla\int_{\partial\Omega}(\Phi_k-\Phi_0)(x, z)\C_{\mathring{\epsilon_r}}e*(z)\,d\sigma(z).
    \end{equation}
    Then combining with \eqref{phi}, }from \cite{Metrea} and \cite[Theorem 3.3]{CK}, it is direct to know that  {\bl 
	\begin{equation}\label{re-s}
	\lVert	\nabla \mathscr{S}_{\partial\Omega}^{k,\C_{\mathring{\epsilon}_r}}e*\rVert_{C^{0,\alpha}({\overline{\Omega}})}\leq (k^2+1)c^{(1)}_{\mathscr{S}}(\alpha,\Omega)c_r^{-3}c_{P_0}^{\epsilon_r}\lVert e*\rVert_{C^{0,\alpha}(\partial\Omega)},
	\end{equation} }
	where $c^{(1)}_{\mathscr{S}}$ is the positive constant depending on $\alpha$ and $\Omega$. Similarly, we have {\bl 
	\begin{equation}\notag
		\lVert	\nabla \mathscr{S}_{\partial\Omega}^{k,\C_{\mathring{\mu}_r}}h*\rVert_{C^{0,\alpha}({\overline{\Omega}})}\leq (k^2+1)c^{(1)}_{\mathscr{S}}(\alpha,\Omega)c_r^{-3}c_{P_0}^{\mu_r}\lVert h*\rVert_{C^{0,\alpha}(\partial\Omega)}.
	\end{equation} }
	
	Recall \eqref{re-s1}. By restricting the estimate on the boundary, we can get
	\begin{equation}\label{re-m}
		\lVert \gamma_n\S_{\Omega}^{k,\C_{\mathring{\epsilon}_r}}E\rVert_{C^{0,\alpha}(\partial\Omega)}=\lVert\vec{n}\cdot\S_{\Omega}^{k,\C_{\mathring{\epsilon}_r}}E|_{\partial\Omega}\rVert_{C^{0,\alpha}(\partial\Omega)}\leq (k^{1+\alpha}+k+1) c_{\S}^{(2)}(\alpha, \Omega) c_r^{-3}c_{P_0}^{\epsilon_r}\lVert E\rVert_{C^{0,\alpha}({\overline{\Omega}})},
	\end{equation}
	where $c^{(2)}_\S$ is the positive constant depending on $\alpha\mbox{ and }\Omega$.
	Similarly, there holds
	\begin{equation}\notag
		\lVert \gamma_n\S_{\Omega}^{k,\C_{\mathring{\mu}_r}}H\rVert_{C^{0,\alpha}(\partial\Omega)}=\lVert\vec{n}\cdot\S_{\Omega}^{k,\C_{\mathring{\mu}_r}}H|_{\partial\Omega}\rVert_{C^{0,\alpha}(\partial\Omega)}\leq (k^{1+\alpha}+k+1)c_{\S}^{(2)}(\alpha, \Omega) c_r^{-3}c_{P_0}^{\mu_r}\lVert H\rVert_{C^{0,\alpha}({\overline{\Omega}})}.
	\end{equation}
	Besides, by restricting the estimate for \eqref{re-m1} on the boundary, we can obtain that
	\begin{equation}\label{re-d}
		\lVert\gamma_n\M_\Omega^{k,\C_{\mathring{\epsilon}_r}}E\rVert_{C^{0,\alpha}(\partial\Omega)}=\lVert\vec{n}\cdot\M_\Omega^{k,\C_{\mathring{\epsilon}_r}}E|_{\partial\Omega}\rVert_{C^{0,\alpha}(\partial\Omega)}
		\leq k(k^2+1) c^{(2)}_\S(\alpha, \Omega)c_r^{-3}c_{P_0}^{\epsilon_r}\lVert E\rVert_{C^{0,\alpha}({\overline{\Omega}})},
	\end{equation}
	and {\bl
	\begin{equation}\notag
		\lVert\gamma_n\M_\Omega^{k,\C_{\mathring{\mu}_r}}H\rVert_{C^{0,\alpha}(\partial\Omega)}=\lVert\vec{n}\cdot\M_\Omega^{k,\C_{\mathring{\mu}_r}}H|_{\partial\Omega}\rVert_{C^{0,\alpha}(\partial\Omega)}\leq k(k^2+1) c^{(2)}_\S(\alpha, \Omega)c_r^{-3}c_{P_0}^{\mu_r}\lVert H\rVert_{C^{0,\alpha}({\overline{\Omega}})}.
	\end{equation} }
	 
	Furthermore, we can derive from \eqref{re-s} that {\bl 
	\begin{equation}\label{re-b}
		\lVert\gamma_1\mathscr{S}_{\partial\Omega}^{k,\C_{\mathring{\epsilon}_r}}e*\rVert_{C^{0,\alpha}(\partial\Omega)}=\lVert\vec{n}\cdot\nabla\mathscr{S}_{\partial\Omega}^{k,\C_{\mathring{\epsilon}_r}}e*\rVert_{C^{0,\alpha}(\partial\Omega)}\leq (k^2+1)c^{(2)}_{\mathscr{S}}(\alpha, \Omega)c_r^{-3}c_{P_0}^{\epsilon_r}\lVert e*\rVert_{C^{0,\alpha}(\partial\Omega)},
	\end{equation}
	and 
	\begin{equation}\notag
	\lVert\gamma_1\mathscr{S}_{\partial\Omega}^{k,\C_{\mathring{\mu}_r}}h*\rVert_{C^{0,\alpha}(\partial\Omega)}=\lVert\vec{n}\cdot\nabla\mathscr{S}_{\partial\Omega}^{k,\C_{\mathring{\mu}_r}}h*\rVert_{C^{0,\alpha}(\partial\Omega)}\leq (k^2+1) c^{(2)}_{\mathscr{S}}(\alpha, \Omega)c_r^{-3}c_{P_0}^{\mu_r}\lVert h*\rVert_{C^{0,\alpha}(\partial\Omega)},	
	\end{equation} }
	where $c^{(2)}_{\mathscr{S}}(\alpha, \Omega)$ is the positive constant depending on $\alpha$ and $\Omega$.
	
	Combining with \eqref{add2}, \eqref{re-s1}, \eqref{re-m1}, \eqref{re-s}, \eqref{re-m}, \eqref{re-d} and \eqref{re-b}, \eqref{norm} can be further written as 
	\begin{align}\notag
		\lVert\mathscr{A}X\rVert_{C^{0,\alpha}}&\leq k(k^{2+\alpha}+k^2+k+1)(c^{(1)}_\S(\alpha, \Omega)+c^{(2)}_\S(\alpha, \Omega))c_r^{-3}c_\infty^{(\epsilon_r, \mu_r)}(\lVert E\rVert_{C^{0,\alpha}(\overline{\Omega})}+\lVert H\rVert_{C^{0,\alpha}(\overline{\Omega})})\notag\\
		&+(k^2+1)(c^{(1)}_{\mathscr{S}}(\alpha, \Omega)+c^{(2)}_{\mathscr{S}}(\alpha, \Omega))c_r^{-3}c_\infty^{(\epsilon_r, \mu_r)}(\lVert e*\rVert_{C^{0,\alpha}(\partial\Omega)}+\lVert h*\rVert_{C^{0,\alpha}(\partial\Omega)}).\notag
	\end{align}
	Denote
	\begin{equation}\label{c-reg}
		c_{\mathrm{reg}}(\alpha, \Omega):=\max\{c^{(1)}_\S(\alpha, \Omega), c^{(2)}_\S(\alpha, \Omega), c^{(1)}_{\mathscr{S}}(\alpha, \Omega), c^{(2)}_{\mathscr{S}}(\alpha, \Omega)\},
	\end{equation}
	and 
	\begin{equation}\notag
		g(\alpha, k):=k^{3+\alpha}+k^3+k^2+k+1,
	\end{equation}
	then  {\bl 
	\begin{align}\notag
		\lVert\mathscr{A}X\rVert_{C^{0,\alpha}}&\leq g(\alpha, k)c_{\mathrm{reg}}(\alpha, \Omega)c_r^{-3}c_\infty^{(\epsilon_r, \mu_r)}(\lVert E\rVert_{C^{0,\alpha}(\overline{\Omega})}+\lVert H\rVert_{C^{0,\alpha}(\overline{\Omega})}+\lVert e*\rVert_{C^{0,\alpha}(\partial\Omega)}+\lVert h*\rVert_{C^{0,\alpha}(\partial\Omega)})\notag\\
		&\leq g(\alpha, k)c_{\mathrm{reg}}(\alpha, \Omega))c_r^{-3}c_\infty^{(\epsilon_r, \mu_r)}\lVert X\rVert_{C^{0,\alpha}}.\notag
	\end{align} }
	From the condition \eqref{cond}, it is direct to verify that $I+\mathscr{A}$ is invertible by the corresponding convergent Born series.
	
	Finally, combining with (\romannumeral1)--(\romannumeral3), by the existence and the uniqueness of the solution to \eqref{ls2}, we can derive that $(E, H)\in C^{0,\alpha}(\overline{\Omega})\times C^{0,\alpha}(\overline{\Omega})$ by the fact that $(E^{in}, H^{in})$ are smooth enough.
\end{proof}

With the help the regularity result presented above, we shall prove our main result in the next section.

\section{The equivalent medium and the associated error estimate for the far-field}\label{sec-compare}

In this section, we focus on the investigation of the equivalent medium as $a\rightarrow0$ for the electromagnetic scattering problem \eqref{model-m}. In particular, 
we evaluate the error estimate between the discrete form of the Foldy-Lax approximation and the corresponding asymptotic form with respect to $a$ and $c_r$.

Before presenting the proof of our main theorem, we first introduce some important notations and definitions. 
We set $S_m\subset\Omega_m$ to be the largest ball contained in $\Omega_m$ for every $m=1,\cdots, M$. 

\begin{definition}\label{def2}
	Let $[C_{\epsilon_r}^{\mathrm{T}}]$ and $[C_{\mu_r}^\mathrm{T}]$ be two constant $3\times3$-tensors, denoted by two positive definite real symmetric matrices. 
	Define the bounded linear operators $K^{\epsilon_r}$ and $K^{\mu_r}$ for  constant vector function $f_{\mathrm{const}}$ as follows
	\begin{align}\label{KB}
	K^{\epsilon_r}(f_{\mathrm{const}}):&=\frac{1}{|S_m|}\int_{\Omega_m}\int_{S_m}\Pi_0(x, z)c_r^{-3}[C_{\epsilon_r}^\mathrm{T}]f_{\mathrm{const}}\,dxdz,\notag\\
	K^{\mu_r}(f_{\mathrm{const}}):&=\frac{1}{|S_m|}\int_{\Omega_m}\int_{S_m}\Pi_0(x, z)c_r^{-3}[C_{\mu_r}^\mathrm{T}]f_{\mathrm{const}}\,dxdz,
	\end{align}
	which satisfy the following two conditions
	\begin{enumerate}[a)]
		\item $[C_{\epsilon_r}^\mathrm{T}](I-K^{\epsilon_r})^{-1}=[P_0^{\epsilon_r}]$,\quad $[C_{\mu_r}^\mathrm{T}](I-K^{\mu_r})^{-1}=[P_0^{\mu_r}]$;
		
		\item $\lVert K^{\epsilon_r}\rVert_{L^{\infty}}<1$, $\lVert K^{\mu_r}\rVert_{L^{\infty}}<1$,
	\end{enumerate}
	where $\Pi_0(x,z)=\nabla\nabla\Phi_0(x,z)$.
\end{definition}

\begin{remark}\label{rem-kb}
We have the following observations for the construction of $K^{\epsilon_r}$ and $K^{\mu_r}$.
	\begin{enumerate}
	\item Let $(e_i)_{i=1}^3$ be the canonical orthonormal basis in $\mathbb{R}^3$. Then, the bounded linear operator $K^{\epsilon_r}$ can be formulated by the matrix with the components $(K^{\epsilon_r}e_i\cdot e_j)_{i,j=1}^3$. We have a similar representation $K^{\mu_r}$.

\item The notation 
	\begin{equation}\label{baseK}
	K^0(f):=\frac{1}{|S_m|}\int_{\Omega_m}\int_{S_m}\Pi_0(x,z)f\,dx\,dz,
	\end{equation}
	should be understood in the following sense:
	\begin{equation}\notag
	K^0(f)=\frac{1}{|S_m|}\int_{\Omega_m}\nabla\int_{S_m}\nabla\Phi_0(x,z)f\,dx\,dz.
	\end{equation}
	To simplify the notations, throughout this paper, we follow the representation \eqref{baseK}. 
	
	\item Moreover, if $\Omega_m:=\delta \Omega +z_m$, it is straightforward to verify that $K^0$ can also be expressed by the following equivalent form
	\begin{equation}\notag
	K^0(f_{\mathrm{const}})=\frac{1}{\vert S\vert}\int_{\Omega}\int_{S}\Pi_0(x,z)f_{\mathrm{const}}\,dx\,dz,
	\end{equation}
	where $\Omega$ is the cube of size 2 and $S$ is the largest ball contained in $\Omega$ with radius 1. Thus $K^0$ is independent of $\delta$.
\end{enumerate}
\end{remark}


\begin{remark}\label{add-rem}
	From condition a), we see that
	\begin{equation}\notag
	[C^{\mathrm{T}}_{\epsilon_r}]=[P_0^{\epsilon_r}](I-c_r^{-3}K^0[C^{\mathrm{T}}_{\epsilon_r}]).
	\end{equation}
	By rearranging terms, we have 
	\begin{equation}\notag
	(I-c_r^{-3}[P_0^{\epsilon_r}]K^0)[C^{\mathrm{T}}_{\epsilon_r}]=[P_0^{\epsilon_r}].
	\end{equation}
	Thus $[C^{\mathrm{T}}_{\epsilon_r}]$ can be written more explicitly as 
	\begin{equation}\label{exp-C1}
	[C^{\mathrm{T}}_{\epsilon_r}]=(I-c_r^{-3}[P_0^{\epsilon_r}]K^0)^{-1}[P_0^{\epsilon_r}].
	\end{equation}
	Similarly, we have 
	\begin{equation}\label{exp-C2}
	[C^{\mathrm{T}}_{\mu_r}]=(I-c_r^{-3}[P_0^{\mu_r}]K^0)^{-1}[P_0^{\mu_r}].
	\end{equation}
	Furthermore, based on Remark \ref{rem-kb}, it is clear from \eqref{KB}, \eqref{exp-C1} and \eqref{exp-C2} that condition b) can be fulfilled if $c_r$ is large enough, which implies that $I-K^{\epsilon_r}$ and $I-K^{\mu_r}$ are invertible.
\end{remark}

Now, recall the Lippmann-Schwinger equation \eqref{ls1}. It is easy to verify that the solution $(U^{\mathring{\epsilon}_r}, V^{\mathring{\mu}_r})$ fulfills the electromagnetic scattering problem \eqref{LipSch}. The far-field corresponding to $(U^{\mathring{\epsilon}_r}, V^{\mathring{\mu}_r})$ of \eqref{ls1} possesses the form
\begin{equation}\label{lim-far}
	E^{\infty}_{\mathrm{eff}}(\hat{x})=\int_{\Omega}\left(\frac{k^2}{4\pi}e^{-i k\hat{x}\cdot z}\hat{x}\times (c_r^{-3}[C^{\mathrm{T}}_{\epsilon_r}]U^{\mathring{\epsilon}_r}(z)\times\hat{x})
	+\frac{i k}{4\pi}e^{-i k \hat{x}\cdot z}\hat{x}\times c_r^{-3}[C^{\mathrm{T}}_{\mu_r}]V^{\mathring{\mu}_r}(z)\right)\,dz.
\end{equation}
The discussions on \eqref{lim-far} will be given later in the proof of Theorem \ref{main-1}.

{\bl Rcall the notations \eqref{add2}. We further denote that}
\begin{equation}\notag
c_0^{\epsilon_r}:=\lVert[C^{\mathrm{T}}_{\epsilon_r}]\rVert_{L^\infty(\Omega)},\quad
c_0^{\mu_r}:=\lVert[C^{\mathrm{T}}_{\mu_r}]\rVert_{L^\infty(\Omega)},
\end{equation}
\begin{equation}\label{notation}
	 c_0^{\epsilon_r,-1}:=\lVert[P_0^{\epsilon_r}]^{-1}\rVert_{L^\infty(\Omega)},\quad c_0^{\mu_r,-1}:=\lVert[P_0^{\mu_r}]^{-1}\rVert_{L^\infty(\Omega)}.
\end{equation}

The following lemma which plays a very important role, provides a method of counting the number of the particles especially evaluating the distance between any two points $z_m\in D_m$ and $z_j\in D_j$; see \cite[Section 3.3]{ACP} for more details.

\begin{lemma}\cite[Section 3.3]{ACP}\label{conting1}
	For $k>0$, we have the following formulas:
	\begin{enumerate}[1.]
		\item If $k<3$, then
		\begin{equation}\notag
		\sum_{j=1\atop j\neq m}^M|z_j-z_m|^{-k}=\Oh(\delta^{-3}).
		\end{equation}
		
		\item If $k=3$, then
		\begin{equation}\notag
		\sum_{j=1\atop j\neq m}^M|z_j-z_m|^{-k}=\Oh(\delta^{-3}(1+\ln \delta)).
		\end{equation}
		
		\item  If $k>3$, then
		\begin{equation}\notag
		\sum_{j=1\atop j\neq m}^M|z_j-z_m|^{-k}=\Oh(\delta^{-k}).	
		\end{equation}
	\end{enumerate}	
\end{lemma}

In the subsequent discussions, for the simplification of the notation, we denote
\begin{equation}\notag
|z_j-z_m|^{-k}:=\delta_{mj}^{-k}\quad\mbox{for }z_m\in D_m, z_j\in D_j.
\end{equation}

Based on Lemma \ref{conting1}, next, we present two essential propositions to preliminarily investigate the relation between the discrete linear algebraic system and the Lippmann-Schwinger equation. To clear the structure of our paper, we postpone the proof of Proposition \ref{lem-eq} later in Section \ref{sec-prop}.

\begin{proposition}\label{lem-eq}
	Let $[C^{\mathrm{T}}_{\epsilon_r}]$ and $[C^{\mathrm{T}}_{\mu_r}]$ be as introduced in Definition \ref{def2} and $K^{\epsilon_r}$ and $K^{\mu_r}$ are the associated bounded linear operators defined by \eqref{KB}. Under the assumptions $(\uppercase\expandafter{\romannumeral1}, \uppercase\expandafter{\romannumeral2}, \uppercase\expandafter{\romannumeral3}, \uppercase\expandafter{\romannumeral4})$, we have the following equations
	\begin{align}\label{main-est1}
		&((I-K^{\epsilon_r})\frac{1}{|S_m|}\int_{S_m}U^{\mathring{\epsilon}_r}(x)\,dx-U_m)-a^3\sum_{j=1\atop j\neq m}^M\left[\right.\Pi_k(z_m,z_j)[P_0^{\epsilon_r}]((I-K^{\epsilon_r})\frac{1}{|S_j|}\int_{S_j}U^{\mathring{\epsilon}_r}(y)\,dy-U_j)\notag\\
		&+i k \nabla\Phi_k(z_m,z_j)\times ([P_0^{\mu_r}]((I-K^{\mu_r})\frac{1}{|S_j|}\int_{S_j}V^{\mathring{\mu}_r}(y)\,dy-V_j))
		\left.\right]\notag\\
		&=\Oh(c_ra)+A+I_B+B^R_{1,m}+B_{2,m}+C,
	\end{align}
	where $S_m$ is the largest ball contained in $\Omega_m$, $U_m$ is the solution to \eqref{dislinear} and 
	\begin{align}
		&A=\frac{1}{|S_m|}\sum_{j=1\atop j\neq m}^M\int_{\Omega_j}\int_{S_m}\Pi_k(x,z)c_r^{-3}[C^\mathrm{T}_{\epsilon_r}]U^{\mathring{\epsilon}_r}(z)\,dx\,dz-a^3\sum_{j=1\atop j\neq m}^M\Pi_k(z_m,z_j)[C^\mathrm{T}_{\epsilon_r}]\frac{1}{|S_j|}\int_{S_j}U^{\mathring{\epsilon}_r}(y)\,dy\notag\\
		&+i k \frac{1}{|S_m|}\sum_{j=1\atop j\neq m}^M\int_{\Omega_j}\int_{S_m}\nabla\Phi_k(x,z)\times(c_r^{-3}[C^\mathrm{T}_{\mu_r}]V^{\mathring{\mu_r}})\,dx\,dz-a^3\sum_{j=1\atop j\neq m}^M i k \nabla\Phi_k(z_m,z_j)\times([C^\mathrm{T}_{\mu_r}]\frac{1}{|S_j|}\int_{S_j}V^{\mathring{\mu}_r}(y)\,dy),\notag \\
		&I_B=\frac{1}{|S_m|}\int_{\Omega_m}\int_{S_m}\Pi_0(x,z)c_r^{-3}[C^\mathrm{T}_{\epsilon_r}](U^{\mathring{\epsilon}_r}(z)-\frac{1}{|S_m|}\int_{S_m}U^{\mathring{\epsilon}_r}(y)\,dy)\,dx\,dz,\notag\\
		&B^R_{1,m}=\frac{1}{|S_m|}\int_{\Omega_m}\int_{S_m}(\Pi_k-\Pi_0)(x,z)c_r^{-3}[C^\mathrm{T}_{\epsilon_r}]U^{\mathring{\epsilon}_r}(z)\,dx\,dz\notag\\
		&B_{2,m}=\frac{1}{|S_m|}\int_{\Omega_m}\int_{S_m}ik\nabla\Phi_k(x,z)\times(c_r^{-3}[C^\mathrm{T}_{\mu_r}]V^{\mathring{\mu}_r}(z))\,dx\,dz\notag\\
		&C=\frac{1}{|S_m|}\int_L\int_{S_m}(\Pi_k(x,z)c_r^{-3}[C^\mathrm{T}_{\epsilon_r}]U^{\mathring{\epsilon}_r}(z)+ik\nabla\Phi_k(x,z)\times(c_r^{-3}[C^\mathrm{T}_{\mu_r}]V^{\mathring{\mu}_r}(z)))\,dx\,dz,\quad L=\Omega\backslash\cup_{m=1}^M\Omega_m.\notag
	\end{align}
	Furthermore, there holds the following estimates according to $A, I_B, B^R_{1,m}, B_2$ and $C$ with respect to $c_r$ and $a$ for sufficiently small $a$:
\begin{align}\label{group-est}
	&\sum_{m=1}^M|A|^2=\Oh({c_0^{\epsilon_r}}^2c_r^{-9}a^{-3}),\quad
	\sum_{m=1}^M|I_B|^2=\Oh( {c_0^{\epsilon_r}}^2c_\alpha^2c_r^{2\alpha-9}a^{2\alpha-3}),\quad\alpha\in(0,1)\notag\\
	&\sum_{m=1}^M|B_{1,m}^R|^2=\Oh({k^4c_0^{\epsilon_r}}^2\lVert U^{\mathring{\epsilon}_r}\rVert^2_{L^\infty(\Omega)}c_r^{-5}a),
	\sum_{m=1}^M|B_{2,m}|^2=\Oh(k^2{c_0^{\mu_r}}^2\lVert V^{\mathring{\mu}_r}\rVert^2_{L^\infty(\Omega)}(12+9k^2c_r^2a^2)c_r^{-7}a^{-1}),\notag\\
	&\sum_{m=1}^M|C|^2=\Oh\left(\right.({c_0^{\epsilon_r}}^2\lVert U^{\mathring{\epsilon}_r}\rVert^2_{L^\infty(\Omega)}(32c_{\mathrm{aver}}+k^4(36c_0^4\pi^2c_r^2a^2+k^2))\notag\\
	&\qquad+\pi^2k^2c_0^2{c_0^{\mu_r}}^2\lVert V^{\mathring{\mu}_r}\rVert^2_{L^\infty(\Omega)}(4+k^2c_0^2c_r^2a^2))c_r^{-7}a^{-1}\left.\right)
\end{align}
where $c_\alpha$, $c_{\mathrm{aver}}$ and $c_0$ are positive constants.
\end{proposition}

Following the similar arguments in Proposition \ref{lem-eq}, we also obtain the comparing equation and corresponding related estimates for $V^{\mathring{\mu}_r}$ and $V_m$. We state them in the following corollary.

\begin{corollary}\label{cor-v}
	Let $[C^\mathrm{T}_{\epsilon_r}]$, $[C^\mathrm{T}_{\mu_r}]$, $K^{\epsilon_r}$ and $K^{\mu_r}$ be the same as Proposition \ref{lem-eq}. Under the assumptions $(\uppercase\expandafter{\romannumeral1}, \uppercase\expandafter{\romannumeral2}, \uppercase\expandafter{\romannumeral3}, \uppercase\expandafter{\romannumeral4})$, we have the following equation
	\begin{align}
	&((I-K^{\mu_r})\frac{1}{|S_m|}\int_{S_m}V^{\mathring{\mu}_r}(x)\,dx-V_m)-a^3\sum_{j=1\atop j\neq m}^M\left[\right.\Pi_k(z_m,z_j)[P_0^{\mu_r}]((I-K^{\mu_r})\frac{1}{|S_j|}\int_{S_j}V^{\mathring{\mu}_r}(y)\,dy-V_j)\notag\\
	&-i k \nabla\Phi_k(z_m,z_j)\times ([P_0^{\epsilon_r}]((I-K^{\epsilon_r})\frac{1}{|S_j|}\int_{S_j}U^{\mathring{\epsilon}_r}(y)\,dy-U_j))
	\left.\right]\notag\\
	&=\Oh(c_ra)+\tilde{A}+\tilde{I}_B+\tilde{B}^{R}_{1,m}+\tilde{B}_{2,m}+\tilde{C},\notag
	\end{align}
	where $S_m$ is the largest sphere contained in $\Omega_m$, $V_m$ is the solution to \eqref{dislinear} and 
	\begin{align}
	&\tilde{A}=\frac{1}{|S_m|}\sum_{j=1\atop j\neq m}^M\int_{\Omega_j}\int_{S_m}\Pi_k(x,z)c_r^{-3}[C^\mathrm{T}_{\mu_r}]V^{\mathring{\mu}_r}(z)\,dx\,dz-a^3\sum_{j=1\atop j\neq m}^M\Pi_k(z_m,z_j)[C^\mathrm{T}_{\mu_r}]\frac{1}{|S_j|}\int_{S_j}V^{\mathring{\mu}_r}(y)\,dy\notag\\
	&-i k( \frac{1}{|S_m|}\sum_{j=1\atop j\neq m}^M\int_{\Omega_j}\int_{S_m}\nabla\Phi_k(x,z)\times(c_r^{-3}[C^\mathrm{T}_{\epsilon_r}]U^{\mathring{\epsilon}_r})\,dx\,dz-a^3\sum_{j=1\atop j\neq m}^M  \nabla\Phi_k(z_m,z_j)\times([C^\mathrm{T}_{\epsilon_r}]\frac{1}{|S_j|}\int_{S_j}U^{\mathring{\epsilon}_r}(y)\,dy)),\notag \\
	&\tilde{I}_B=\frac{1}{|S_m|}\int_{\Omega_m}\int_{S_m}\Pi_0(x,z)c_r^{-3}[C^\mathrm{T}_{\mu_r}](V^{\mathring{\mu}_r}(z)-\frac{1}{|S_m|}\int_{S_m}V^{\mathring{\mu}_r}(y)\,dy)\,dx\,dz,\notag\\
	&\tilde{B}^R_{1,m}=\frac{1}{|S_m|}\int_{\Omega_m}\int_{S_m}(\Pi_k-\Pi_0)(x,z)c_r^{-3}[C^\mathrm{T}_{\mu_r}]V^{\mathring{\mu}_r}(z)\,dx\,dz\notag\\
	&\tilde{B}_{2,m}=-\frac{1}{|S_m|}\int_{\Omega_m}\int_{S_m}ik\nabla\Phi_k(x,z)\times(c_r^{-3}[C^\mathrm{T}_{\epsilon_r}]U^{\mathring{\epsilon}_r}(z))\,dx\,dz\notag\\
	&\tilde{C}=\frac{1}{|S_m|}\int_L\int_{S_m}(\Pi_k(x,z)c_r^{-3}[C^\mathrm{T}_{\mu_r}]V^{\mathring{\mu}_r}(z)-ik\nabla\Phi_k(x,z)\times(c_r^{-3}[C^\mathrm{T}_{\epsilon_r}]U^{\mathring{\epsilon}_r}(z)))\,dx\,dz,\quad L=\Omega\backslash\cup_{m=1}^M\Omega_m.\notag
	\end{align}
	Furthermore, there holds the following estimates according to $\tilde{A}, \tilde{I}_B, \tilde{B}^R_{1,m}, \tilde{B}_2$ and $\tilde{C}$ with respect to $c_r$ and $a$ for sufficiently small $a$:
	\begin{align}\label{V-est}
	&\sum_{m=1}^M|\tilde{A}|^2=\Oh({c_0^{\mu_r}}^2c_r^{-9}a^{-3}),\quad
	\sum_{m=1}^M|\tilde{I}_B|^2=\Oh( {c_0^{\mu_r}}^2c_\alpha^2c_r^{2\alpha-9}a^{2\alpha-3}),\quad\alpha\in(0,1)\notag\\
	&\sum_{m=1}^M|\tilde{B}_{1,m}^R|^2=\Oh(k^4{c_0^{\mu_r}}^2\lVert V^{\mathring{\mu}_r}\rVert^2_{L^\infty(\Omega)}c_r^{-5}a),
	\sum_{m=1}^M|\tilde{B}_2|^2=\Oh(k^2{c_0^{\epsilon_r}}^2\lVert U^{\mathring{\epsilon}_r}\rVert^2_{L^\infty(\Omega)}(12+9k^2c_r^2a^2)c_r^{-7}a^{-1}),\notag\\
	&\sum_{m=1}^M|\tilde{C}|^2=\Oh\left(\right.({c_0^{\mu_r}}^2\lVert V^{\mathring{\mu}_r}\rVert^2_{L^\infty(\Omega)}(32c_{\mathrm{aver}}+k^4(36c_0^4\pi^2c_r^2a^2+k^2))\notag\\
	&\qquad+\pi^2k^2c_0^2{c_0^{\epsilon_r}}^2\lVert U^{\mathring{\epsilon}_r}\rVert^2_{L^\infty(\Omega)}(4+k^2c_0^2c_r^2a^2))c_r^{-7}a^{-1}\left.\right)
	\end{align}
for positive constants $c_\alpha$, $c_{\mathrm{aver}}$ and $c_0$.	
\end{corollary}

In the following proposition, on the basis of Proposition \ref{lem-eq}, we further present an $\ell_2$- estimate for  $\{(I-K^{\epsilon_r})\frac{1}{|S_m|}\int_{S_m}U^{\mathring{\epsilon}_r}(x)\,dx-U_m\}_{m=1}^M$.

\begin{proposition}\label{l2-esti}
	Following the same notations and assumptions in Proposition \ref{lem-eq}, there holds the $\ell_2$-esimate associated with $U^{\mathring{\epsilon}_r}$ and $U_m$ as 
	\begin{equation}\notag
		\left(\sum_{m=1}^M\big|(I-K^{\epsilon_r})\frac{1}{|S_m|}\int_{S_m}U^{\mathring{\epsilon}_r}(x)\,dx-U_m\big|^2\right)^{\frac{1}{2}}=\Oh\left(\lambda^+_{(\epsilon_r,\mu_r)}c_0^{\epsilon_r,-1}(c_0^{\epsilon_r}+c_0^{\mu_r})c_r^{-\frac{9}{2}}a^{-\frac{3}{2}}\right),
	\end{equation}
	where $\lambda^+_{(\epsilon_r,\mu_r)}, c_0^{\epsilon_r,-1}, c_0^{\epsilon_r}$ and $c_0^{\mu_r}$ are the positive constants given by \eqref{lamda+} and \eqref{notation}.
\end{proposition}

\begin{proof}
	Recall the discrete linear algebraic system \eqref{dislinear} with respect to $(U_m, V_m)_{m=1}^M$ and the estimates for $(\R_m^{\epsilon_r},\Q_m^{\mu_r})_{m=1}^M$ given in Proposition \ref{lem-dis}. Under the variable substitution \eqref{v-change}, where we know that
	\begin{equation}\notag
		U_m=a^{-3}[P_0^{\epsilon_r}]^{-1}\R_m^{\epsilon_r},\quad V_m=a^{-3}[P_0^{\mu_r}]^{-1}\Q_m^{\mu_r},
	\end{equation}
	we can rewrite the estimate \eqref{dis-inver0} as 
	\begin{align}\label{inver-est}
		(\sum_{m=1}^M|U_m|^2)^{\frac{1}{2}}&\leq \frac{9\lambda^+_{(\epsilon_r,\mu_r)}}{8}c_0^{\epsilon_r,-1}\left(\frac{1}{3} \left( \sum_{m=1}^M |H^{in}(z_m)|^2\right)^\frac{1}{2}+\left( \sum_{m=1}^M|E^{in}(z_m)|^2\right)^{\frac{1}{2}}\right),\notag\\
		(\sum_{m=1}^M|V_m|^2)^{\frac{1}{2}}&\leq \frac{9\lambda^+_{(\epsilon_r,\mu_r)}}{8}c_0^{\mu_r,-1}\left( \left( \sum_{m=1}^M |H^{in}(z_m)|^2\right)^\frac{1}{2}+\frac{1}{3}\left( \sum_{m=1}^M|E^{in}(z_m)|^2\right)^{\frac{1}{2}}\right).
	\end{align}
	Based on the $\ell_2$-estimate for $U_m$ in \eqref{inver-est}, by comparing with \eqref{main-est1} and \eqref{dislinear}, the following $\ell_2$-estimate for $((I-K^{\epsilon_r})\frac{1}{|S_m|}\int_{S_m}U^{\mathring{\epsilon}_r}(x)\,dx-U_m)_{m=1}^M$ holds
	\begin{align}\notag
		&\left( \sum_{m=1}^M\big| (I-K^{\epsilon_r})\frac{1}{|S_m|}\int_{S_m} U^{\mathring{\epsilon}_r}(x)\,dx-U_m\big|^2 \right)^{\frac{1}{2}}\leq \frac{9\lambda^+_{(\epsilon_r,\mu_r)}}{8}c_0^{\epsilon_r,-1}((\sum_{m=1}^M|\Oh(c_ra)+A+I_B\notag\\
		&+B^R_{1,m}+B_{2,m}+C|^2)^{\frac{1}{2}}+(\sum_{m=1}^M|\Oh(c_ra)+\tilde{A}+\tilde{I}_B+\tilde{B}^R_{1,m}+\tilde{B}_{2,m}+\tilde{C}|^2)^{\frac{1}{2}}).\notag
	\end{align}
	With the help of the estimates \eqref{group-est} in Proposition \ref{lem-eq} and \eqref{V-est} in Corollary \ref{cor-v}, we can further obtain that for sufficiently small $a$,
	\begin{align}\notag
		&\left( \sum_{m=1}^M\big| (I-K^{\epsilon_r})\frac{1}{|S_m|}\int_{S_m} U^{\mathring{\epsilon}_r}(x)\,dx-U_m\big|^2 \right)^{\frac{1}{2}}\leq \frac{9\lambda^+_{(\epsilon_r,\mu_r)}}{8}c_0^{\epsilon_r,-1}\Oh(d_0a^{-\frac{1}{2}}+d_1c_r^{-\frac{9}{2}}a^{-\frac{3}{2}}\notag\\
		&+d_2a^{\alpha-\frac{3}{2}}+d_3a^{\frac{1}{2}})=\Oh(\lambda^+_{(\epsilon_r,\mu_r)}c_0^{\epsilon_r,-1}d_1c_r^{-\frac{9}{2}}a^{-\frac{3}{2}}),\quad 0<\alpha<1,\notag
	\end{align}
	where
	\begin{align}\notag
		d_0&=c_r^{-\frac{1}{2}}+[k(12+9k^2c_r^2a^2)^{\frac{1}{2}}+(32c_{\mathrm{aver}}+k^4(36c_0^4\pi^2c_r^2a^2+k^2))^{\frac{1}{2}}\notag\\
		&+kc_0(2+kc_0c_ra)](c_0^{\epsilon_r}\lVert U^{\mathring{\epsilon}_r}\rVert_{L^\infty(\Omega)}+c_0^{\mu_r}\lVert V^{\mathring{\mu}_r}\rVert_{L^{\infty}(\Omega)})c_r^{-\frac{7}{2}},\notag\\
		d_1&=c_0^{\epsilon_r}+c_0^{\mu_r},\ d_2=c_\alpha(c_0^{\epsilon_r}+c_0^{\mu_r})c_r^{\alpha-\frac{9}{2}},\ 
		d_3=k^2(c_0^{\epsilon_r}\lVert U^{\mathring{\epsilon}_r}\rVert_{L^\infty(\Omega)}+c_0^{\mu_r}\lVert V^{\mathring{\mu}_r}\rVert_{L^{\infty}(\Omega)})c_r^{-\frac{5}{2}}.\notag
	\end{align}
\end{proof}

\begin{corollary}\label{Vl2-est}
	Following the same notations and assumptions in Corollary \ref{cor-v}, there holds the $\ell_2$-esimate associated with $V^{\mathring{\mu}_r}$ and $V_m$ as 
	\begin{equation}\notag
	\left(\sum_{m=1}^M\big|(I-K^{\mu_r})\frac{1}{|S_m|}\int_{S_m}V^{\mathring{\mu}_r}(x)\,dx-V_m\big|^2\right)^{\frac{1}{2}}=\Oh\left(\lambda^+_{(\epsilon_r,\mu_r)}c_0^{\mu_r,-1}(c_0^{\epsilon_r}+c_0^{\mu_r})c_r^{-\frac{9}{2}}a^{-\frac{3}{2}}\right),
	\end{equation}
	where $\lambda^+_{(\epsilon_r,\mu_r)}, c_0^{\mu_r,-1}, c_0^{\epsilon_r}$ and $c_0^{\mu_r}$ are the positive constants given by \eqref{lamda+} and \eqref{notation}.
\end{corollary}

Now, with the help of Proposition \ref{lem-eq} and Proposition \ref{l2-esti}, we are ready to give the proof of our main result, i.e. Theorem \ref{main-1}. 

\begin{proof}[Proof of Theorem \ref{main-1}]
	Recall the Foldy-Lax approximation \eqref{dis-far0} presented in Proposition \ref{lem-dis}. With variable substitution \eqref{v-change}, we can rewrite \eqref{dis-far0} as
	\begin{align}\label{far-dis}
		E^{\infty}(\hat{x})=\sum_{m=1}^{M}\left(\frac{k^2}{4\pi}e^{-i k \hat{x}\cdot z_m}\hat{x}\times (a^3[P_0^{\epsilon_r}]U_m\times \hat{x})+\frac{i k}{4\pi}e^{-i k \hat{x}\cdot z_m}\hat{x}\times a^3[P_0^{\mu_r}]V_m\right)\notag\\
		+\Oh(k c_\infty(\frac{1}{c_{\infty}^{\epsilon-}}+\frac{1}{c_{\infty}^{\mu-}})c_r^{-7}).
	\end{align}
	
	Meanwhile, the far-field corresponding to the Lippmann-Schwinger equation \eqref{ls1} associated with the equivalent medium with respect to the constant tensors $c_r^{-3}[C^{\mathrm{T}}_{\epsilon_r}]$ and $c_r^{-3}[C^{\mathrm{T}}_{\mu_r}]$ is formulated by
	\begin{equation}\label{far-int0}
		E^{\infty}_{\mathrm{eff}}(\hat{x})=\int_{\Omega}\left( \frac{k^2}{4\pi}e^{-ik\hat{x}\cdot z}\hat{x}\times (c_r^{-3}[C^{\mathrm{T}}_{\epsilon_r}]U^{\mathring{\epsilon}_r}(z)\times\hat{x})+\frac{ik}{4\pi}e^{-ik\hat{x}\cdot z}\hat{x}\times c_r^{-3}[C^{\mathrm{T}}_{\mu_r}]V^{\mathring{\mu}_r}(z)\right)\,dz.
	\end{equation}
	Since from \eqref{exp-C1}, where we know that 
	\begin{equation}\notag
		[C^\mathrm{T}_{\epsilon_r}]=(I-c_r^{-3}[P_0^{\epsilon_r}]K^0)^{-1}[P_0^{\epsilon_r}],
	\end{equation}
	it is easy to see the $I-c_r^{-3}[P_0^{\epsilon_r}]K^0$ is invertible for sufficiently large $c_r$ and there holds the following equation by Born series
	\begin{equation}\label{born}
		(I-c_r^{-3}[P_0^{\epsilon_r}]K^0)^{-1}=I+c_r^{-3}[P_0^{\epsilon_r}]K^0+\sum_{t=2}^\infty\left(c_r^{-3}[P_0^{\epsilon_r}]K^0\right)^t.
	\end{equation}
	Similar for $[C^\mathrm{T}_{\mu_r}]$, we have 
	\begin{equation}\label{born2}
		(I-c_r^{-3}[P_0^{\mu_r}]K^0)^{-1}=I+c_r^{-3}[P_0^{\mu_r}]K^0+\sum_{t=2}^\infty\left(c_r^{-3}[P_0^{\mu_r}]K^0\right)^t.
	\end{equation}
	Substituting \eqref{born} and \eqref{born2} into \eqref{far-int0}, we can know that
	\begin{align}\notag
		E^\infty_{\mathrm{eff}}(\hat{x})&=\int_{\Omega}\left(\right.\frac{k^2}{4\pi}e^{-ik \hat{x}\cdot z}\hat{x}\times[c_r^{-3}(I+c_r^{-3}[P_0^{\epsilon_r}]K^0+\sum_{t=2}^\infty\left(c_r^{-3}[P_0^{\epsilon_r}]K^0\right)^t)[P_0^{\epsilon_r}]U^{\mathring{\epsilon_r}}(z)\times\hat{x}]\notag\\
		&+\frac{ik}{4\pi}e^{-ik \hat{x}\cdot z}\hat{x}\times[c_r^{-3}(I+c_r^{-3}[P_0^{\mu_r}]K^0+\sum_{t=2}^\infty\left(c_r^{-3}[P_0^{\mu_r}]K^0\right)^t)[P_0^{\mu_r}]V^{\mathring{\mu_r}}(z)]
		\left.\right)\,dz,\notag
	\end{align}
	where for simplification, we write
	\begin{align}\notag
		&c_r^{-3}(I+c_r^{-3}[P_0^{\epsilon_r}]K^0+\sum_{t=2}^\infty\left(c_r^{-3}[P_0^{\epsilon_r}]K^0\right)^t)[P_0^{\epsilon_r}]\notag\\
		&=c_r^{-3}[P_0^{\epsilon_r}]+c_r^{-6}[P_0^{\epsilon_r}]K^0[P_0^{\epsilon_r}]+\Oh\left(c_r^{-9}[P_0^{\epsilon_r}]K^0[P_0^{\epsilon_r}]\right)\notag.
	\end{align}
	Similar expression works for $[C_{\mu_r}^\mathrm{T}]$. Therefore, $E^\infty_{\mathrm{eff}}(\hat{x})$ possesses the form
	\begin{align}\label{born-far}
		E^\infty_{\mathrm{eff}}(\hat{x})&=\int_{\Omega}\left(\right.\frac{k^2}{4\pi}e^{-ik\hat{x}\cdot z}\hat{x}\times(c_r^{-3}[P_0^{\epsilon_r}]+c_r^{-6}[P_0^{\epsilon_r}]K^0[P_0^{\epsilon_r}]+\Oh\left(c_r^{-9}[P_0^{\epsilon_r}]K^0[P_0^{\epsilon_r}]\right))U^{\mathring{\epsilon_r}}(z)\times\hat{x}\notag\\
		&+\frac{ik}{4\pi}e^{-ik\hat{x}\cdot z}\hat{x}\times(c_r^{-3}[P_0^{\mu_r}]+c_r^{-6}[P_0^{\mu_r}]K^0[P_0^{\mu_r}]+\Oh\left(c_r^{-9}[P_0^{\mu_r}]K^0[P_0^{\mu_r}]\right))V^{\mathring{\mu_r}}(z)
		\left.\right)\,dz.
	\end{align}	
	Following similar arguments, the effective electric permittivity $\mathring{\epsilon_r}$ and the magnetic permeability $\mathring{\mu_r}$ in \eqref{new-eu} can be represented as 
	\begin{align}\label{born-permi}
		\mathring{\epsilon_r}&=I+c_r^{-3}[P_0^{\epsilon_r}]+c_r^{-6}[P_0^{\epsilon_r}]K^0[P_0^{\epsilon_r}]+\Oh\left(c_r^{-9}[P_0^{\epsilon_r}]K^0[P_0^{\epsilon_r}]\right),\notag\\
		\mathring{\mu_r}&=I+c_r^{-3}[P_0^{\epsilon_r}]+c_r^{-6}[P_0^{\epsilon_r}]K^0[P_0^{\epsilon_r}]+\Oh\left(c_r^{-9}[P_0^{\epsilon_r}]K^0[P_0^{\epsilon_r}]\right).
	\end{align}
	
	However, in order to compare $E^\infty$ and $E^\infty_{\mathrm{eff}}$ more directly and for the convenience of calculations, we apply the following form of $E^\infty_{\mathrm{eff}}$, which is derive from \eqref{far-int0} as
	\begin{align}\label{far-int1}
		E^{\infty}_{\mathrm{eff}}(\hat{x})&=\int_{\Omega}\left(\right. \frac{k^2}{4\pi}e^{-ik\hat{x}\cdot z}\hat{x}\times (c_r^{-3}[C^{\mathrm{T}}_{\epsilon_r}](I-K^{\epsilon_r})^{-1}(I-K^{\epsilon_r})U^{\mathring{\epsilon}_r}(z)\times\hat{x})\notag\\
		&+\frac{ik}{4\pi}e^{-ik\hat{x}\cdot z}\hat{x}\times c_r^{-3}[C^{\mathrm{T}}_{\mu_r}](I-K^{\mu_r})^{-1}(I-K^{\mu_r})V^{\mathring{\mu}_r}(z)\left.\right)\,dz\notag\\
		=&\int_{\Omega}\left( \frac{k^2}{4\pi}e^{-ik\hat{x}\cdot z}\hat{x}\times (c_r^{-3}[P_0^{\epsilon_r}](I-K^{\epsilon_r})U^{\mathring{\epsilon}_r}(z)\times\hat{x})+\frac{ik}{4\pi}e^{-ik\hat{x}\cdot z}\hat{x}\times c_r^{-3}[P_0^{\mu_r}](I-K^{\mu_r})V^{\mathring{\mu}_r}(z)\right)\,dz,
	\end{align}
	under condition a) in Definition \ref{def2}. 
	
	To investigate the error between $E^{\infty}(\hat{x})$ and $E^{\infty}_{\mathrm{eff}}(\hat{x})$, we subtract \eqref{far-dis} from \eqref{far-int1}, to obtain
	\begin{align}\label{ero-m}
		&E^{\infty}_{\mathrm{eff}}(\hat{x})-E^{\infty}(\hat{x})\notag\\
		&=\sum_{m=1}^M\int_{\Omega_m}\left( \frac{k^2}{4\pi}e^{-ik\hat{x}\cdot z}\hat{x}\times (c_r^{-3}[P_0^{\epsilon_r}](I-K^{\epsilon_r})U^{\mathring{\epsilon}_r}(z)\times\hat{x})-\frac{k^2}{4\pi}e^{-i k \hat{x}\cdot z_m}\hat{x}\times (c_r^{-3}[P_0^{\epsilon_r}]U_m\times \hat{x})\right)\,dz\notag\\
		&+\sum_{m=1}^M\int_{\Omega_m}\left(\frac{ik}{4\pi}e^{-ik\hat{x}\cdot z}\hat{x}\times c_r^{-3}[P_0^{\mu_r}](I-K^{\mu_r})V^{\mathring{\mu}_r}(z)-\frac{i k}{4\pi}e^{-i k \hat{x}\cdot z_m}\hat{x}\times c_r^{-3}[P_0^{\mu_r}]V_m \right)\,dz\notag\\
		&+\int_{L}\left( \frac{k^2}{4\pi}e^{-ik\hat{x}\cdot z}\hat{x}\times (c_r^{-3}[P_0^{\epsilon_r}](I-K^{\epsilon_r})U^{\mathring{\epsilon}_r}(z)\times\hat{x})+\frac{ik}{4\pi}e^{-ik\hat{x}\cdot z}\hat{x}\times c_r^{-3}[P_0^{\mu_r}](I-K^{\mu_r})V^{\mathring{\mu}_r}(z)\right)\,dz\notag\\
		&+\Oh(k c_\infty(\frac{1}{c_{\infty}^{\epsilon-}}+\frac{1}{c_{\infty}^{\mu-}})c_r^{-7})
	\end{align}
	where $L=\Omega\backslash\cup_{m=1}^M \Omega_m$. 
	{\bl
	Denote
	\begin{equation}\notag
		F^\infty:=\sum_{m=1}^M\int_{\Omega_m}\left( \frac{k^2}{4\pi}e^{-ik\hat{x}\cdot z}\hat{x}\times (c_r^{-3}[P_0^{\epsilon_r}](I-K^{\epsilon_r})U^{\mathring{\epsilon}_r}(z)\times\hat{x})-\frac{k^2}{4\pi}e^{-i k \hat{x}\cdot z_m}\hat{x}\times (c_r^{-3}[P_0^{\epsilon_r}]U_m\times \hat{x})\right)\,dz,
	\end{equation}
	\begin{equation}\notag
		G^\infty:=\sum_{m=1}^M\int_{\Omega_m}\left(\frac{ik}{4\pi}e^{-ik\hat{x}\cdot z}\hat{x}\times c_r^{-3}[P_0^{\mu_r}](I-K^{\mu_r})V^{\mathring{\mu}_r}(z)-\frac{i k}{4\pi}e^{-i k \hat{x}\cdot z_m}\hat{x}\times c_r^{-3}[P_0^{\mu_r}]V_m \right)\,dz,
	\end{equation}
	and 
	\begin{equation}\notag
		H^\infty:=\int_{L}\left( \frac{k^2}{4\pi}e^{-ik\hat{x}\cdot z}\hat{x}\times (c_r^{-3}[P_0^{\epsilon_r}](I-K^{\epsilon_r})U^{\mathring{\epsilon}_r}(z)\times\hat{x})+\frac{ik}{4\pi}e^{-ik\hat{x}\cdot z}\hat{x}\times c_r^{-3}[P_0^{\mu_r}](I-K^{\mu_r})V^{\mathring{\mu}_r}(z)\right)\,dz.
	\end{equation}  }
We can rewrite $F^\infty$ as
	\begin{align}\notag
		F^{\infty}
		&=\sum_{m=1}^M\int_{\Omega_m}( \frac{k^2}{4\pi}e^{-ik\hat{x}\cdot z}\hat{x}\times (c_r^{-3}[P_0^{\epsilon_r}](I-K^{\epsilon_r})(U^{\mathring{\epsilon}_r}(z)-\frac{1}{|S_m|}\int_{S_m}U^{\mathring{\epsilon}_r}(y)\,dy)\times\hat{x})\,dz\notag\\
		&+\sum_{m=1}^M\int_{\Omega_m}\left(\right.\frac{k^2}{4\pi}e^{-ik\hat{x}\cdot z}\hat{x}\times (c_r^{-3}[P_0^{\epsilon_r}](I-K^{\epsilon_r}) \frac{1}{|S_m|}\int_{S_m}U^{\mathring{\epsilon}_r}(y)\,dy\times\hat{x})\notag\\
		&-\frac{k^2}{4\pi}e^{-i k \hat{x}\cdot z_m}\hat{x}\times (c_r^{-3}[P_0^{\epsilon_r}]U_m\times \hat{x})\left.\right)\,dz:=F_1^\infty+F_2^\infty.\notag
	\end{align}
	Since $I-K^{\epsilon_r}$ is invertible and $K^{\epsilon_r}$ satisfies condition b) in Definition \ref{def2}, we see that $\lVert I-K^{\epsilon_r}\rVert_{L^\infty}<2$. Combining with \eqref{U-reg}, which shows that $U^{\mathring{\epsilon_r}}\in C^{0,\alpha}(\Omega)$ for $\alpha\in(0,1)$, it is direct to obtain
	\begin{equation}\label{F1}
		|F_1^\infty|\leq \sum_{m=1}^M \frac{k^2}{4\pi}c_r^{-3}c_{P_0}^{\epsilon_r}|\Omega_m|\lVert I-K^{\epsilon_r}\rVert_{L^\infty} c_\alpha 3^{\frac{\alpha}{2}}\delta^\alpha\leq 
		3^{\frac{\alpha}{2}}\frac{k^2}{2\pi}c_r^{-3}c_{P_0}^{\epsilon_r}c_\alpha\delta^\alpha=\Oh(k^2c_{P_0}^{\epsilon_r}c_\alpha c_r^{\alpha-3}a^\alpha),
	\end{equation}
	where $c^{\epsilon_r}_{P_0}$ is defined in \eqref{notation}. 
	
	For $F_2^\infty$, since by Taylor expansion there holds 
	\begin{equation}\notag
		e^{-ik\hat{x}\cdot z}=e^{-ik\hat{x}\cdot z_m}+\sum_{n=1}^\infty(-ik\hat{x})^{2n}\frac{e^{-ik\hat{x}\cdot z_m}}{n!}(z-z_m)^n,
	\end{equation}
	we have
	\begin{align}\notag
		F_2^\infty&=\sum_{m=1}^M\int_{\Omega_m}\frac{k^2}{4\pi}e^{-ik\hat{x}\cdot z_m}\hat{x}\times(c_r^{-3}[P_0^{\epsilon_r}]((I-K^{\epsilon_r}) \frac{1}{|S_m|}\int_{S_m}U^{\mathring{\epsilon}_r}(y)\,dy-U_m)\times\hat{x})\,dz\notag\\
		&+\sum_{m=1}^M\int_{\Omega_m}\frac{k^2}{4\pi}(\sum_{n=1}^\infty (-ik\hat{x})^{2n}\frac{e^{-i k \hat{x}\cdot z_m}}{n!}(z-z_m)^n)\hat{x}\times (c_r^{-3}[P_0^{\epsilon_r}](I-K^{\epsilon_r})\frac{1}{|S_m|}\int_{S_m}U^{\mathring{\epsilon}_r}(y)\,dy\times\hat{x})\,dz\notag
	\end{align}
	{\bl Denote
	\begin{equation}\notag
		F_{21}^\infty:=\sum_{m=1}^M\int_{\Omega_m}\frac{k^2}{4\pi}e^{-ik\hat{x}\cdot z_m}\hat{x}\times(c_r^{-3}[P_0^{\epsilon_r}]((I-K^{\epsilon_r}) \frac{1}{|S_m|}\int_{S_m}U^{\mathring{\epsilon}_r}(y)\,dy-U_m)\times\hat{x})\,dz,
	\end{equation}
	and
	\begin{equation}\notag
		F_{22}^{\infty}:=\sum_{m=1}^M\int_{\Omega_m}\frac{k^2}{4\pi}(\sum_{n=1}^\infty (-ik\hat{x})^{2n}\frac{e^{-i k \hat{x}\cdot z_m}}{n!}(z-z_m)^n)\hat{x}\times (c_r^{-3}[P_0^{\epsilon_r}](I-K^{\epsilon_r})\frac{1}{|S_m|}\int_{S_m}U^{\mathring{\epsilon}_r}(y)\,dy\times\hat{x})\,dz.
	\end{equation}  }
	By utilizing Proposition \ref{l2-esti}, we know that
	\begin{align}\label{F21}		|F_{21}^\infty|&\leq\frac{k^2}{4\pi}c_r^{-3}c^{\epsilon_r}_{P_0}\sum_{m=1}^{M}\big|(I-K^{\epsilon_r}) \frac{1}{|S_m|}\int_{S_m}U^{\mathring{\epsilon}_r}(y)\,dy-U_m
		\big|\cdot|\Omega_m|\notag\\
		&\leq \frac{k^2}{4\pi}c_r^{-3}c^{\epsilon_r}_{P_0}|\Omega_m|M^{\frac{1}{2}}\left(\sum_{m=1}^{M}\big|(I-K^{\epsilon_r})\frac{1}{|S_m|}\int_{S_m}U^{\mathring{\epsilon}_r}(y\,dy-U_m)\big|^2\right)^{\frac{1}{2}}\notag\\
		&=\Oh\left( \frac{k^2}{4\pi}c_r^{-3}c^{\epsilon_r}_{P_0}\delta^{\frac{3}{2}}\lambda^+_{(\epsilon_r,\mu_r)}c_0^{\epsilon_r,-1}(c_0^{\epsilon_r}+c_0^{\mu_r})c_r^{-\frac{9}{2}}a^{-\frac{3}{2}}\right)\notag\\
		&=\Oh\left(k^2c^{\epsilon_r}_{P_0}\lambda^+_{(\epsilon_r,\mu_r)}c_0^{\epsilon_r,-1}(c_0^{\epsilon_r}+c_0^{\mu_r})c_r^{-6}\right).
	\end{align}
	For $F_{22}^\infty$, it is direct to notice that for sufficiently small $a$
	\begin{align}\label{F22}
	|F_{22}^\infty|&\leq \sum_{m=1}^{M}\frac{k^2}{4\pi}|\Omega_m|\big|\sum_{n=1}^{\infty}k^{2n}\frac{(\sqrt{3}\delta)^n}{n!}\big|c_r^{-3}c^{\epsilon_r}_{P_0}\lVert I-K^{\epsilon_r}\rVert_{L^{\infty}}\lVert U^{\mathring{\epsilon}_r}\rVert_{L^\infty(\Omega)}\notag\\
	&=\Oh\left(k^2c_r^{-3}c^{\epsilon_r}_{P_0}\lVert U^{\mathring{\epsilon}_r}\rVert_{L^\infty(\Omega)}\delta\right)=\Oh\left(k^2c^{\epsilon_r}_{P_0}\lVert U^{\mathring{\epsilon}_r}\rVert_{L^\infty(\Omega)}c_r^{-2}a\right).
	\end{align}
	Combining with \eqref{F1}, \eqref{F21} and \eqref{F22}, we can obtain that
	\begin{align}\label{est-F}
	|F^\infty|&\leq|F_1^\infty|+|F_{21}^\infty|+|F_{22}^\infty|\notag\\
	&=\Oh\left( k^2c_{P_0}^{\epsilon_r}c_\alpha c_r^{\alpha-3}a^\alpha+k^2c^{\epsilon_r}_{P_0}\lambda^+_{(\epsilon_r,\mu_r)}c_0^{\epsilon_r,-1}(c_0^{\epsilon_r}+c_0^{\mu_r})c_r^{-6}+k^2c^{\epsilon_r}_{P_0}\lVert U^{\mathring{\epsilon}_r}\rVert_{L^\infty(\Omega)}c_r^{-2}a \right)\notag\\
	&=\Oh\left( k^2c^{\epsilon_r}_{P_0}\lambda^+_{(\epsilon_r,\mu_r)}c_0^{\epsilon_r,-1}(c_0^{\epsilon_r}+c_0^{\mu_r})c_r^{-6} \right)
	\end{align}
	for $a$ small enough, where $0<\alpha<1$.
	
	Recall \eqref{ero-m} and let us consider $G^\infty$. Similar to $F^\infty$, we can write $G^\infty$ as 
	\begin{align}\notag
		&G^\infty=\sum_{m=1}^M\int_{\Omega_m}\left(\frac{ik}{4\pi}e^{-ik\hat{x}\cdot z}\hat{x}\times c_r^{-3}[P_0^{\mu_r}](I-K^{\mu_r})V^{\mathring{\mu}_r}(z)-\frac{i k}{4\pi}e^{-i k \hat{x}\cdot z_m}\hat{x}\times c_r^{-3}[P_0^{\mu_r}]V_m \right)\,dz\notag\\
		&=\sum_{m=1}^M\int_{\Omega_m}\left(\frac{ik}{4\pi}e^{-ik\hat{x}\cdot z}\hat{x}\times c_r^{-3}[P_0^{\mu_r}](I-K^{\mu_r})\left(V^{\mathring{\mu}_r}(z)-\frac{1}{|S_m|}\int_{S_m}V^{\mathring{\mu}_r}(y)\,dy\right)\right)\,dz\notag\\
		&+\sum_{m=1}^M\int_{\Omega_m}\frac{ik}{4\pi}\left(e^{-ik\hat{x}\cdot z}\hat{x}\times c_r^{-3}[P_0^{\mu_r}](I-K^{\mu_r})\frac{1}{|S_m|}\int_{S_m}V^{\mathring{\mu}_r}(y)\,dy-e^{-ik\hat{x}\cdot z_m}\hat{x}\times c_r^{-3}[P_0^{\mu_r}] V_m \right)\,dz.\notag
	\end{align}
	Denote 
	\begin{equation}\notag
		G_1^\infty=\sum_{m=1}^M\int_{\Omega_m}\left(\frac{ik}{4\pi}e^{-ik\hat{x}\cdot z}\hat{x}\times c_r^{-3}[P_0^{\mu_r}](I-K^{\mu_r})\left(V^{\mathring{\mu}_r}(z)-\frac{1}{|S_m|}\int_{S_m}V^{\mathring{\mu}_r}(y)\,dy\right)\right)\,dz,
	\end{equation}
	and 
	\begin{equation}\notag
	G_2^\infty=\sum_{m=1}^M\int_{\Omega_m}\frac{ik}{4\pi}\left(e^{-ik\hat{x}\cdot z}\hat{x}\times c_r^{-3}[P_0^{\mu_r}](I-K^{\mu_r})\frac{1}{|S_m|}\int_{S_m}V^{\mathring{\mu}_r}(y)\,dy-e^{-ik\hat{x}\cdot z_m}\hat{x}\times c_r^{-3}[P_0^{\mu_r}] V_m \right)\,dz.
	\end{equation}
	By \eqref{V-reg} which shows that $V^{\mathring{\mu_r}}\in C^{0,\alpha}(\Omega)$, following the similar arguments to the estimate for $F_1^\infty$ in \eqref{F1}, we directly have
	\begin{equation}\notag
		|G_1^\infty|=\Oh(kc_{P_0}^{\mu_r}c_\alpha c_r^{\alpha-3}a^\alpha).
	\end{equation}
	Utilizing Taylor expansion for $e^{-ik\hat{x}\cdot z}$ again, $G_2^\infty$ can be reformulated as
	\begin{equation}\notag
		G_2^\infty:=G_{21}^\infty+G_{22}^\infty,
	\end{equation}
	where
	\begin{equation}\notag
		G^\infty_{21}:=\sum_{m=1}^M\int_{\Omega_m}\frac{ik}{4\pi}e^{-ik\hat{x}\cdot z_m}\hat{x}\times c_r^{-3}[P_0^{\mu_r}]\left((I-K^{\mu_r})\frac{1}{|S_m|}\int_{S_m}V^{\mathring{\mu}_r}(y)\,dy-V_m\right)\,dz,
	\end{equation}
	and 
	\begin{equation}\notag
		G^\infty_{22}:=\sum_{m=1}^M\int_{\Omega_m}\frac{ik}{4\pi}\left(\sum_{n=1}^{\infty}(-ik\hat{x})^{2n}\frac{e^{-ik\hat{x}\cdot z_m}}{n!}(z-z_m)^n\right)\hat{x}\times c_r^{-3}[P_0^{\mu_r}](I-K^{\mu_r})\frac{1}{|S_m|}\int_{S_m}V^{\mathring{\mu}_r}(y)\,dy.
	\end{equation}
	Similar to \eqref{F21} and \eqref{F22}, we can deduce from Corollary \ref{Vl2-est} that
	\begin{equation}\notag
		|G_{21}^\infty|=\Oh(kc^{\mu_r}_{P_0}\lambda^+_{(\epsilon_r,\mu_r)}c_0^{\mu_r,-1}(c_0^{\epsilon_r}+c_0^{\mu_r})c_r^{-6}),\quad
		|G_{22}^\infty|=\Oh(kc^{\mu_r}_{P_0}\lVert V^{\mathring{\mu}_r}\rVert_{L^\infty(\Omega)}c_r^{-2}a).
	\end{equation}
	Therefore, there holds
	\begin{align}\label{G-est}
		|G^\infty|&\leq |G_1^\infty|+|G_{21}^\infty|+|G_{22}^\infty|\notag\\
		&=\Oh(kc_{P_0}^{\mu_r}c_\alpha c_r^{\alpha-3}a^\alpha+kc^{\mu_r}_{P_0}\lambda^+_{(\epsilon_r,\mu_r)}c_0^{\mu_r,-1}(c_0^{\epsilon_r}+c_0^{\mu_r})c_r^{-6}+kc^{\mu_r}_{P_0}\lVert V^{\mathring{\mu}_r}\rVert_{L^\infty(\Omega)}c_r^{-2}a)\notag\\
		&=\Oh(kc^{\mu_r}_{P_0}\lambda^+_{(\epsilon_r,\mu_r)}c_0^{\mu_r,-1}(c_0^{\epsilon_r}+c_0^{\mu_r})c_r^{-6}),\quad 0<\alpha<1,
	\end{align}
	for $a$ small enough.
	
	Finally, we evaluate $H^\infty$ as follows. Recall the representation of $H^\infty$ in \eqref{ero-m} as 
	\begin{equation}\notag
		H^\infty=\int_{L}\left( \frac{k^2}{4\pi}e^{-ik\hat{x}\cdot z}\hat{x}\times (c_r^{-3}[P_0^{\epsilon_r}](I-K^{\epsilon_r})U^{\mathring{\epsilon}_r}(z)\times\hat{x})+\frac{ik}{4\pi}e^{-ik\hat{x}\cdot z}\hat{x}\times c_r^{-3}[P_0^{\mu_r}](I-K^{\mu_r})V^{\mathring{\mu}_r}(z)\right)\,dz.
	\end{equation}
	Since we know that $|L|=\Oh(\delta)$ by Remark \ref{layer}, we can directly give the estimate for $H^\infty$ as
	\begin{align}\label{H-est}
		|H^\infty|&\leq \left(\frac{k^2}{4\pi}c_r^{-3}c_{P_0}^{\epsilon_r}\lVert I-K^{\epsilon_r}\rVert_{L^\infty}\lVert U^{\mathring{\epsilon}_r}\rVert_{L^{\infty}(\Omega)}+\frac{k}{4\pi}c_r^{-3}c_{P_0}^{\mu_r}\lVert I-K^{\mu_r}\rVert_{L^\infty}\lVert V^{\mathring{\mu}_r}\rVert_{L^\infty(\Omega)}\right)\cdot|L|\notag\\
		&=\Oh\left((k^2c_{P_0}^{\epsilon_r}\lVert U^{\mathring{\epsilon}_r}\rVert_{L^{\infty}(\Omega)}+kc_{P_0}^{\mu_r}\lVert V^{\mathring{\mu}_r}\rVert_{L^\infty(\Omega)})c_r^{-2}a\right)
	\end{align}
	for sufficiently small $a$. 
	
	Substituting \eqref{est-F}, \eqref{G-est} and \eqref{H-est} into \eqref{ero-m}, we derive our main result
	\begin{equation}\label{final-err}
		E^\infty(\hat{x})-E^\infty_{\mathrm{eff}}(\hat{x})=\Oh\left(
		k\lambda^+_{(\epsilon_r,\mu_r)}(kc^{\epsilon_r}_{P_0}c_0^{\epsilon_r,-1}+c^{\mu_r}_{P_0}c_0^{\mu_r,-1})(c_0^{\epsilon_r}+c_0^{\mu_r})c_r^{-6}
		\right).
	\end{equation}
	Based on \eqref{final-err}, which indicates that the error estimate between $E^\infty(\hat{x})$ and $E^\infty_{\mathrm{eff}}(\hat{x})$ is of order $c_r^{-6}$, we can finalize our proof by rewriting the far-field \eqref{born-far} for the equivalent medium in terms of \eqref{final-err} as
	\begin{equation}\notag
		E^{\infty}_{\mathrm{eff}}(\hat{x})=\int_{\Omega}\left(\frac{k^2}{4\pi}e^{-i k\hat{x}\cdot z}\hat{x}\times (c_r^{-3}[P_0^{\epsilon_r}] U^{\mathring{\epsilon}_r}\times\hat{x})+\frac{i k}{4\pi}e^{-i k\hat{x}\cdot z}\hat{x}\times c_r^{-3}[P_0^{\mu_r}]V^{\mathring{\mu}_r}\right)\,dz.
	\end{equation}
	In addition, the similar representations works for $\mathring{\epsilon_r}$ and $\mathring{\mu_r}$ in \eqref{born-permi} as
	\begin{equation}\notag
		\mathring{\epsilon}_r=I+c_r^{-3}[P^{\epsilon_r}_0], \quad
		\mathring{\mu}_r=I+c_r^{-3}[P^{\mu_r}_0],
	\end{equation} 
	in terms of \eqref{final-err}.
	
	The proof is complete.
\end{proof}

\section{Proof of Proposition \ref{lem-eq}}\label{sec-prop}

Recall the discrete linear algebraic system \eqref{dislinear} and the Lippmann-Schwinger equation \eqref{ls1}. 
	Taking the average value with respect to $x$ over $S_m$ in \eqref{ls1}, we obtain
	\begin{align}\label{aver1}
	\frac{1}{|S_m|}\int_{S_m}U^{\mathring{\epsilon}_r}(x)\,dx-\int_{\Omega}(\frac{1}{|S_m|}\int_{S_m}\Pi_k(x,z)\,dx c_r^{-3}[C^{\mathrm{T}
	}_{\epsilon_r}]U^{\mathring{\epsilon}_r}(z)+ik \frac{1}{|S_m|}\int_{S_m}\nabla\Phi_k(x,z)\,dx\times\notag\\
	(c_r^{-3}[C^{\mathrm{T}}_{\mu_r}]V^{\mathring{\mu}_r}(z)))\,dz =\frac{1}{|S_m|}\int_{S_m} E^{in}(x)\,dx,
	\end{align}
	where $S_m$ is the largest ball contained in $\Omega_m$. 
	
	Combining with \eqref{aver1} and \eqref{dislinear}, we have the following equation
	\begin{align}\label{main-eq1}
	&\frac{1}{|S_m|}\int_{S_m}U^{\mathring{\epsilon}_r}(x)\,dx-a^3\sum_{j=1\atop j\neq m}^M(\Pi_k(z_m,z_j)[C^{\mathrm{T}}_{\epsilon_r}]\frac{1}{|S_j|}\int_{S_j}U^{\mathring{\epsilon}_r}(y)\,dy+ik\nabla\Phi_k(z_m,z_j)\notag\\
	&\times([C^{\mathrm{T}}_{\mu_r}]\frac{1}{|S_j|}\int_{S_j}V^{\mathring{\mu}_r}(y)\,dy))=\frac{1}{|S_m|}\int_{S_m}E^{in}(x)\,dx+\int_{\Omega}(\frac{1}{|S_m|}\int_{S_m}\Pi_k(x,z)\,dx c_r^{-3}[C^{\mathrm{T}
	}_{\epsilon_r}]U^{\mathring{\epsilon}_r}(z)+\notag\\
	&ik \frac{1}{|S_m|}\int_{S_m}\nabla\Phi_k(x,z)\,dx\times
	(c_r^{-3}[C^{\mathrm{T}}_{\mu_r}]V^{\mathring{\mu}_r}(z)))\,dz
	-a^3\sum_{j=1\atop j\neq m}^M(\Pi_k(z_m,z_j)[C^{\mathrm{T}}_{\epsilon_r}]\frac{1}{|S_j|}\int_{S_j}U^{\mathring{\epsilon}_r}(y)\,dy\notag\\
	&+ik\nabla\Phi_k(z_m,z_j)
	\times([C^{\mathrm{T}}_{\mu_r}]\frac{1}{|S_j|}\int_{S_j}V^{\mathring{\mu}_r}(y)\,dy)).
	\end{align}
	It is easy to see that by Taylor expansion of $E^{in}$ at $z_m\in \Omega_m$, there holds
	\begin{equation}\notag
	E^{in}(x)=Pe^{ik\theta\cdot x}=P(e^{ik\theta\cdot z_m}\sum_{n=0}^\infty\frac{1}{n!}(ik\theta\cdot x-ik \theta\cdot z_m)^n)=Pe^{ik\theta\cdot z_m}+\Oh(c_r a)=E^{in}(z_m)+\Oh(c_r a).
	\end{equation}
	for $a$ small enough. Therefore, \eqref{main-eq1} can be further written as 
	\begin{align}\label{main-eq2}
	&\frac{1}{|S_m|}\int_{S_m}U^{\mathring{\epsilon}_r}(x)\,dx-a^3\sum_{j=1\atop j\neq m}^M(\Pi_k(z_m,z_j)[C^{\mathrm{T}}_{\epsilon_r}]\frac{1}{|S_j|}\int_{S_j}U^{\mathring{\epsilon}_r}(y)\,dy+ik\nabla\Phi_k(z_m,z_j)\notag\\
	&\times([C^{\mathrm{T}}_{\mu_r}]\frac{1}{|S_j|}\int_{S_j}V^{\mathring{\mu}_r}(y)\,dy))=E^{in}(z_m)+\Oh(c_r a)+\frac{1}{|S_m|}\sum_{j=1\atop j\neq m}^M\int_{\Omega_j}\int_{S_m}\Pi_k(x,z)c_r^{-3}[C^{\mathrm{T}}_{\epsilon_r}]U^{\mathring{\epsilon}_r}(z)\,dx\,dz\notag\\
	&-a^3\sum_{j=1\atop j\neq m}^M\Pi_k(z_m,z_j)[C^{\mathrm{T}}_{\epsilon_r}]\frac{1}{|S_j|}\int_{S_j}U^{\mathring{\epsilon_r}}(y)\,dy
	+ik\frac{1}{|S_m|}\sum_{j=1\atop j\neq m}^M\int_{\Omega_j}\int_{S_m}\nabla\Phi_k(x,z)\times (c_r^{-3}[C^{\mathrm{T}}_{\mu_r}]V^{\mathring{\mu}_r}(z))\,dx\,dz\notag\\
	&-a^3 \sum_{j=1\atop j\neq m}^M ik \nabla\Phi_k(z_m,z_j)
	\times([C^{\mathrm{T}}_{\mu_r}]\frac{1}{|S_j|}\int_{S_j}V^{\mathring{\mu}_r}(y)\,dy)\notag\\
	&+\frac{1}{|S_m|}\int_{\Omega_m}\int_{S_m}(\Pi_k(x,z)c_r^{-3}[C^{\mathrm{T}}_{\epsilon_r}]U^{\mathring{\epsilon}_r}(z)
	+ik\nabla\Phi_k(x,z)
	\times(c_r^{-3}[C^{\mathrm{T}}_{\mu_r}]V^{\mathring{\mu}_r}(z)))\,dx\,dz\notag\\
	&+\frac{1}{|S_m|}\int_L\int_{S_m}(\Pi_k(x,z)c_r^{-3}[C^{\mathrm{T}}_{\epsilon_r}]U^{\mathring{\epsilon}_r}(z)
	+ik\nabla\Phi_k(x,z)\times(c_r^{-3}[C^{\mathrm{T}}_{\mu_r}]V^{\mathring{\mu}_r}(z)))\,dx\,dz,
	\end{align}
	where $L:=\Omega\backslash\cup_{m=1}^M\Omega_m$. Denote
	\begin{align}\label{A}
	A&:=\frac{1}{|S_m|}\sum_{j=1\atop j\neq m}^M\int_{\Omega_j}\int_{S_m}\Pi_k(x,z)c_r^{-3}[C^{\mathrm{T}}_{\epsilon_r}]U^{\mathring{\epsilon}_r}(z)\,dx\,dz-a^3\sum_{j=1\atop j\neq m}^M\Pi_k(z_m,z_j)[C^{\mathrm{T}}_{\epsilon_r}]\frac{1}{|S_j|}\int_{S_j}U^{\mathring{\epsilon}_r}(y)\,dy\notag\\
	&+ik \frac{1}{|S_m|}\sum_{j=1\atop j\neq m}^M\int_{\Omega_j}\int_{S_m}\nabla\Phi_k(x,z)\times(c_r^{-3}[C^{\mathrm{T}}_{\mu_r}]V^{\mathring{\mu}_r}(z))\,dx\,dz-a^3\sum_{j=1\atop j\neq m}^M ik\nabla\Phi_k(z_m,z_j)\notag\\
	&\times([C^{\mathrm{T}}_{\mu_r}]\frac{1}{|S_j|}\int_{S_j}V^{\mathring{\mu}_r}(y)\,dy),
	\end{align}
	\begin{equation}\label{B}
	B:=\frac{1}{|S_m|}\int_{\Omega_m}\int_{S_m}(\Pi_k(x,z)c_r^{-3}[C^{\mathrm{T}}_{\epsilon_r}]U^{\mathring{\epsilon}_r}(z)
	+ik\nabla\Phi_k(x,z)
	\times(c_r^{-3}[C^{\mathrm{T}}_{\mu_r}]V^{\mathring{\mu}_r}(z)))\,dx\,dz,
	\end{equation}
	and 
	\begin{equation}\label{C}
	C:=\frac{1}{|S_m|}\int_{L}\int_{S_m}(\Pi_k(x,z)c_r^{-3}[C^{\mathrm{T}}_{\epsilon_r}]U^{\mathring{\epsilon}_r}(z)
	+ik\nabla\Phi_k(x,z)
	\times(c_r^{-3}[C^{\mathrm{T}}_{\mu_r}]V^{\mathring{\mu}_r}(z)))\,dx\,dz.
	\end{equation}
	
	In the following discussions, we evaluate $A$, $B$ and $C$, respectively.
	
	(\romannumeral1)\emph{The estimate for $A$}.
	Denote 
	\begin{align}\notag
	A_{1,m}:&=\frac{1}{|S_m|}\sum_{j=1\atop j\neq m}^M\int_{\Omega_j}\int_{S_m}\Pi_k(x,z)c_r^{-3}[C^{\mathrm{T}}_{\epsilon_r}]U^{\mathring{\epsilon}_r}(z)\,dx\,dz-a^3\sum_{j=1\atop j\neq m}^M\Pi_k(z_m,z_j)[C^{\mathrm{T}}_{\epsilon_r}]\frac{1}{|S_j|}\int_{S_j}U^{\mathring{\epsilon}_r}(y)\,dy\notag
	\end{align} 
	\begin{equation}\notag
	A_{1,m}^0:=\frac{1}{|S_m|}\sum_{j=1\atop j\neq m}^M\int_{\Omega_j}\int_{S_m}\Pi_0(x,z)c_r^{-3}[C^{\mathrm{T}}_{\epsilon_r}]U^{\mathring{\epsilon}_r}(z)\,dx\,dz-a^3\sum_{j=1\atop j\neq m}^M\Pi_0(z_m,z_j)[C^{\mathrm{T}}_{\epsilon_r}]\frac{1}{|S_j|}\int_{S_j}U^{\mathring{\epsilon}_r}(y)\,dy,
	\end{equation}
	and 
	\begin{align}\label{A1mr}
	A_{1,m}^R&:=\frac{1}{|S_m|}\sum_{j=1\atop j\neq m}^M\int_{\Omega_j}\int_{S_m}(\Pi_k-\Pi_0)(x,z)c_r^{-3}[C^{\mathrm{T}}_{\epsilon_r}]U^{\mathring{\epsilon}_r}(z)\,dx\,dz\notag\\
	&-a^3\sum_{j=1\atop j\neq m}^M(\Pi_k-\Pi_0)(z_m,z_j)[C^{\mathrm{T}}_{\epsilon_r}]\frac{1}{|S_j|}\int_{S_j}U^{\mathring{\epsilon}_r}(y)\,dy.
	\end{align}
	{\bl It is clear that 
	\begin{equation}\notag
		A_{1,m}=A_{1,m}^0+A_{1,m}^R.
	\end{equation}  }
	Once more, we use the splitting:
	\begin{align}\label{A10}
	&A_{1,m}^0=\frac{1}{|S_m|}\sum_{j=1\atop j\neq m}^M\int_{\Omega_j}\int_{S_m}\Pi_0(x,z)c_r^{-3}[C^{\mathrm{T}}_{\epsilon_r}](U^{\mathring{\epsilon}_r}(z)-\frac{1}{|S_j|}\int_{S_j}U^{\mathring{\epsilon}_r}(y)\,dy)\,dx\,dz\notag\\
	&+\frac{1}{|S_m|}\sum_{j=1\atop j\neq m}^M\int_{\Omega_j}\int_{S_m}\Pi_0(x,z)c_r^{-3}[C^{\mathrm{T}}_{\epsilon_r}]\frac{1}{|S_j|}\int_{S_j}U^{\mathring{\epsilon}_r}(y)\,dy\,dx\,dz-a^3\sum_{j=1\atop j\neq m}^M\Pi_0(z_m,z_j)[C^{\mathrm{T}}_{\epsilon_r}]\frac{1}{|S_j|}\int_{S_j}U^{\mathring{\epsilon}_r}(y)\,dy,
	\end{align}
	{\bl where for the simplification of representation, we denote
	\begin{equation}\notag
		I_{A,m}:=\frac{1}{|S_m|}\sum_{j=1\atop j\neq m}^M\int_{\Omega_j}\int_{S_m}\Pi_0(x,z)c_r^{-3}[C^{\mathrm{T}}_{\epsilon_r}](U^{\mathring{\epsilon}_r}(z)-\frac{1}{|S_j|}\int_{S_j}U^{\mathring{\epsilon}_r}(y)\,dy)\,dx\,dz,
	\end{equation}
	and
	\begin{align}\notag
		J_{A,m}&:=\frac{1}{|S_m|}\sum_{j=1\atop j\neq m}^M\int_{\Omega_j}\int_{S_m}\Pi_0(x,z)c_r^{-3}[C^{\mathrm{T}}_{\epsilon_r}]\frac{1}{|S_j|}\int_{S_j}U^{\mathring{\epsilon}_r}(y)\,dy\,dx\,dz\notag\\
		&-a^3\sum_{j=1\atop j\neq m}^M\Pi_0(z_m,z_j)[C^{\mathrm{T}}_{\epsilon_r}]\frac{1}{|S_j|}\int_{S_j}U^{\mathring{\epsilon}_r}(y)\,dy.\notag
	\end{align}  }
	Define
	\begin{equation}\notag
	F_j(z):=[C^{\mathrm{T}}_{\epsilon_r}](U^{\mathring{\epsilon}_r}(z)-\frac{1}{|S_j|}\int_{S_j}U^{\mathring{\epsilon}_r}(y)\,dy).
	\end{equation}
	Then there holds
	\begin{equation}\notag
	I_{A,m}=\frac{1}{|S_m|}\sum_{j=1\atop j\neq m}^M\int_{\Omega_j}\int_{S_m}\Pi_0(x,z)c_r^{-3}F_j(z)\,dx\,dz,
	\end{equation}
	which indicates that
	\begin{equation}\notag
	|I_{A,m}|\leq\frac{1}{|S_m|}\sum_{j=1\atop j\neq m}^M|S_m|\frac{1}{\delta_{mj}^3}c_r^{-3}\int_{\Omega_j}F_j(z)\,dz\leq c_r^{-3}\sum_{j=1\atop j\neq m}^M\frac{1}{\delta_{mj}^3}\int_{\Omega_j}F_j(z)\,dz
	\end{equation}
	for any $x\in S_m\subset\Omega_m$ and $z\in\Omega_j$, $j\neq m$.
	Since from Theorem \ref{thm-regu}, we know that $U^{\mathring{\epsilon}_r}\in C^{0,\alpha}(\Omega)$, then
	\begin{equation}\label{U-reg}
	|U^{\mathring{\epsilon}_r}(z)-U^{\mathring{\epsilon}_r}(y)|\leq c_\alpha|z-y|^\alpha,\quad\mbox{for }z,y\in\Omega, \ \alpha\in (0,1),
	\end{equation}
	where $c_{\alpha}$ is the positive H\"{o}lder constant. By substituting \eqref{U-reg} into $F_j(z)$, we obtain that
	\begin{align}\label{Fj}
	\int_{\Omega_j}|F_j(z)|\,dz&=\int_{\Omega_j}\big|[C^{\mathrm{T}}_{\epsilon_r}](U^{\mathring{\epsilon}_r}(z)-\frac{1}{|S_j|}\int_{S_j}U^{\mathring{\epsilon}_r}(y)\,dy)\big|\,dz\notag\\
	&\leq c_0^{\epsilon_r}\int_{\Omega_j}\frac{1}{|S_j|}\int_{S_j}|U^{\mathring{\epsilon}_r}(z)-U^{\mathring{\epsilon}_r}(y)|\,dy\,dz\leq c_0^{\epsilon_r}|\Omega_j|c_\alpha(\sqrt{3}\delta)^{\alpha}\leq 3^{\frac{\alpha}{2}}c_0^{\epsilon_r}c_\alpha c_r^{3+\alpha}a^{3+\alpha}
	\end{align}
	by the fact that $|\Omega_j|=\Oh(\delta^3)=\Oh(c_r^3a^3)$. Combining with \eqref{Fj} and Lemma \ref{conting1}, we derive the estimate 
	\begin{equation}\notag
	|I_{A,m}|\leq3^{\frac{\alpha}{2}}c_0^{\epsilon_r}c_\alpha c_I c_r^{-3}\delta^{\alpha}(1+\ln\delta),
	\end{equation}
	and thus
	\begin{equation}\label{Iam}
	\sum_{m=1}^M|I_{A,m}|^2\leq \delta^{-3}3^{\alpha+1}{c_0^{\epsilon_r}}^2c_\alpha^2c_I^2c_r^{-6}\delta^{2\alpha}(1+(\ln\delta)^2),
	\end{equation}
	where $M=\Oh(\delta^{-3})$ {\bl and $c_I$ is a positive constant}. Since it is known that $\ln\delta\sim\delta^{-\gamma}$ for $\gamma>0$ as $\delta\rightarrow0$, by letting $\gamma\in(0,\alpha)$, \eqref{Iam} can be further written as
	\begin{equation}\label{Iam1}
	\sum_{m=1}^M|I_{A,m}|^2\leq3^{\alpha+1}{c_0^{\epsilon_r}}^2c_\alpha^2c_I^2c_r^{2\alpha-9}a^{2\alpha-3}(1+c_r^{-2\gamma}a^{-2\gamma})
	\end{equation}
	for $0<\gamma<\alpha$.
	
	For $J_{A,m}$, since $\Pi_0(x,z)$ is harmonic, by the Mean Value theorem we have 
	\begin{equation}\notag
	\frac{1}{|S_m|}\int_{S_m}\Pi_0(x,z)\,dx=\Pi_0(z_m,z).
	\end{equation}
	Thus we have 
	\begin{equation}\notag
	J_{A,m}=\sum_{j=1\atop j\neq m}^M\int_{\Omega_j}(\Pi_0(z_m,z)-\Pi_0(z_m,z_j))c_r^{-3}[C^{\mathrm{T}}_{\epsilon_r}]\frac{1}{|S_j|}\int_{S_j}U^{\mathring{\epsilon}_r}(y)\,dy\,dz,
	\end{equation}
	where 
	\begin{equation}\notag
	|\Pi_0(z_m,z)-\Pi_0(z_m,z_j)|=|\nabla\nabla\Phi_0(z_m,z)-\nabla\nabla\Phi_0(z_m,z_j)|=\Oh(\frac{1}{\delta_{mj}^4}\delta).
	\end{equation}
	Then, based on Lemma \ref{conting1}, we can derive that
	\begin{equation}\notag
	|J_{A,m}|\leq \sum_{j=1\atop j\neq m}^M |\Omega_j|\Oh(\frac{1}{\delta_{mj}^4})\delta c_r^{-3}c_0^{\epsilon_r}\lVert U^{\mathring{\epsilon}_r}\rVert_{L^{\infty}(\Omega)}\leq \delta^{4}c_r^{-3}c_0^{\epsilon_r}\lVert U^{\mathring{\epsilon}_r}\rVert_{L^\infty(\Omega)}\Oh(\sum_{j=1\atop j\neq m}^M\frac{1}{\delta_{mj}^4})\leq c_Jc_0^{\epsilon_r}\lVert U^{\mathring{\epsilon}_r}\rVert_{L^{\infty}(\Omega)}c_r^{-3}
	\end{equation}
	for the positive constant $c_J$. Therefore, it is clear that
	\begin{equation}\label{JA}
	\sum_{m=1}^M|J_{A,m}|^2\leq \delta^{-3}c_J^2 {c_0^{\epsilon_r}}^2\lVert U^{\mathring{\epsilon}_r}\rVert^2_{L^\infty(\Omega)} c_r^{-6}\leq c_J^2 {c_0^{\epsilon_r}}^2\lVert U^{\mathring{\epsilon}_r}\rVert^2_{L^\infty(\Omega)}c_r^{-9}a^{-3}.
	\end{equation}
	Combining with \eqref{Iam1} and \eqref{JA}, we deduce that
	\begin{align}\label{A1m0}
	\sum_{m=1}^M|A^0_{1,m}|^2&\leq 2(\sum_{m=1}^M|I_{A,m}|^2+\sum_{m=1}^M|J_{A,m}|^2)\notag\\
	&\leq3^{\alpha+2}{c_0^{\epsilon_r}}^2c_\alpha^2c_I^2c_r^{2\alpha-9}a^{2\alpha-3}(1+c_r^{-2\gamma}a^{-2\gamma})+2c_J^2{c_0^{\epsilon_r}}^2\lVert U^{\mathring{\epsilon}_r}\rVert^2_{L^\infty(\Omega)}c_r^{-9}a^{-3}
	\end{align}
	for $0<\gamma<\alpha<1$.
	
	Now, let us considet $A^R_{1,m}$. In \eqref{A1mr}, since 
	\begin{align}\notag
	(\Pi_k-\Pi_0)(x,z)&=k^2I\Phi_k(x,z)+\nabla\nabla\Phi_k(x,z)-\nabla\nabla\Phi_0(x,z)\notag\\
	&=k^2I\Phi_k(x,z)+\left(\frac{-k^2}{4\pi}\int_{0}^1\frac{e^{ikt|x-z|}}{|x-z|}[I+(ikt-\frac{1}{|x-z|})\frac{(x-z)(x-z)^{\mathrm{T}}}{|x-z|}]\,dt\right)\notag
	\end{align}
	for any $x\in S_m\subset\Omega_m$ and $z\in\Omega_j$, $j\neq m$, then
	\begin{equation}\notag
	|(\Pi_k-\Pi_0)(x,z)|\leq\frac{k^2}{\delta_{mj}}+k^2(\frac{2}{\delta_{mj}}+k)=\frac{3k^2}{\delta_{mj}}+k^3.
	\end{equation}
	Hence we can know that there holds
	\begin{equation}\label{Ar1m}
	|A^R_{1,m}|\leq\frac{1}{|S_m|}\sum_{j=1\atop j\neq m}^M|S_m|\left(\frac{3k^2}{\delta_{mj}}+k^3\right)c_r^{-3}c_0^{\epsilon_r}\int_{\Omega_j}\big|U^{\mathring{\epsilon}_r}(z)-\frac{1}{|S_j|}\int_{S_j}U^{\mathring{\epsilon}_r}(y)\,dy\big|\,dz.
	\end{equation}
	By using \eqref{U-reg} and Lemma \ref{conting1}, \eqref{Ar1m} can be further written as
	\begin{equation}\notag
	|A^R_{1,m}|\leq c_r^{-3}c_0^{\epsilon_r}3^{\frac{\alpha}{2}}c_\alpha\delta^{\alpha+3}\sum_{j=1\atop j\neq m}^M\left(\frac{3k^2}{\delta_{mj}}+k^3\right)\leq 3^{\frac{\alpha}{2}}c_\alpha c_0^{\epsilon_r}c_r^{-3}\delta^{\alpha+3}(3k^2c_R\delta^{-3}+k^3\delta^{-3})
	\end{equation}
	for the positive constant $c_R$.
	Thus we can know that 
	\begin{equation}\label{Ar1m-1}
	\sum_{m=1}^M|A^R_{1,m}|^2\leq 3^{\alpha+3}c_\alpha^2{c_0^{\epsilon_r}}^2c_r^{-6}k^4(c_R^2+k^2)\delta^{2\alpha-3}\leq 3^{\alpha+3}c_\alpha^2{c_0^{\epsilon_r}}^2k^4(c_R^2+k^2)c_r^{2\alpha-9}a^{2\alpha-3}.
	\end{equation}
	Combining with \eqref{A1m0} and \eqref{Ar1m-1}, it yields that
	\begin{align}\label{A1m}
	\sum_{m=1}^M|A_{1,m}|^2&\leq2\left(\sum_{m=1}^M|A_{1,m}^0|^2+\sum_{m=1}^M|A_{1,m}^R|^2\right)\notag\\
	&\leq 3^{\alpha+3}{c_0^{\epsilon_r}}^2c_r^{-9}a^{-3}\left( c_\alpha^2c_r^{2\alpha}c_I^2a^{2\alpha}((1+c_r^{-2\gamma}a^{-2\gamma})+k^4(c_R^2+k^2))+c_J^2\lVert U^{\mathring{\epsilon}_r}\rVert^2_{L^\infty(\Omega)}\right)\notag\\
	&=\Oh({c_0^{\epsilon_r}}^2c_r^{-9}a^{-3}),
	\end{align}
	for $0<\gamma<\alpha<1$.
	
	Recall \eqref{A}. Denote
	\begin{align}\notag
	A_{2,m}:= \frac{1}{|S_m|}\sum_{j=1\atop j\neq m}^M\int_{\Omega_j}\int_{S_m}ik\nabla\Phi_k(x,z)\times(c_r^{-3}[C^{\mathrm{T}}_{\mu_r}]V^{\mathring{\mu}_r}(z))\,dx\,dz\notag\\
	-a^3\sum_{j=1\atop j\neq m}^M ik\nabla\Phi_k(z_m,z_j)\times([C^{\mathrm{T}}_{\mu_r}]\frac{1}{|S_j|}\int_{S_j}V^{\mathring{\mu}_r}(y)\,dy).\notag
	\end{align}
	Since for any $x\in S_m\subset\Omega_m$ and $z\in\Omega_j$, $j\neq m$, there holds
	\begin{equation}\notag
	|\nabla\Phi_k(x,z)|\leq\frac{1}{4\pi}\frac{1}{\delta_{mj}}\left(\frac{1}{\delta_{mj}}+k\right),
	\end{equation}
	then 
	\begin{equation}\notag
	|A_{2,m}|\leq\frac{k}{|S_m|}\sum_{j=1\atop j\neq m}^M|S_m|\frac{1}{\delta_{mj}}\left(\frac{1}{\delta_{mj}}+k\right)c_r^{-3}c_0^{\mu_r}\int_{\Omega_j}\left( V^{\mathring{\mu}_r}(z)-\frac{1}{|S_j|}\int_{S_j}V^{\mathring{\mu}_r}(y)\,dy\right)\,dz.
	\end{equation}
	Since from Theorem \ref{thm-regu}, we have $V^{\mathring{\mu}_r}\in C^{0,\alpha}(\Omega)$, which indicates that
	\begin{equation}\label{V-reg}
	|V^{\mathring{\mu}_r}(z)-V^{\mathring{\mu}_r}(y)|\leq c_\alpha|z-y|^{\alpha}\quad\mbox{for }z,y\in \Omega,\ \alpha\in (0,1),
	\end{equation}
	where $c_\alpha$ is the same positive constant appearing in \eqref{U-reg}. Thus we know that
	\begin{align}\notag
	|A_{2,m}|&\leq k \sum_{j=1\atop j\neq m}^M\frac{1}{\delta_{mj}}\left(\frac{1}{\delta_{mj}}+k\right)c_r^{-3}c_0^{\mu_r}|\Omega_j|c_\alpha3^{\frac{\alpha}{2}}\delta^\alpha.
	\end{align}
	By Lemma \ref{conting1}, we can further obtain
	\begin{equation}\notag
	|A_{2,m}|\leq 3^{\frac{\alpha}{2}}kc_0^{\mu_r}c_\alpha c_r^{-3}\delta^{3+\alpha}(c_{A_2,1}+kc_{A_2,2})\delta^{-3}
	\leq 3^{\frac{\alpha}{2}}kc_0^{\mu_r}c_\alpha c_r^{-3}(c_{A_2,1}+kc_{A_2,2})\delta^{\alpha},
	\end{equation}
	where $c_{A_2,1}$ and $c_{A_2,2}$ are two positive constants. Therefore, we can derive that
	\begin{align}\label{A2m}
	\sum_{m=1}^M|A_{2,m}|^2&\leq \delta^{-3}3^{\alpha+1}k^2{c_0^{\mu_r}}^2c_\alpha^2c_r^{-6}(c_{A_2,1}+kc_{A_2,2})^2\delta^{2\alpha}\leq 3^{\alpha+1}k^2{c_0^{\mu_r}}^2c_\alpha^2c_r^{-6}
	(c_{A_2,1}+kc_{A_2,2})^2\delta^{2\alpha-3} \notag\\
	&\leq 3^{\alpha+1}k^2{c_0^{\mu_r}}^2c_\alpha^2(c_{A_2,1}+kc_{A_2,2})^2 c_r^{2\alpha-9}a^{2\alpha-3}=\Oh(k^2{c_0^{\mu_r}}^2c_\alpha^2c_r^{2\alpha-9}a^{2\alpha-3}).
	\end{align}
	Combining with \eqref{A1m} and \eqref{A2m}, we derive the estimate
	\begin{equation}\label{A-final}
	\sum_{m=1}^M|A|^2\leq 2\left(\sum_{m=1}^M|A_{1,m}|^2+\sum_{m=1}^M|A_{2,m}|^2\right)\leq\Oh({c_0^{\epsilon_r}}^2c_r^{-9}a^{-3}+k^2{c_0^{\mu_r}}^2c_\alpha^2c_r^{2\alpha-9}a^{2\alpha-3})=\Oh({c_0^{\epsilon_r}}^2c_r^{-9}a^{-3})
	\end{equation}
	for sufficiently small $a$ and $0<\alpha<1$.
	
	(\romannumeral2)\emph{The estimate for terms associated with $B$}.
	Next, we evaluate the term $B$ given by \eqref{B}. Denote
	\begin{align}\label{B2m-def}
	B_{1,m}:&=\frac{1}{|S_m|}\int_{\Omega_m}\int_{S_m}\Pi_k(x,z)c_r^{-3}[C^{\mathrm{T}}_{\epsilon_r}]U^{\mathring{\epsilon}_r}(z)\,dx\,dz,\notag\\
	B_{2,m}:&=\frac{1}{|S_m|}\int_{\Omega_m}\int_{S_m}ik\nabla\Phi(x,z)
	\times(c_r^{-3}[C^{\mathrm{T}}_{\mu_r}]V^{\mathring{\mu}_r}(z))\,dx\,dz,
	\end{align}
	{\bl and further
	\begin{align}\label{B-split}
	B_{1,m}^0&:=\frac{1}{|S_m|}\int_{\Omega_m}\int_{S_m}\Pi_0(x,z)c_r^{-3}[C^{\mathrm{T}}_{\epsilon_r}]U^{\mathring{\epsilon}_r}(z)\,dx\,dz,\notag\\
	B_{1,m}^R&:=\frac{1}{|S_m|}\int_{\Omega_m}\int_{S_m}(\Pi_k-\Pi_0)(x,z)c_r^{-3}[C^{\mathrm{T}}_{\epsilon_r}]U^{\mathring{\epsilon}_r}(z)\,dx\,dz.
	\end{align}
	It is clear that $B=B_{1,m}+B_{2,m}$ and $B_{1,m}=B_{1,m}^0+B_{1,m}^R$.  }
	
	We first investigate $B_{1,m}^R$. Since for any $x\in S_m\subset\Omega_m$ and $z\in \Omega_m$, there holds
	\begin{equation}\notag
	|(\Pi_k-\Pi_0)(x,z)|\leq |k^2I\Phi_k(x,z)+\nabla\nabla\Phi_k(x,z)-\nabla\nabla\Phi_0(x,z)|\leq\frac{6k^2}{4\pi}\frac{1}{|x-z|}.
	\end{equation}
	It is easy to get 
	\begin{align}\label{B1mr-1}
	|B_{1,m}^R|&\leq\frac{k^2}{|S_m|}c_r^{-3}c_0^{\epsilon_r}\lVert U^{\mathring{\epsilon}_r}\rVert_{L^\infty(\Omega)}\int_{\Omega_m}\int_{S_m}\frac{1}{|x-z|}\,dx\,dz\notag\\
	&\leq \frac{k^2}{|S_m|}c_r^{-3}c_0^{\epsilon_r}\lVert U^{\mathring{\epsilon}_r}\rVert_{L^\infty(\Omega)}|\Omega_m|\int_{B(z,\sqrt{3}\delta)}\frac{1}{|x-z|}\,dx
	\end{align}
	by the fact that $S_m\subset\Omega_m\subset B(z,\sqrt{3}\delta)$ for any $z\in\Omega_m$. Thus by direct calculations under spherical coordinate system, we can obtain that
	\begin{equation}\notag
	|B_{1,m}^R|\leq\frac{k^2}{\frac{\pi}{6}\delta^3}c_r^{-3}c_0^{\epsilon_r}\lVert U^{\mathring{\epsilon}_r}\rVert_{L^\infty(\Omega)}\delta^3\int_{0}^{2\pi}\int_{0}^\pi\int_{0}^{\sqrt{3}\delta}\frac{1}{r}r^2\sin\theta\,dr\,d\theta\,d\phi\leq 36k^2c_0^{\epsilon_r}\lVert U^{\mathring{\epsilon}_r}\rVert_{L^\infty(\Omega)}c_r^{-1}a^2,
	\end{equation}
	and thus
	\begin{equation}\label{B1mr}
	\sum_{m=1}^M|B_{1,m}^R|^2\leq 36^2k^4{c_0^{\epsilon_r}}^2\lVert U^{\mathring{\epsilon}_r}\rVert_{L^\infty(\Omega)}^2c_r^{-5}a=\Oh(k^4{c_0^{\epsilon_r}}^2\lVert U^{\mathring{\epsilon}_r}\rVert_{L^\infty(\Omega)}^2c_r^{-5}a).
	\end{equation}
	
	To give the estimate for $B_{2,m}$, we rewrite $
	\nabla\Phi_k(x,z)=\Phi_k(x,z)\left(ik-\frac{1}{|x-z|}\right)\frac{(x-z)}{|x-z|},
	$
	which indicates that
	\begin{equation}\notag
	|\nabla\Phi_k(x,z)|\leq\frac{1}{4\pi}\frac{1}{|x-z|}\left(\frac{1}{|x-z|}+k\right)\quad\mbox{for }\ x\in S_m, z\in \Omega_m.
	\end{equation}
	Thus similar to \eqref{B1mr-1}, we have 
	\begin{align}\notag
	|B_{2,m}|&\leq\frac{k}{|S_m|}c_r^{-3}c_0^{\mu_r}\lVert V^{\mathring{\mu}_r}\rVert_{L^\infty(\Omega)}\int_{\Omega_m}\int_{S_m}\frac{1}{|x-z|}\left(\frac{1}{|x-z|}+k\right)\,dx\,dz\notag\\
	&\leq \frac{k}{|S_m|}c_r^{-3}c_0^{\mu_r}\lVert V^{\mathring{\mu}_r}\rVert_{L^\infty(\Omega)}|\Omega_m|\int_{B(z,\sqrt{3}\delta)}\left(\frac{1}{|x-z|^2}+\frac{k}{|x-z|}\right)\,dx\notag\\
	&\leq 12kc_0^{\mu_r}c_r^{-3}\lVert V^{\mathring{\mu}_r}\rVert_{L^\infty(\Omega)}(2\sqrt{3}+3k\delta)\delta,\notag
	\end{align}
	and therefore
	\begin{align}\label{B2m}
	\sum_{m=1}^M|B_{2,m}|^2&\leq 288k^2{c_0^{\mu_r}}^2c_r^{-6}\lVert V^{\mathring{\mu}_r}\rVert_{L^\infty(\Omega)}^2(12+9k^2\delta^2)\delta^{-1}\notag\\
	&\leq 288k^2{c_0^{\mu_r}}^2\lVert V^{\mathring{\mu}_r}\rVert_{L^\infty(\Omega)}^2(12+9k^2c_r^2a^2)c_r^{-7}a^{-1}\notag\\
	&=\Oh(k^2{c_0^{\mu_r}}^2\lVert V^{\mathring{\mu}_r}\rVert^2_{L^\infty(\Omega)}(12+9k^2c_r^2a^2)c_r^{-7}a^{-1}).
	\end{align}
	
	Now we focus on the term $B_{1,m}^0$ given by \eqref{B-split}. It is obvious that $B_{1,m}^0$ can be written as 
	\begin{align}\label{B1m0}
	B_{1,m}^0&=\frac{1}{|S_m|}\int_{\Omega_m}\int_{S_m}\Pi_0(x,z)c_r^{-3}[C^{\mathrm{T}}_{\epsilon_r}]\left(U^{\mathring{\epsilon}_r}(z)-\frac{1}{|S_m|}\int_{S_m}U^{\mathring{\epsilon}_r}(y)\,dy\right)\,dx\,dz\notag\\
	&+\frac{1}{|S_m|}\int_{\Omega_m}\int_{S_m}\Pi_0(x,z)c_r^{-3}[C^{\mathrm{T}}_{\epsilon_r}]\frac{1}{|S_m|}\int_{S_m}U^{\mathring{\epsilon}_r}(y)\,dy\,dx\,dz.
	\end{align}
	{\bl Denote
	\begin{equation}\notag
		I_B:=\frac{1}{|S_m|}\int_{\Omega_m}\int_{S_m}\Pi_0(x,z)c_r^{-3}[C^{\mathrm{T}}_{\epsilon_r}]\left(U^{\mathring{\epsilon}_r}(z)-\frac{1}{|S_m|}\int_{S_m}U^{\mathring{\epsilon}_r}(y)\,dy\right)\,dx\,dz,
	\end{equation}
	and 
	\begin{equation}\notag
		K^{\epsilon_r}:=\frac{1}{|S_m|}\int_{\Omega_m}\int_{S_m}\Pi_0(x,z)c_r^{-3}[C^{\mathrm{T}}_{\epsilon_r}]\frac{1}{|S_m|}\int_{S_m}U^{\mathring{\epsilon}_r}(y)\,dy\,dx\,dz.
	\end{equation}  }
	Using \eqref{U-reg}, we can directly know that
	\begin{equation}\label{IB1}
	|I_B|\leq\frac{1}{|S_m|}c_r^{-3}c_0^{\epsilon_r}c_\alpha3^{\frac{\alpha}{2}}\delta^{\alpha}\int_{\Omega_m}\int_{S_m}\Pi_0(x,z)\,dx\,dz.
	\end{equation}
	By the H\"{o}lder and Cald\'{e}ron-Zygmund inequalities, see \cite[Theorem 9.9]{GT}, we obtain from \eqref{IB1} that {\bl
	\begin{align}\notag
	|I_B|&\leq \frac{1}{|S_m|}3^{\frac{\alpha}{2}}c_0^{\epsilon_r}c_\alpha c_r^{-3} \delta^{\alpha}\lVert \int_{\Omega_m}\Pi_0(x,z)\,dz\rVert_{L^p(S_m)}\cdot|S_m|^{\frac{1}{q}}\notag\\
	&\leq 3^{\frac{\alpha}{2}}|S_m|^{-\frac{1}{p}}c_0^{\epsilon_r}c_\alpha c_r^{-3}\delta^\alpha\lVert \int_{\mathbb{R}^3}\Pi_0(x,z)\chi_{\Omega_m}(z)\,dz\rVert_{L^p(\mathbb{R}^3)}\notag\\
	&\leq 3^{\frac{\alpha}{2}}|S_m|^{-\frac{1}{p}}c_0^{\epsilon_r}c_\alpha c_r^{-3}\delta^\alpha\lVert\chi_{\Omega_m}\rVert_{L^p(\mathbb{R}^3)}\leq 3^{\frac{\alpha}{2}}\left(\frac{6}{\pi}\right)^{\frac{1}{p}}c_0^{\epsilon_r}c_\alpha c_r^{-3}\delta^\alpha\notag
	\end{align} }
	for $p,q>1$ such that $1/p+1/q=1$ and $\alpha\in(0,1)$. Thus we can derive that
	\begin{equation}\label{IB}
	\sum_{m=1}^M|I_B|^2\leq 3^\alpha \left(\frac{6}{\pi}\right)^{\frac{2}{p}} {c_0^{\epsilon_r}}^2c_\alpha^2c_r^{-6}\delta^{2\alpha-3}=\Oh({c_0^{\epsilon_r}}^2c_\alpha^2c_r^{2\alpha-9}a^{2\alpha-3}).
	\end{equation}
	
	(\romannumeral3)\emph{The estimate for $C$}. We rewrite \eqref{C} as 
	\begin{equation}\notag
	C:=C_{1,m}^0+C_{1,m}^R+C_{2,m},
	\end{equation} 
	where
	\begin{align}\notag
	C_{1,m}^0&:=\frac{1}{|S_m|}\int_{L}\int_{S_m}\Pi_0(x,z)c_r^{-3}[C^{\mathrm{T}}_{\epsilon_r}]U^{\mathring{\epsilon}_r}(z)\,dx\,dz,\notag\\
	C_{1,m}^R&:=\frac{1}{|S_m|}\int_{L}\int_{S_m}(\Pi_k-\Pi_0)(x,z)c_r^{-3}[C^{\mathrm{T}}_{\epsilon_r}]U^{\mathring{\epsilon}_r}(z)\,dx\,dz\notag\\
	C_{2,m}&:=\frac{1}{|S_m|}\int_{L}\int_{S_m}ik\nabla\Phi_k(x,z)\times(c_r^{-3}[C^{\mathrm{T}}_{\mu_r}]V^{\mathring{\mu}_r}(z))\,dx\,dz.\notag
	\end{align}
	By the Mean Value Theorem for the harmonic function $\Pi_0(x,z)$, we have
	\begin{equation}\notag
	C_{1,m}^0=\int_{L}\Pi_0(z_m,z)c_r^{-3}[C^{\mathrm{T}}_{\epsilon_r}]U^{\mathring{\epsilon}_r}(z)\,dz
	\end{equation}
	and thus
	\begin{equation}\label{C1m0}
	\sum_{m=1}^M|C_{1,m}^0|^2\leq \sum_{m=1}^Mc_r^{-6}{c_0^{\epsilon_r}}^2\lVert U^{\mathring{\epsilon}_r}\rVert_{L^{\infty}(\Omega)}^2\left( \int_{L}\frac{1}{|z_m-z|^3}\,dz\right)^2.
	\end{equation}
	In order to evaluate the integral over $L$ in \eqref{C1m0}, we give the following way of counting the particles. 
	
	Let an arbitrary position $z_m$ be fixed. From $z_m$, we split the space into equidistant cubes $\Sigma_\ell$ of distance $\delta$. Let $F_\ell$ be the surfaces of $\Sigma_\ell$ for a fixed direction such that each $F_\ell$ supports some of the $(z_\ell)_{\ell=1, \ell\neq m}^M$. It is know from \cite[Lemma A.1]{AM} that the number of such cubes $\Sigma_\ell$ is at most of $\Oh(M^{\frac{1}{3}})$. To estimate the integral over $L:=\Omega\backslash\cup_{m=1}^M\Omega_m$, since $a$ is small enough, we can assume that any small part of $\partial\Omega$ is flat, recalling that $\partial\Omega\in C^{1,\alpha}$. We divide the region in $L$ close to the considering flat part into concentric layers such that there are at most $(2n+1)^2$ small cubes $\Omega_n$ up to the $n$-th layer intersecting with the surface, for $n=1,\cdots, M^{\frac{1}{3}}$. It is easy to see that the number of the particles located in the $n$-th layer is at most of order $[(2n+1)^2-(2n-1)^2]$ and the distance from any small cube $\Omega_n$ containing the particle to $z_m$ is larger than $(n+\ell)\delta$. We refer Fig \ref{fig:counting} for the schematic illustration.
	
	\begin{figure}[htbp]
		\centering
		\includegraphics[width=0.6\linewidth]{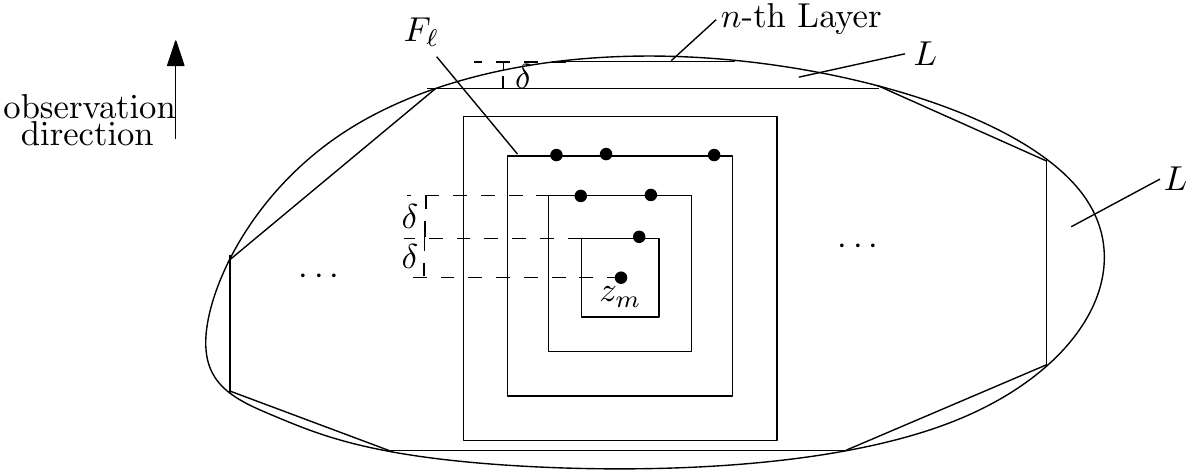}
		\caption{A schematic illustration on the counting criterion.}
		\label{fig:counting}
	\end{figure}
	
	Following this counting criterion, we can further see from \eqref{C1m0} that
	\begin{align}\label{C1m0-1}
	\sum_{m=1}^M|C^0_{1,m}|^2&\leq c_r^{-6}{c_0^{\epsilon_r}}^2\lVert U^{\mathring{\epsilon}_r}\rVert_{L^{\infty}(\Omega)}^2\sum_{\ell=1}^{M^{\frac{1}{3}}}\sum_{z_\ell\in F_\ell}\left(\sum_{n=1}^{M^{\frac{1}{3}}}|\Omega_n|[(2n+1)^2-(2n-1)^2]\frac{1}{(n+\ell)^3\delta^3}\right)^2\notag\\
	&\leq 64c_r^{-6}{c_0^{\epsilon_r}}^2\lVert U^{\mathring{\epsilon}_r}\rVert_{L^{\infty}(\Omega)}^2\sum_{\ell=1}^{M^\frac{1}{3}}\left(\sum_{n=1}^{M^{\frac{1}{3}}} \frac{n\ell}{(n+\ell)^3}\right)^2,
	\end{align}
	where the number of the particles on each $F_\ell$ is of order $\ell^2$; see \cite[Lemma A.1]{AM} for more detailed discussions.
	In \eqref{C1m0-1}, it is easy to see that
	\begin{equation}\notag
	\sum_{n=1}^{M^{\frac{1}{3}}}\frac{n\ell}{(n+\ell)^3}\leq \ell\int_{1}^{M^{\frac{1}{3}}}\frac{n}{(n+\ell)^3}\,dn=\ell\cdot\left(-\frac{\ell+2n}{2(\ell+n)^2}\right)\big|^{M^{\frac{1}{3}}}_{n=1},
	\end{equation}
	and thus
	\begin{equation}\label{nl}
	\sum_{\ell=1}^{M^\frac{1}{3}}\left(\sum_{n=1}^{M^{\frac{1}{3}}} \frac{n\ell}{(n+\ell)^3}\right)^2\leq 3\sum_{\ell=1}^{M^{\frac{1}{3}}}\left(\frac{\ell^2(\ell+2)^2}{(\ell+1)^4}+\frac{\ell^2}{4(\ell+M^{\frac{1}{3}})^2}+\frac{M^{\frac{2}{3}}\ell^2}{4(\ell+M^{\frac{1}{3}})^4}\right).
	\end{equation}
	Recall that $M=\Oh(\delta^{-3})$. Now, we evaluate the three terms on the right hand side of \eqref{nl}, respectively.
	It is clear that for the first term, there holds
	\begin{equation}\label{term1}
	\sum_{\ell=1}^{[\delta^{-1}]}\frac{\ell^2(\ell+2)^2}{(\ell+1)^4}\leq c_{\mathrm{count}}^{(1)}\delta^{-1},
	\end{equation}
	where $c_{\mathrm{count}}$ is a positive constant. For the second term, we have 
	\begin{align}\label{term2}
	\sum_{\ell=1}^{[\delta^{-1}]}\frac{\ell^2}{4(\ell+\delta^{-1})^2}&\leq  \int_{1}^{[\delta^{-1}]}\frac{\ell^2}{(\ell+\delta^{-1})^2}\,d\ell\notag\\
	&\leq \int_{1}^{[\delta^{-1}]}\frac{(\ell+\delta^{-1})^2}{(\ell+\delta^{-1})^2}\,d\ell-\int_{1}^{[\delta^{-1}]}\frac{2\delta^{-1}\ell}{(\ell+\delta^{-1})^2}\,d\ell-\int_{1}^{[\delta^{-1}]}\frac{\delta^{-2}}{(\ell+\delta^{-1})^2}\,d\ell\notag\\
	&\leq
	(\delta^{-1}-1)-2\delta^{-1}\ln\frac{2\delta^{-1}}{1+\delta^{-1}}+\frac{\delta^{-2}}{1+\delta^{-1}}-\frac{\delta^{-1}}{2}\notag\\
	&=\frac{3\delta^{-2}-\delta^{-1}-2}{2(1+\delta^{-1})}-2\delta^{-1}\ln\frac{2\delta^{-1}}{1+\delta^{-1}},
	\end{align}
	by direct computations. Similarly, for the third term, we get
	\begin{align}\label{term3}
	\sum_{\ell=1}^{[\delta^{-1}]}\frac{\delta^{-2}\ell^2}{4(\ell+\delta^{-1})^4}&\leq \delta^{-2}\int_{1}^{[\delta^-1]}\frac{\ell^2}{(\ell+\delta^{-1})^4}\,d\ell\notag\\
	&\leq \delta^{-2}\int_{1}^{[\delta^{-1}]}\left(\frac{(\ell+\delta^{-1})^2}{(\ell+\delta^{-1})^4}-\frac{2\delta^{-1}\ell}{(\ell+\delta^{-1})^4}-\frac{\delta^{-2}}{(\ell+\delta^{-1})^4}\right)\,d\ell\notag\\
	&\leq \frac{3}{2}\delta^{-1}+\frac{\delta^{-1}}{1+\delta^{3}}-\frac{\delta^{-1}}{1+\delta^2}.
	\end{align}
	Subsituting \eqref{term1}, \eqref{term2} and \eqref{term3} into \eqref{nl}, we obtain
	\begin{equation}\label{nl1}
	\sum_{\ell=1}^{M^\frac{1}{3}}\left(\sum_{n=1}^{M^{\frac{1}{3}}} \frac{n\ell}{(n+\ell)^3}\right)^2\leq \left(3c_{\mathrm{count}}^{(1)}+\frac{9}{2}\right)\delta^{-1}+\frac{2(3\delta^{-2}-\delta^{-1}-2)}{1+\delta^{-1}}-6\delta^{-1}\ln\frac{2\delta^{-1}}{1+\delta^{-1}}+\frac{3\delta(1-\delta)}{(1+\delta^3)(1+\delta^2)}.
	\end{equation}
	By direct calculations, as $a\to0$, we see that
	\begin{equation}\notag
	\frac{2(3\delta^{-2}-\delta^{-1}-2)}{1+\delta^{-1}}-6\delta^{-1}\ln\frac{2\delta^{-1}}{1+\delta^{-1}}+\frac{3\delta(1-\delta)}{(1+\delta^3)(1+\delta^2)}=\Oh(\delta^{-1}):=c_{\mathrm{count}}^{(2)}\delta^{-1}
	\end{equation}
	for $c_{\mathrm{count}}^{(2)}$ being a positive constant. Denote $c_{\mathrm{aver}}:=\frac{1}{3}(3c_{\mathrm{count}}^{(1)}+\frac{9}{2}+c_{\mathrm{count}}^{(2)})$, then from \eqref{nl1}, we can deduce that
	\begin{equation}\label{C1m0-final}
	\sum_{m=1}^M|C^0_{1,m}|^2\leq 64c_r^{-6}{c_0^{\epsilon_r}}^2\lVert U^{\mathring{\epsilon}_r}\rVert_{L^{\infty}(\Omega)}^2 c_{\mathrm{aver}}\delta^{-1}=64{c_0^{\epsilon_r}}^2\lVert U^{\mathring{\epsilon}_r}\rVert_{L^{\infty}(\Omega)}^2 c_{\mathrm{aver}}c_r^{-7}a^{-1}.
	\end{equation}
	
To estimate $C_{1,m}^R$, based on the inequality 
	\begin{equation}\notag
	|(\Pi_k-\Pi_0)(x,z)|\leq \frac{3k^2}{|x-z|}+k^3\quad\mbox{for }\ x\in S_m, z\in L,
	\end{equation}
	we have 
	\begin{align}\label{cr1m0}
	|C^R_{1,m}|&\leq \frac{1}{|S_m|}c_r^{-3}c_0^{\epsilon_r}\lVert U^{\mathring{\epsilon}_r}\rVert_{L^\infty(\Omega)}\int_{L}\int_{S_m}\left( \frac{3k^2}{|x-z|}+k^3\right)\,dx\,dz.\notag\\
	&\leq \frac{1}{|S_m|}c_r^{-3}c_0^{\epsilon_r}\lVert U^{\mathring{\epsilon}_r}\rVert_{L^\infty(\Omega)}|S_m|\left( \int_{B(x,c_0\delta)}\frac{3k^2}{|x-z|}\,dz+k^3|L|\right),
	\end{align}
	where we use the fact that for any $x\in S_m$, there exists a constant $c_0>0$ large enough such that $L\subset B(x,c_0\delta)$ by the distribution of the small particles. Moreover, recalling Remark \ref{layer}, we know that $|L|=\Oh(\delta)$.
	Therefore, we can derive from \eqref{cr1m0} by direct calculations that
	\begin{align}
	|C^R_{1,m}|&\leq c_r^{-3}c_0^{\epsilon_r}\lVert U^{\mathring{\epsilon}_r}\rVert_{L^\infty(\Omega)} \left( 3k^2\int_{0}^{2\pi}\int_{0}^{\pi}\int_{0}^{c_0\delta}\frac{1}{r}r^2\sin\theta\,dr\,d\theta\,d\phi+k^3\delta\right)\notag\\
	&\leq k^2 c_r^{-3}c_0^{\epsilon_r}\lVert U^{\mathring{\epsilon}_r}\rVert_{L^\infty(\Omega)}\delta(6\pi c_0^2\delta+k),\notag
	\end{align}
	and thus
	\begin{align}\label{C1mr}
	\sum_{m=1}^M|C^R_{1,m}|^2&\leq 2k^4 c_r^{-6}{c_0^{\epsilon_r}}^2\lVert U^{\mathring{\epsilon}_r}\rVert_{L^\infty(\Omega)}^2(36\pi^2 c_0^4\delta^2+k^2)\delta^{-1}\notag\\
	&\leq 2k^4 {c_0^{\epsilon_r}}^2\lVert U^{\mathring{\epsilon}_r}\rVert_{L^\infty(\Omega)}^2(36\pi^2 c_0^4c_r^2a^2+k^2)c_r^{-7}a^{-1}.
	\end{align}
	
	Now we consider $C_{2,m}$. Since there holds
	\begin{equation}\notag
	|\nabla\Phi_k(x,z)|\leq\frac{1}{|x-z|}\left(\frac{1}{|x-z|}+k\right)\quad\mbox{for }\ x\in S_m, z\in L,
	\end{equation}
	we can know that
	\begin{equation}\notag
	|C_{2,m}|\leq\frac{k}{|S_m|}c_r^{-3}c_0^{\mu_r}\lVert V^{\mathring{\mu}_r}\rVert_{L^{\infty}(\Omega)}\int_{L}\int_{S_m}\frac{1}{|x-z|}\left(\frac{1}{|x-z|}+k\right)\,dx\,dz.
	\end{equation}
	Similar to \eqref{cr1m0}, we can deduce that
	\begin{equation}\notag
	|C_{2,m}|\leq k c_r^{-3}c_0^{\mu_r}\lVert V^{\mathring{\mu}_r}\rVert_{L^{\infty}(\Omega)} 2\pi c_0\delta(2+kc_0\delta),
	\end{equation}
	and 
	\begin{align}\label{C2m}
	\sum_{m=1}^M|C_{2,m}|^2&\leq 8\pi^2c_0^2\delta^{-3}k^2 c_r^{-6}{c_0^{\mu_r}}^2\lVert V^{\mathring{\mu}_r}\rVert_{L^{\infty}(\Omega)}^2\delta^2(4+k^2c_0^2\delta^2)\notag\\
	&\leq 8\pi^2c_0^2k^2 {c_0^{\mu_r}}^2\lVert V^{\mathring{\mu}_r}\rVert_{L^{\infty}(\Omega)}^2(4+k^2c_0^2c_r^2a^2)c_r^{-7}a^{-1}.
	\end{align}
	
	Combining with \eqref{C1m0-final}, \eqref{C1mr} and \eqref{C2m}, we obtain that
	\begin{align}\label{C-final}
	\sum_{m=1}^M|C|^2&\leq3\left(\sum_{m=1}^M|C_{1,m}^0|^2+\sum_{m=1}^M|C^R_{1,m}|^2+\sum_{m=1}^M|C_{2,m}|^2\right)\notag\\
	&=\Oh\left(\right. [{c_0^{\epsilon_r}}^2\lVert U^{\mathring{\epsilon}_r}\rVert_{L^{\infty}(\Omega)}^2 (32c_{\mathrm{aver}}+
	k^4 (36\pi^2 c_0^4c_r^2a^2+k^2))\notag\\
	&+\pi^2c_0^2k^2 {c_0^{\mu_r}}^2\lVert V^{\mathring{\mu}_r}\rVert_{L^{\infty}(\Omega)}^2(4+k^2c_0^2c_r^2a^2)]c_r^{-7}a^{-1}\left.\right)	
	\end{align}
	So far, we have proved the estimates \eqref{group-est} with \eqref{A-final}, \eqref{B1mr}, \eqref{B2m}, \eqref{IB} and \eqref{C-final}.

	Recalling Definition \ref{def2} for the bounded linear operator $K^{\epsilon_r}$, it is clear that in \eqref{B1m0}
	\begin{equation}\label{key}
	K^{\epsilon_r}\left(\frac{1}{|S_m|}\int_{S_m}U^{\mathring{\epsilon}_r}(y)\,dy\right)=\frac{1}{|S_m|}\int_{\Omega_m}\int_{S_m}\Pi_0(x,z)c_r^{-3}[C^{\mathrm{T}}_{\epsilon_r}]\left(\frac{1}{|S_m|}\int_{S_m}U^{\mathring{\epsilon}_r}(y)\,dy\right)\,dx\,dz
	\end{equation}
	for the average value of $U^{\mathring{\epsilon}_r}$ over $S_m$. 
	
	By rearranging terms in \eqref{main-eq2}, we deduce that
	\begin{align}\label{U-m1}
	&(I-K^{\epsilon_r})\left(\frac{1}{|S_m|}\int_{S_m}U^{\mathring{\epsilon}_r}(x)\,dx\right)-a^3\sum_{j=1\atop j\neq m}^M(\Pi_k(z_m,z_j)[C^{\mathrm{T}}_{\epsilon_r}]\frac{1}{|S_j|}\int_{S_j}U^{\mathring{\epsilon}_r}(y)\,dy\notag\\
	&+ik\nabla\Phi_k(z_m,z_j)\times([C^{\mathrm{T}}_{\mu_r}]\frac{1}{|S_j|}\int_{S_j}V^{\mathring{\mu}_r}(y)\,dy))\notag\\
	&=E^{in}(z_m)+\Oh(c_ra)+A+I_B+B^R_{1,m}+B_{2,m}+C
	\end{align}
	where $A, I_B, B^R_{1,m}, B_{2,m}$ and $C$ are defined by \eqref{A}, \eqref{B1m0},\eqref{B-split}, \eqref{B2m-def} and \eqref{C}, respectively. Subtracting the discrete linear algebraic system \eqref{dislinear} from \eqref{U-m1}, under condition b) in Definition \ref{def2}, where we know that $I-K^{\epsilon_r}$ and $I-K^{\mu_r}$ are invertible, then
	we get that
	\begin{align}\label{U-m2}
	&\left((I-K^{\epsilon_r})\frac{1}{|S_m|}\int_{S_m}U^{\mathring{\epsilon}_r}(x)\,dx-U_m\right)-a^3\sum_{j=1\atop j\neq m}^M((\Pi_k(z_m,z_j)[C^{\mathrm{T}}_{\epsilon_r}](I-K^{\epsilon_r})^{-1}(I-K^{\epsilon_r})\notag\\
	&\frac{1}{|S_j|}\int_{S_j}U^{\mathring{\epsilon}_r}(y)\,dy- \Pi_k(z_m,z_j)[P^{\epsilon_r}_0]U_j)+ik\nabla\Phi_k(z_m,z_j)\times([C^{\mathrm{T}}_{\mu_r}](I-K^{\mu_r})^{-1}(I-K^{\mu_r})\notag\\
	&\frac{1}{|S_j|}\int_{S_j}V^{\mathring{\mu}_r}(y)\,dy-[P_0^{\mu_r}]V_j))=\Oh(c_ra)+A+I_B+B^R_{1,m}+B_{2,m}+C.
	\end{align}
	Utilizing condition a) in Definition \ref{def2}, we can further obtain from \eqref{U-m2} that
	\begin{align}\notag
	&\left((I-K^{\epsilon_r})\frac{1}{|S_m|}\int_{S_m}U^{\mathring{\epsilon}_r}(x\,dx)-U_m\right)-a^3\sum_{j=1\atop j\neq m}^M(\Pi_k(z_m,z_j)[P_0^{\epsilon_r}]\left(\right. (I-K^{\epsilon_r})\frac{1}{|S_j|}\int_{S_j}U^{\mathring{\epsilon}_r}(y)\,dy\notag\\
	&-U_j\left.\right)+ ik \nabla\Phi_k(z_m,z_j)\times([P_0^{\mu_r}]((I-K^{\mu_r})\frac{1}{|S_j|}\int_{S_j}V^{\mathring{\mu}_r}(y)\,dy-V_j)))\notag\\
	&=\Oh(c_ra)+A+I_B+B^R_{1,m}+B_{2,m}+C,\notag
	\end{align}
	which completes the proof.

\end{document}